\documentclass[a4paper]{article}

\usepackage{graphicx}
\usepackage[colorlinks=true,linkcolor=blue,citecolor=red]{hyperref}
\usepackage{epstopdf}
\usepackage[english]{babel}
\usepackage{subfigure}
\usepackage{color}
\usepackage{amsfonts,amsmath,amssymb,amsthm,mathtools}
\usepackage{a4wide}

\usepackage{cite}
\newtheorem{rem}{Remark}

\newtheorem{prop}{Proposition}
\newtheorem{thm}{Theorem}

\allowdisplaybreaks

\def\be{\begin{equation}}
\def\ee{\end{equation}}
\def\bea{\begin{eqnarray}}
\def\eea{\end{eqnarray}}

\def\RR{\mathbb R}

\def\hf{\hat{\textbf{f}}}
\def\hg{\hat{\textbf{g}}}

\def\hfq{\widehat{\textbf{f\,}^q}}
\def\z{\mathbf{z}}
\def\u{\mathbf{u}}
\def\v{\mathbf{v}}
\def\x{\mathbf{x}}
\def\y{\mathbf{y}}

\begin{document}

\title{Micro-macro stochastic Galerkin methods for nonlinear Fokker-Plank equations with random inputs}

\date{}

\author{Giacomo Dimarco \thanks{Department of Mathematics and Computer Science \& Center for Modeling, Computing and Statistics CMCS, University of Ferrara, Via N. Machiavelli 35, 44121 Ferrara, Italy ({\tt giacomo.dimarco@unife.it}).} \and 
Lorenzo Pareschi \thanks{Department of Mathematics and Computer Science \& Center for Modeling, Computing and Statistics CMCS, University of Ferrara, Via N. Machiavelli 35, 44121 Ferrara, Italy ({\tt lorenzo.pareschi@unife.it}).} \and 
Mattia Zanella\thanks{Department of Mathematics "F. Casorati", University of Pavia, Via A. Ferrata 5, 27100 Pavia, Italy ({\tt mattia.zanella@unipv.it}).}}

\maketitle
\begin{abstract}
Nonlinear Fokker-Planck equations play a major role in modeling large systems of interacting particles with a proved effectiveness in describing real world phenomena ranging from classical fields such as fluids and plasma to social and biological dynamics. Their mathematical formulation has often to face with physical forces having a significant random component or with particles living in a random environment which characterization may be deduced through experimental data and leading consequently to uncertainty-dependent equilibrium states. In this work, to address the problem of effectively solving stochastic Fokker-Planck systems, we will construct a new equilibrium preserving scheme through a micro-macro approach based on stochastic Galerkin methods. The resulting numerical method, contrarily to the direct application of a stochastic Galerkin projection in the parameter space of the unknowns of the underlying Fokker-Planck model, leads to highly accurate description of the uncertainty dependent large time behavior. Several numerical tests in the context of collective behavior for social and life sciences are presented to assess the validity of the present methodology against standard ones. 
\end{abstract}

\medskip


\medskip


\tableofcontents

\section{Introduction}
\label{sec1}
One of the most important partial differential equations (PDEs) used to describe the behavior of systems composed by a large number of interacting agents is certainly the so-called Fokker-Planck equation which owes its name to Adriaan Fokker^^>\cite{Fokker} and Max Planck^^>\cite{Planck} who first derive it at the beginning of the last century to describe the time evolution of the probability density function of the velocity of a particle under the influence of drag forces and random forces. After these pioneering works, this model has subsequently been used in classical statistical physics, rarefied gas dynamics and plasma physics^^>\cite{cercignani1988BOOK,Rosenbluth} since the middle of last century. More recently, especially in the last few decades, there has been a growing trend toward applications to interdisciplinary fields beyond mathematics and physics, ranging from biology and socio-economy^^>\cite{ha2008KRM,carrillo2010MSSET,pareschi2013BOOK} to traffic flow and crowd motion^^>\cite{HP,TZ19,dimarcotosin}. Without wishing to review the vast literature on these topics, we refer the reader to some recent papers and surveys where Fokker-Planck equations are used to model different phenomena, see for example^^>\cite{albi2016CHAPTER,pareschitoscani,during2009PRSA,furioli2017M3AS,ATZ} and the references therein. 

However, all of the above references do not take into account the important fact that uncertainties in the mathematical description of real-world phenomena seem to be unavoidable for applications, especially when dealing with these newly emerging fields. In fact, in almost all situations that differ from classical physics, we have at most statistical information about the modeling parameters, which must be estimated from experiments or derived from observations^^>\cite{albi2015MPE,Ballerini2008,Bongini2017,dimarco2017CHAPTER,TZ_17,ABBDPTZ21}. This is especially true in situations in which differential equations are used for describing self-organization dynamics in biology, economy, traffic dynamics or opinion formation. In fact, in all mentioned examples the underlying physics is not based on first principles, that are usually replaced by empirical social forces of which we have at most statistical information. In addition, in the cases in which experimental results exist, these are typically affected by some bias. As a consequence, the construction of mathematical models in those fields is dictated by the capacity of the Fokker-Planck dynamics to describe qualitatively instead of quantitatively the system's behavior and the formation of emergent social structures and results have to be interpreted in a probabilistic sense. Therefore, the methods of uncertainty quantification (UQ) play a key role in the modeling process and in assessing the validity of the results^^>\cite{dimarcoUQ,dimarcoUQ2,GHI,xiu2010BOOK,xiu2002SISC,Zan20,zhu2017MMS,Hu2016}.    

From a mathematical viewpoint, the equations we will consider in the present work are characterized by linear and nonlinear Fokker--Planck equations, i.e. by equations giving the time evolution of a probability density function in term of a balance between a linear/non linear advection (or aggregation) and a diffusion process. In addition, to this underlying dynamics, we will consider the presence of random inputs taking into account uncertainties in the initial data, in the interaction terms and/or in the boundary conditions. More precisely, in the sequel we consider the evolution of a distribution function $f(\z,\v,t)$, $t\ge 0$, $\v\in V \subseteq\RR^{d_\v}$, $d_v\ge 1$, and $\z\in I_{Z}\subseteq\RR^{d_{z}}$ a random vector with distribution $p(\z)$, representing the density of particles/agents with state $\v$ and at time $t$ whose evolution is given by the following general Fokker-Planck model
\be\begin{split}\label{eq:MF_general}
\dfrac{\partial}{\partial t} f(\z,\v,t)= \mathcal J(f,f)(\z,\v,t)
\end{split}\ee
where $\mathcal J(\cdot,\cdot)$ is a nonlinear operator describing the interaction of particles/agents.
In particular, we consider operators of the form
\be\label{eq:J_meanfield}
\mathcal J(f,f) (\z,\v,t)= \nabla_{{\v}} \cdot \big( \mathcal B[f](\z,\v,t)f(\z,\v,t)+\nabla_{{\v}} (D(\z,\v)f(\z,\v,t)) \big)
\ee
where $\mathcal B[\cdot]$ is a nonlocal term 
\begin{equation}
\label{eq:Bf}
\mathcal B[f](\z,\v,t) =   \int_{V}P(\z,\v,\v_*){(\v-\v_*)}f(\z,\v_*,t)d\v_*,
\end{equation}
whereas $D(\z,\v)\ge 0$ is a function describing the local relevance of diffusion which vanishes on the boundary of $V \subset \RR^{d_v}$ or at infinity if $V=\RR^{d_v}$.  
In \eqref{eq:Bf} the interaction forces are described by the operator $P\ge0$. We stress that, in some cases, under additional assumptions, the equilibrium solutions $f^\infty(\z,\v)$ of the Fokker-Planck operator defined by $\mathcal J(f^\infty,f^\infty)=0$ are explicitly computable. 

It is important to emphasize that, even if we concentrate on Fokker-Planck equations of the type \eqref{eq:MF_general}-\eqref{eq:J_meanfield}, the methods here developed also find natural application to the case of Vlasov-Fokker-Planck equations in the general form
\be\begin{split}\label{eq:MF_general2}
\frac{\partial}{\partial t} f(\z,\x,\v,t)+\v\cdot \nabla_{{\x}} f(\z,\x,\v,t)= \mathcal J(f,f)(\z,\x,\v,t),
\end{split}\ee
where now $\x\in\RR^{d_x}$ denotes the space variable and $\mathcal J(f,f)$ is defined as in \eqref{eq:J_meanfield} with the additional space dependence. We refer in particular to^^>\cite{pareschi2013BOOK,villani2002BOOK} for an introduction to the subject in relation with classical gas dynamics and models for collective behavior. 

From the numerical point of view, in general, the development of methods for approximating equations of type \eqref{eq:MF_general} presents several difficulties due to the high dimensionality and the intrinsic structural properties of the solution. Preservation of these structural properties, which are important in applications, is even more challenging in presence of uncertainties that contribute to increasing the dimensionality of the problem and changing the character of the equations. We refer to^^>\cite{dimarcopareschiACTA14,PZ2018} for a review on numerical schemes for such models in absence of uncertainties, while we refer to^^>\cite{dimarco2017CHAPTER,dimarcoMF,Hu2017,Pareschi} for recent surveys about numerical methods {in which the presence of uncertainties} is considered.

Among the most popular numerical methods for UQ, stochastic Galerkin {(sG)} methods based on generalized polynomial chaos (gPC) expansions gained in recent years an increased interest and play nowadays an important role. In particular, one of the main features of these methods is that they lead to spectral convergence {in a weighted $L^2$ space}  under suitable regularity assumptions on the solution^^>\cite{xiu2010BOOK,xiu2002SISC,zhu2017MMS,zhu2018SIAM,LW,Liao}. As a main drawback, however, their intrusive nature strongly modifies the original system and consequently they require the development of new {computational methods and implementations to be set into practice. }
Furthermore, when applied to kinetic and hyperbolic equations, {sG}-gPC techniques may lead to the loss of structural properties like positivity of the solution, entropy dissipation, loss of hyperbolicity and large time behavior^^>\cite{CZ2019,CPZ_17,despresBOOKCH,dimarco2017CHAPTER,GHY} and for these reasons, special care is needed when designed. 

{Beside sG-based methods, non-intrusive approaches for UQ have been developed in recent years like stochastic collocation (sC) methods^^>\cite{CEP,BNT,TZ,xiu2010BOOK,xiu_hesth,XF}, multifidelity and multilevel methods, and related approaches^^>\cite{dimarcoMF,LLiu,MS}. These methods have the nice feature to keep the structural properties of the underlying numerical solver for the deterministic problem and to avoid the construction of new discretization methods in the physical space. Even though sC methods guarantee the preservation of structural properties at each collocation node, the preservation of asymptotic properties in the whole random space is still an open problem which goes beyond the scopes of the present manuscript. In this direction we refer to^^>\cite{JXZ} where stochastic asymptotic preserving properties have been discussed in relation to PDEs with source terms. 
}

In the present work, we focus on the construction of {sG} numerical schemes for Fokker-Planck type equation such that the uncertainty dependent long time behavior of the solution is captured with high accuracy. This is achieved through a suitable reformulation of the problem using a generalization of the classical micro-macro decomposition^^>\cite{Liu} where the equilibrium part is also assumed time dependent. 
On the contrary, a standard implementation of {sG} methods leads to a loss of accuracy in the physical space when the long time behavior of the system is considered. A remarkable feature is that the novel approach provides highly accurate approximations even when the local equilibrium states of the system are not explicitly known. Analogous approaches in the deterministic case has been recently proposed in^^>\cite{Kuramoto,Filbet,Pareschi2017}.

A related problem concerns the possible loss of accuracy in the random space of the gPC projection for long-time computation, a well-known phenomenon that has already been studied in the literature (see^^>\cite{GSVK} and references therein). Here we have analyzed the problem with the help of an analytical solution of the Fokker-Planck equation highlighting the nature of this issue due to the vanishing of the variance in the random space for long times and the corresponding loss of uncertainty in the long-time behavior of the Fokker-Planck equation.  

The rest of the paper has the following structure. In the next section, we survey some examples of Fokker-Planck models with random inputs and their local equilibrium states, together with some regularity properties of the solution. In Section \ref{sec:SG} we first introduce some preliminary arguments concerning Stochastic Galerkin methods and then we detail the derivation of the novel micro-macro {sG} approach. Section \ref{sec:numerics} is devoted to present several numerical examples for different Fokker-Planck models including opinion formation and swarming dynamics that highlight the properties enjoyed by the proposed method in comparison with standard {sG} schemes. Finally, in Section \ref{sec:conc} we report some concluding considerations.

\section{Fokker-Planck type equations with random inputs}
In this section, we survey some examples of Fokker-Planck models with random inputs and the related equilibrium state solutions. First we introduce the classical Fokker-Planck equation with uncertainty in the diffusion coefficient. Next we discuss a kinetic swarming model where one additionally considers the space dependence. Finally, we focus on several recent examples of Fokker-Planck equations in the socio-economic sciences. A survey on some regularity results in the random space concludes the section.
We stress that the numerical approach we propose is not restricted to the above cases, even if the design of a suitable scheme for a different model may require special care and attention. 

\subsection{The classical Fokker-Planck equation}\label{sect:classic_FP}

The standard Fokker-Planck equation with uncertainty is obtained from \eqref{eq:MF_general}-\eqref{eq:J_meanfield} by assuming
\begin{equation}\label{eq:DP_FP}
{D(\z) = \sigma^2(\z)}, \qquad \textrm{and}\qquad P(\z,\v,\v_*) = 1
\end{equation}
which leads to 
\begin{equation}\label{eq:B_FP}
\mathcal B[f](\z,\v) = { \v-\u(\z)}, 
\end{equation}
where  
\be
\u(\z) = \int_{\RR^{d_v}} \v f(\z,\v,0)d\v,\qquad \sigma^2(\z)= \frac1{d_v}\int_{\RR^{d_v}} {|\v -\u(\z)|^2} f(\z,\v,0)d\v,
\ee
are the conserved uncertain momentum and temperature respectively.

The stationary distribution is characterized by a Maxwellian density defined as
\be
f^{\infty}(\z,\v) = \left(\dfrac{1}{2\pi \sigma(\z)^2}\right)^{d_\v/2} \exp\left\{ -\dfrac{|\v-\u(\z)|^2}{2\sigma^2(\z)} \right\}, 
\ee
where we assumed the total mass has been normalized to one.

The study of equilibrium solutions, complemented with boundary conditions on the velocity space and with a given initial distribution $f_0$, is of paramount importance in kinetic theory, see for instance^^>\cite{carrillo1998M2AS,toscani1999quarterly} and the recent works^^>\cite{AJW,furioli2017M3AS}. 
The trends to equilibrium in the space homogeneous case are then determined in terms of the Boltzmann H-functional  
\be
\mathbb H(f) = \int_{\mathbb R^{d_v}} f \log f d\v. 
\ee
In particular, if $f$ is solution to the space homogeneous initial value problem \eqref{eq:MF_general}-\eqref{eq:J_meanfield} with drift and diffusion given by \eqref{eq:DP_FP}-\eqref{eq:B_FP} and $f(\z,\v,0)$ has finite entropy, then $f(\z,\v,t)$ converges in relative entropy $\mathbb K(f) = \mathbb H(f)-\mathbb H(f^\infty)$ to the equilibrium $f^\infty(\z,\v)$ and
\be
\mathbb K(f(\z,\v,t)) \le e^{-2t/\sigma^2(\z)}\mathbb K(f(\z,\v,0)), \qquad t\ge 0,
\ee
meaning that the solution of the problem decays exponentially fast to equilibrium. Let observe anyway that, at variance with classical setting, here the rate of convergence depends on the uncertainty distribution. Regularity of solutions to this kind of problems has been studied for instance in^^>\cite{Lions} for constant diffusion matrices and in^^>\cite{Lions1} for irregular coefficients. 

\subsection{Kinetic swarming models with uncertainties}
\label{sec:swarming}
The next example concerns a kinetic equation for describing the swarming behavior^^>\cite{carrillo2010MSSET,CFRT,BCCD,CZ2019}. In particular we focus on a model with self--propulsion and uncertain diffusion, see^^>\cite{BarDeg,BCCD}. The dynamics for the density in the general case depends also on space $f=f(\z,\x,\v,t)$ and is described by the Vlasov-Fokker-Planck equation (\ref{eq:MF_general2}) characterized by 
\be
\begin{split}
	\label{eq:swarming_general1}
	\mathcal B[f](\z,\x,\v,t) = \alpha(\z) \v(1-|\v|^2)+(\v-\u_f(\v,\x,t)),
\end{split}
\ee
where 
\be
u_f(\z,\x,t) = \dfrac{\int_{\RR^{d_x}\times\RR^{d_v}}K(\x,\y)\v f(\z,\y,\v,t)\,d\v,d\y}{\int_{\RR^{d_x}\times \RR^{d_v}}K(\x,\y)f(\z,\y,\v,t)\,d\v\,d\y},
\ee
with $K(\x,\y)>0$ a localization kernel, $\alpha(\z)>0$ a self--propulsion term and $D(\z)>0$ the uncertain noise intensity. 

In the space-homogeneous case, i.e. $f=f(\z,\v,t)$, the model can be written as a gradient flow with uncertainties. Indeed, by defining
\[
\xi(\z,\v,t) = \Psi(\z,\v) + (U\ast f)(\z,\v,t) + D(\z)\log f(\z,\v,t),
\]
where $U(\v) = \frac{|\v|^2}{2}$ is a Coloumb potential and $\Psi(\z,\v)$ of the form 
\[
\Psi(\z,\v) = \alpha(\z) \left(\dfrac{|\v|^4}{4} - \dfrac{|\v|^2}{2} \right),
\] 
the space-homogeneous version reads
\be\label{eq:swarming_hom}
\frac{\partial}{\partial t} f(\z,\v,t) = \nabla_v \cdot \left( f(\z,\v,t) \nabla_v \xi(\z,\v,t) \right).
\ee
For the above equation it can be shown that the free energy functional $\mathcal E(\z,t)$ defined as
\[\begin{split}
\mathcal E(\z,t) =&\int_{\mathbb R^{d_v}} \left( \alpha(\z) \dfrac{|\v|^4}{4}+(1-\alpha(\z))\dfrac{|\v|^2}{2}\right)f(\z,\v,t)dv \\
&- \dfrac{|u_f(\z,t)|^2}{2} + D(\z) \int_{\mathbb R^{d_v}} f(\z,\v,t) \log f(\z,\v,t)dv,
\end{split}\]
where 
\[
u_f(\z,t) = \dfrac{\int_{\mathbb R^{d_v}}\v f(\z,\v,t)dv}{\int_{\mathbb{R}^{d_v}} f(\z,\v,t)dv},
\]
dissipates along solutions. The stationary solutions have consequently the form 
\be
f^{\infty}(\v,\z) = C\exp \left\{-\dfrac{1}{D(\z)}\left(\alpha(\z)\dfrac{|\v|^4}{4}+(1-\alpha(\z))\dfrac{|\v|^2}{2}-u_{f^{\infty}}(\z)\cdot \v \right) \right\},
\ee
with $C>0$ a normalization constant.

\subsection{Fokker-Planck models in the socio-economic sciences}\label{sect:FP_collective}
For different choices of the interaction functions \eqref{eq:Bf} and of non constant diffusion matrix $D(\z,\v)$, the Fokker-Planck model \eqref{eq:MF_general}-\eqref{eq:J_meanfield} is frequently used to describe different phenomena with respect to the classical physics which are related to the behavior of large systems of agents in social and life sciences^^>\cite{pareschi2013BOOK}. 
In the one dimensional case $v\in V \subseteq \RR$, which is the relevant situation in the case of opinion formation, wealth distributions and traffic dynamics, the local equilibrium distribution $f^\infty(z,v)$ of  problem \eqref{eq:MF_general}-\eqref{eq:J_meanfield} is given by the solution of the following differential equation
\be\label{eq:rel_inf_hom}
B(z,v)f^\infty(z,v)+\partial_v (D(z,v) f^{\infty}(z,v))=0,
\ee
where to simplify notations we considered $z\in I_{z} \subseteq \RR$.

Now, in order to enlighten the behavior of the Fokker-Planck equation with nonconstant diffusion we can study the monotone decreasing tendency of Lyapunov functionals of the solution as proposed in^^>\cite{furioli2017M3AS} in the deterministic setting. We concentrate on the case of relative Shannon entropy which is defined as 
\be\label{shannon}
H(f,g) = \int_V f\log \dfrac{f}{g}dv,
\ee
being $f,g$ two probability densities. Convergence to equilibrium can be proven by proving that the relative entropy is monotonically decreasing in time. We first need the following result 
\begin{prop}
Let us introduce the quotient $ F(z,v,t)=f(z,v,t)/f^{\infty}(z,v)$, therefore the evolution equation for $F(z,v,t)$ reads
\be\label{eq:Fhom}
\dfrac{\partial}{\partial t} F(z,v,t) = - B(z,v)\dfrac{\partial}{\partial v} F(z,v,t) + D(z,v)\dfrac{\partial^2}{\partial v^2} F(z,v,t) .  
\ee
\end{prop}
\begin{proof}
The proof is reminiscent of the results in^^>\cite{furioli2017M3AS}. We sketch in the following the main points. From \eqref{eq:rel_inf_hom} we have
	\be
	\dfrac{B(z,v)}{D(z,v)} = -\dfrac{\partial}{\partial v} \log(D(z,v)f^{\infty}(z,v)).
	\ee
	Then the right hand side of \eqref{eq:MF_general} can be written in the one dimensional case as
	\begin{equation*}
	\begin{split}
	\dfrac{\partial}{\partial v} (D(z,v)f(z,v,t)) &+ B(z,v)f(z,v,t) =  D(z,v)f(z,v,t) \left[\dfrac{\partial}{\partial v} \log(D(z,v)f(z,v,t))+\dfrac{B(z,v)}{D(z,v)} \right] \\
	& = D(z,v)f(z,v,t) \left[ \dfrac{\partial}{\partial v} \log(D(z,v)f(z,v,t))-\dfrac{\partial}{\partial v} \log(D(z,v)f^{\infty}(z,v)) \right]\\
	& = D(z,v)f(z,v,t) \dfrac{\partial}{\partial v} \log\left(\dfrac{f(z,v,t)}{f^{\infty}(z,v)}\right).
	\end{split}
	\end{equation*}
	Hence, we have  
	\be
	\begin{split}
	\dfrac{\partial}{\partial t} f(z,v,t) &=  \dfrac{\partial}{\partial v} \left[ D(z,v)f(z,v,t) \dfrac{\partial}{\partial v} \log \dfrac{f(z,v,t)}{f^{\infty}(z,v)}\right] = \dfrac{\partial}{ \partial v} \left[ D(z,v)f^{\infty}(z,v)\dfrac{\partial}{\partial v} \dfrac{f(z,v,t)}{f^{\infty}(z,v)}\right],
	\end{split}
	\ee
	which leads to 
	\be
	f^{\infty}(z,v)\dfrac{\partial}{\partial t}F(z,v,t) = D(z,v)f^{\infty}(z,v)\dfrac{\partial^2}{\partial v^2} F(z,v,t)-B(z,v) f^{\infty}(z,v) \dfrac{\partial}{\partial v}F(z,v,t),
	\ee
and we conclude. 
\end{proof}
We can now prove the result below which will directly imply the time decay of the Shannon entropy. 
\begin{prop}\label{prop:FPhom}
Let $F = F(z,v,t)$ be the solution of \eqref{eq:Fhom} endowed with no-flux boundary conditions for all $z\in I_{z}$. Then, for each smooth function $\psi(v)$ such that $|\psi(v)|\le c<+\infty$ if $v\in\partial V$, $V\subset \RR$ or at infinity if $V=\RR$, the following holds
\be
\int_V f^{\infty}(z,v)\psi(v)\frac{\partial}{\partial t} F(z,v,t)dv = -\int_V D(z,v)f^{\infty}(z,v)\partial_v \psi(v)\partial_vF(z,v,t)dv
\ee
\end{prop}
\begin{proof}
Without loss of generalization, we restrict to $V\subset \RR$. The no-flux boundary conditions imply conservation of mass and read
 \be\left. D(z,v)f^{\infty}(z,v)\dfrac{\partial}{\partial v} \dfrac{f(z,v,t)}{f^{\infty}(z,v)}\right\vert_{v\in\partial V}=0,\ee
then
\[
\int_V f^{\infty}(z,v)\psi(v)\frac{\partial}{\partial t} F(z,v,t)dv = \int_V f^{\infty}(z,v)\psi(v) \left(D(z,v)\dfrac{\partial^2}{\partial v^2} F(z,v,t) - B(z,v)\dfrac{\partial}{\partial v} F(z,v,t)\right)dv.
\]
Thus, integrating by part and using the boundary conditions leads to the claimed result, see^^>\cite{furioli2017M3AS} for details. 
\end{proof}
Using the above propositions, the decay of entropy is stated by the subsequent theorem.
\begin{thm}\label{thm1}
Let $\phi$ be a strictly convex function. Then, if $F= F(z,v,t)$, $v\in V\subset \RR$,  is a bounded solution to equation \eqref{eq:Fhom} the functional 
\be
\mathcal H[F]  (z,t)= \int_V f^{\infty}(z,v)\phi(F(z,v,t))dv,
\ee
is monotonically decreasing in time for all $z \in I_{z}$ and the following inequality holds 
\be
\dfrac{d}{dt} \mathcal H[F](z,t) = - I_{\mathcal H}(z,t), 
\ee
being 
\be
I_{\mathcal H}(z,t) = \int_V D(z,v) f^{\infty}(z,v) \phi^{\prime\prime}(F)\Big| \partial_v F(z,v,t)\Big|^2dv>0.
\ee
\end{thm}
\begin{proof}
We observe that for all $z\in I_{z}$ the functional $\mathcal H[\cdot]$ is uniformly bounded. Hence we apply Proposition \ref{prop:FPhom} with $\psi(v)= \phi^{\prime}(F)$. Then, the Shannon entropy decay, i.e. equation \eqref{shannon}, is obtained by taking $\phi(x)=\log x$.
\end{proof}

\subsubsection{Examples: opinion formation, traffic flows and wealth distributions}
\label{sec:examples}
We now focus on some Fokker-Planck applications having the structure 
\be
\frac{\partial}{\partial t} f(z,v,t) = \partial_v \left[ B(z,v) f(z,v,t) + \partial_v (D(z,v)f(z,v,t)) \right], \qquad v\in V\subseteq \RR.
\label{eq:FPs}
\ee
The large time behavior is identified once the form of the interaction function and of the local diffusion function is fixed. 

One example is the opinion formation model studied in^^>\cite{Tosc} where $v$ represents the opinion values with $V=[-1,1]$. For this model one has 
\begin{equation}
\label{eq:BD_opinion}
B(z,v) = \gamma(z)(v-u),\qquad D(z,v) = \dfrac{\sigma^2(z)}{2}(1-v^2)^2,
\end{equation}
 leading to the following equilibrium state
\be
f^\infty(z,v) = \dfrac{C_{z,u}}{(1-v^2)^2} \left( 1+v \right)^{\frac{\gamma(z)u}{2\sigma^2(z)}} \left( 1-v \right)^{\frac{\gamma(z)u}{2\sigma^2(z)}} \exp\left\{ -\dfrac{\gamma(z)(1-uv)}{\sigma^2(z)(1-v^2)} \right\}.
\label{eq:exact_opinion}
\ee
It is also possible to show that the mean opinion $u=\int_{I}vf(z,v,t)dv$ is preserved in time.

Instead, by taking the same drift term of \eqref{eq:BD_opinion} but local diffusion $D(z,v) = \frac{\sigma^2(z)}{2}(1-v^2)$ one obtains a beta equilibrium distribution 
\be
f^\infty(z,v) = \dfrac{(1+v)^{\frac{1+u}{\lambda(z)}-1}(1-v)^{\frac{1-u}{\lambda(z)}-1}}{2^{\frac{2}{\lambda(z)}-1}\mathbb{B}\left( \frac{1+u}{\lambda(z)},\frac{1-u}{\lambda(z)} \right)}, \qquad \lambda(z) = \dfrac{\sigma^2(z)}{\gamma(z)},
\label{eq:beta}
\ee
where $\mathbb B(\cdot,\cdot)$ is the Beta function. Besides opinion formation^^>\cite{PTTZ}, the above equilibrium state has been recently used, for instance, in traffic flows^^>\cite{dimarcotosin}. 

Related models which employ similar choices for the drift and diffusion terms regard the description of wealth distribution, see e.g.^^>\cite{pareschitoscani}. In this case, one considers $v\in\mathbb R_+$, a drift function like in \eqref{eq:BD_opinion} and local diffusion $D(z,v) = \frac{\sigma^2(z)}{2}v^2$. These choices produce an inverse gamma equilibrium
\be
f^\infty(z,v) = \dfrac{(\mu(z)-1)^{\mu(z)}}{\Gamma(\mu(z))v^{1+\mu(z)}}\exp\left\{-\dfrac{\mu(z)-1}{v}\right\},
\label{eq:wealth}
\ee 
begin $\mu(z) = 1+2\gamma(z)/\sigma^2(z)$ the Pareto exponent and $\Gamma(\cdot)$ the Gamma function.

\subsubsection{Quasi-equilibrium distributions}
In the general case, however, a clear analytical insight on the large time behavior of the system is often lost together with sharp trends to equilibrium. One way to tackle the problem consists in considering the solution of the vanishing flux condition in \eqref{eq:FPs} which lead to
\[
D(v,z)\partial_v f(z,v,t) = {-}(\mathcal B[f](z,v,t)+ \partial_v D(z,v)) f(z,v,t),\qquad {D(v,\cdot)>0,}
\]
which in general is not solvable except in special cases as the ones illustrated before. From the above identity, assuming without loss of generality $V=\RR$ and integrating we can write
\be\label{eq:quasi_eq}
f^q(z,v,t) = C_{z}\exp\left\{ - \int_{-\infty}^v \dfrac{\mathcal B[f](z,v_*,t) + \partial_v D(z,v_*)}{D(z,v_*)}dv_* \right\},
\ee
being $C_z>0$ a normalization constant. Therefore, unlike the previous cases, here $f^q$ is not the global equilibrium of the problem, but it shares with that steady state the property of annihilating the flux for any time $t\ge0$. It can be seen in particular that, if $\mathcal B[\cdot]$ does not depend on time, the quasi-equilibrium state \eqref{eq:quasi_eq} coincides with the steady state distribution of the corresponding Fokker-Planck equation, and so, it can be seen as a generalization of the steady state solutions we have illustrated before. As we will see the notion of quasi-equilibrium state will be essential for the derivation of numerical schemes which are capable to capture correctly the long time behavior of the system. 
For further insights on the use of quasi steady-state solutions to design numerical methods we refer to^^>\cite{CC,PZ2018,Pareschi2017,Filbet,Kuramoto}.

\begin{rem}
An important aspect should be noted in the dependence of equilibrium (and quasi-equilibrium) states on uncertainty. In some cases the choice of stochastic parameters can lead to a deterministic stationary behavior. For example, this is the case when $\gamma(z)=C\sigma_2(z)$, with $C>0$ constant, in \eqref{eq:exact_opinion}, \eqref{eq:beta} and \eqref{eq:wealth}. In such a situation, in fact, the uncertainty can be factorized externally to the Fokker-Planck operator and involves the relaxation time but not the steady state. Such cases are critical in terms of accuracy in the random space because the uncertainty decreases with time until it vanishes. We will discuss this issue in more detail in Section \ref{sec:numerics} with the aid of some numerical examples. We refer to^^>\cite{GSVK} for similar phenomena in uncertainty quantification through gPC approaches. 
\end{rem}

\subsection{Regularity of the solution in the random space}
We conclude the section by recalling some regularity results for the linear and non-linear Fokker-Planck model with respect to the random vector $\z\in \RR^{d_{z}}$. We fix $p(\z): I_\z \rightarrow \mathbb R_+$ the probability density of the random variable $\z$ which is supposed to be known. We also go back to the general case of $\v\in\RR_{d_v}$ and we define the weighted norm in $L^2_p(V\times I_\z)$ as follows
\be
\|f(t) \|_{L^2_p(V\times I_\z)} = \left(\int_{V}\int_{I_\z} |f(t,\v,\z)|^2 p(\z)d\z d\v \right)^{1/2}.
\ee
Sufficient conditions to guarantee regularity of the solution of the general Fokker-Planck model  \eqref{eq:MF_general}-\eqref{eq:J_meanfield} are discussed below (see also^^>\cite{zhu2018SIAM,zhu2017MMS,Liao,LW} {for related results}). 
We start from the classical Fokker-Planck equation with uncertain initial condition, see Section \ref{sect:classic_FP}, we have:
\begin{thm}
Let $f(\z,\v,t)$, $\v \in \RR^{d_v}$, $\z\in I_{\z}$, be the solution to
\begin{equation}
\label{eq:th1}
\frac{\partial}{\partial t} f(\z,\v,t) = \nabla_\v \cdot \left[ (\v-\u)f(\z,\v,t) + \dfrac{\sigma^2}{2}\nabla_\v f(\z,\v,t) \right]. 
\end{equation}
and let us define the space $H_p^m=\{f \in L^2_p(\mathbb R^{d_v} \times I_{\z}): \partial_{\z}^\ell f \in L^2_p(\mathbb R^{d_v} \times I_{\z}), \ \ell = 0,\dots,m\} $. If $ \partial_{\z}^\ell f(\z,\v,0) \in L^2_p(\RR^{d_v} \times I_{\z})$, $\ell = 0,\dots,m$,  then $f(\z,\v,t) \in H_p^m$ for all $t\ge 0$.
\end{thm}
\begin{proof}
We apply $\partial_{\z}^\ell$ to \eqref{eq:th1}  and we multiply the resulting equation by $2p(\z)\partial_{\z}^\ell f(\z,\v,t)$, $\ell = 0,\dots,m$. Integrating then over $\RR^{d_v} \times I_{\z}$ we have
\be
\begin{split}
&\frac{\partial}{\partial t} \|\partial_{\z}^\ell f(t) \|_{L^2_p(\RR^{d_v} \times I_{\z})} \\
&\quad= \int_{\RR^{d_v}  \times I_{\z}} \sum_{i=1}^{d_v} 2p(\z) \partial_{\z}^{\ell} f(\z,\v,t) \partial_{v_i}\left[(v_i-u_i)\partial_{\z}^{\ell}f(\z,\v,t) + \dfrac{\sigma^2}{2} \partial_{v_i} \partial_{\z}^{\ell}f(\z,\v,t) \right] d\v d\z. 
\end{split}\ee
We get
\be
\begin{split}
&\int_{\RR^{d_v} \times I_{\z}} \sum_{i=1}^{d_v}2p(\z) \partial_{\z}^{\ell} f(\z,\v,t) \partial_{v_i}\left[(v_i-u_i)\partial_{\z}^{\ell}f(\z,\v,t) + \dfrac{\sigma^2}{2} \partial_{v_i} \partial_{\z}^{\ell} f(\z,\v,t) \right] d\z\, d\v\\
&=  \int_{\RR^{d_v} \times I_{\z}} \sum_{i=1}^{d_v} 2p(\z) \left(\partial_{\z}^{\ell}f^2(\z,\v,t) + (v_i-u_i)\partial_{\z}^{\ell}f(\z,\v,t)\partial_{v_i}\partial_{\z}^{\ell} f(\z,\v,t) \right)d\z\, d\v \\
&\quad+ \sigma^2 \int_{\RR^{d_v} \times I_{\z}}\sum_{i=1}^{d_v} p(\z) \partial_{\z}^{\ell} f(\z,v,t)\partial_{v_i}^2 \partial_{\z}^{\ell} f(\z,\v,t)d\z\, d\v \\
&= \|\partial_{\z}^\ell f(t)\|_{L_p^2(\RR^{d_v} \times I_{\z})}^2 -  \int_{\RR^{d_v} \times I_{\z}} \sum_{i=1}^{d_v} 2p(\z) \partial_{\z}^{\ell} f(\z,\v,t)\partial_{v_i}( (v_i-u_i)\partial_{\z}^{\ell}f(\z,\v,t))d\z\,d\v \\
&\quad-  \sigma^2 \int_{\RR^{d_v}\times I_{\z}}\sum_{i=1}^{d_v} p(\z) \left(\partial_{v_i}  \partial_{\z}^{\ell} f(\z,\v,t)\right)^2d\z\,d\v. 
\end{split}
\ee
From which we obtain 
\be
\frac{\partial}{\partial t} \| \partial_{\z}^{\ell}f(t)\|^2_{L^2_p(\RR^{d_v} \times I_{\z})} \le \dfrac{1}{2} \| \partial_{\z}^{\ell}f \|_{L^2_p(\RR^{d_v} \times I_{\z})}^2
\ee
and thanks to the Gronwall inequality we have for all $\ell = 0,\dots,m$
\be\label{grn}
\| \partial_{\z}^\ell f(t) \|^2_{L^2_p(\RR \times I_{\z})} \le e^{t/2 } \| \partial_{\z}^\ell f(0) \|_{L^2_p(\RR \times I_{\z})}^2, 
\ee
from which we conclude. 
\end{proof}

For the more general case of the Fokker-Planck equations with nonlocal terms we instead have
\begin{thm}
Let $f(\z,\v,t)$, $\v \in \RR^{d_v}$, $\z \in I_{\z}$, be the solution of \eqref{eq:MF_general}-\eqref{eq:J_meanfield}. We suppose that $\partial_{\z}^{\ell} f(\z,\v,t) \in L^2_p(\RR^{d_v} \times I_{\z})$, $\ell = 0,\dots,m$ and that that constants $c_{\mathcal B,1}, c_{\mathcal B,2} <+\infty$, $c_{D}<+\infty$ exist such that 
\be
\|\partial_{\z}^k \mathcal B[f](\z,\v,t)  \|_{L^\infty(\mathbb R^{d_v} \times I_{\z})}\le c_{\mathcal B,1}, \qquad \|\partial_{\z}^k \partial_{v_i} \mathcal B[f](\z,\v,t)  \|_{L^\infty(\mathbb R^{d_v}\times I_{\z})}\le c_{\mathcal B,2}
\ee
for all $k = 0,\dots,m$  and $\| D(\z,v) \|_{L^{\infty}(\RR^{d_v} \times I_{\z})}<c_D$. If $f(\z,\v,t=0) \in H^{m}_p$ then $f(\z,\v,t) \in H^{m}_p$ for all $t>0$.  
\end{thm}
\begin{proof}
Applying $\partial_{\z}^{\ell}$ to \eqref{eq:MF_general} we get 
\be
\partial_{\z}^{\ell} \frac{\partial}{\partial t} f(\z,\v,t) = \sum_{i=1}^{d_v} \partial_{\z}^\ell \partial_{v_i} \left[ \mathcal B[f](\z,\v,t)f(\z,\v,t) + \partial_{v_i}(D(\z,v)f(\z,\v,t)) \right]. 
\ee
Multiplying the above equation by $2 p(\z) \partial_{\z}^{\ell} f$ and integrating over $\RR^{d_v} \times I_{\z}$ we get 
\begin{equation}
\label{eq:thMF}
\begin{split}
&\frac{\partial}{\partial t} \|\partial_{\z}^\ell f(t)\|_{L^2_p(\RR^{d_v} \times I_{\z})} \\
&\quad= \int_{\RR^{d_v} \times I_{\z}} \sum_{i=1}^{d_v} 2p(\z)\partial_{\z}^\ell f \partial_{\z}^\ell \left[ \mathcal B[f](\z,\v,t)\partial_{v_i}f(\z,\v,t) + f(\z,\v,t) \partial_{v_i}\mathcal B[f](\z,\v,t)\right]d\z\,d\v \\
&\quad+\int_{\RR^{d_v} \times I_{\z}} \sum_{i=1}^{d_v} 2p(\z)\partial_{\z}^\ell f\partial_{\z}^\ell\partial_{v_i}^2 \left( D(\z,v)f(\z,\v,t) \right)d\z\,d\v. 
\end{split}
\end{equation}
Hence, we have
\be
\begin{split}
\partial_{\z}^\ell \left( \mathcal B[f](\z,\v,t)\partial_{v_i}f(\z,\v,t)\right) &= \sum_{k=0}^{\ell} \binom{\ell}{k}\partial_{\z}^{k} \mathcal B[f](\z,\v,t) \partial_{\z}^{\ell-k}\partial_{v_i}f(\z,\v,t)\\
& \le c_{\mathcal B,1} \partial_{\z}^\ell \partial_{v_i}f(\z,\v,t)
\end{split}
\ee
and
\be
\begin{split}
\partial_{\z}^\ell \left( \partial_{v_i}\mathcal B[f](\z,\v,t)f(\z,\v,t)\right) &= \sum_{k=0}^{\ell} \binom{\ell}{k}\partial_{\z}^{k} \partial_{v_i}\mathcal B[f](\z,\v,t) \partial_{\z}^{\ell-k}f(\z,\v,t)\\
& \le c_{\mathcal B,2} \partial_{\z}^\ell f(\z,\v,t). 
\end{split}
\ee
From \eqref{eq:thMF} we  have
\begin{equation}
\begin{split}
&\frac{\partial}{\partial t} \| \partial_{\z}^\ell f(t)\|_{L^2_p(\RR^{d_v} \times I_{\z})}\\
&\quad \le c_{\mathcal B,1} \int_{\RR^{d_v} \times I_{\z}} \sum_{i=1}^{d_v}2p(\z) \partial_{\z}^\ell f \partial_{v_i}\partial_\z^\ell f(\z,\v,t)d\z\,dv+ c_{\mathcal B,2} \| \partial_{\z}^\ell f(\z,\v,t)\|_{L^2_p(\RR^{d_v}\times I_{\z})}^2 \\
&\qquad- 2c_D\| \nabla_v \cdot f\|^2_{L^2_p(\RR^{d_v} \times I_{\z})}. 
\end{split}
\end{equation}
We then observe that
\be
\int_{\RR^{d_v} \times I_{\z}}2p(\z) \partial_{\z}^\ell f \partial_{v_i}\partial_\z^\ell f(\z,\v,t)d\z\,d\v = 0, \qquad i = 1,\dots,d_v
\ee
since the boundary terms vanish, from which we can conclude thanks to the Gronwall inequality as done in \eqref{grn} . 
\end{proof}


 \section{Micro-macro stochastic Galerkin methods}\label{sec:SG}
 In this section we introduce the stochastic Galerkin {(sG)} framework for the numerical approximation of nonlinear Fokker-Planck type equations in the form \eqref{eq:MF_general}-\eqref{eq:J_meanfield}. In particular, by means of a suitable reformulation of the Fokker-Planck model in the {sG} setting, we will show how it is possible to construct numerical methods that preserve the asymptotic behavior of the solution as opposite to the direct application of the {sG} approach to the original equation. Within this choice, we will see that highly accurate results compared to a standard approach can be obtained as soon as one gets closer to the steady state solution. We first recall the basic of the {sG} method in the context of kinetic equations following^^>\cite{Hu2017,Pareschi} and we derive some stability results for the models of interest for what concerns the {sG} expansion. 
 
 \subsection{Preliminaries on stochastic Galerkin methods}\label{sec: SG1}
Let $(\Omega,\mathcal F,P)$ be a probability space and let us define a random variable 
\be
\z:(\Omega,\mathcal F)\rightarrow (I_{z},\mathcal B_{\RR^{d_{z}}}), 
\ee
with $I_{z}\subseteq\RR^{d_{z}}$ and $\mathcal B_{\RR^{d_{z}}}$ the Borel set. We consider a function $f(\z,\v,t)$, $f\in L^2(\Omega)$ in the random space. Let now $\mathbb P_M$ be the orthogonal polynomial space of degree $M$
\be
\mathbb P_M = \{ g : I_{z}\rightarrow \RR: g\in \textrm{span}\{\Phi_h\}_{h=0}^M \},
\ee
where $\{\Phi_h\}_{h=0}^M$ is a set of polynomials of $\z$ with degree up to $M\ge 0$. These polynomials form an orthonormal basis of $L^2(\Omega)$, i.e.
\be
\int_{I_{z}}\Phi_h(\z)\Phi_k(\z)\,p(\z)\,d\z = \delta_{hk},
\ee
with $\delta_{hk}$ the Kronecker delta function and $p(\z)$ the probability distribution function random variable $\z\in I_{z}$. 
The function $f$ can be represented in $L^2(\Omega)$ as follows 
\be
f(\z,\v,t) = \sum_{h=0}^{\infty}  \int_{I_{z}}f(\z,\v,t)\Phi_{{h}}(\z)p(\z)d\z \Phi_{{h}}(\z) =  \sum_{k=0}^{\infty} \hat f_{{h}}(\v,t)\Phi_{{h}}(\z).
\ee
In this setting, the idea behind generalized polynomial chaos method is to approximate  $f(\z,\v,t)$ in $\mathbb P_M$ with the orthogonal projection $\mathcal P_M: L^2(\Omega)\to \mathbb P_M$ defined as $\mathcal P_M f = f_M$ with 
\be
f_M(\z,\v,t)=\sum_{{h=0}}^M \hat f_{{h}}(\v,t)\Phi_{{h}}(\z).
\ee
It is well-known that the above projection, referred to as generalized polynomial chaos (gPC) expansion, corresponds to the polynomial of best approximation in $\mathbb P_M$ (see^^>\cite{xiu2010BOOK}).
A stochastic Galerkin method for \eqref{eq:MF_general}-\eqref{eq:J_meanfield} is then obtained as
\be\label{SG_Eq}
\dfrac{\partial}{\partial t}f_M(\z,\v,t) = \mathcal P_M \mathcal J(f_M,f_M)(\z,\v,t) = \mathcal J_M(f_M,f_M)(\z,\v,t), 
\ee
with 
\be\label{eq:J_meanfieldSG}
\mathcal J(f_M,f_M) (\z,\v,t)= \nabla_v \cdot \left[ \mathcal B[f_M](\z,\v,t)f_M(\z,\v,t)+\nabla_v (D(\z,\v)f_M(\z,\v,t)) \right].
\ee
{In a stochastic Galerkin framework, i}n order to derive a system of differential equations for the coefficients $\hat f_h(\v,t)$, $h=1,..,M$ of the expansion, after substituting $f$ with $f_M$ in \eqref{eq:MF_general}-\eqref{eq:J_meanfield}, one multiplies the resulting equation by $\Phi_h(\z)$ and takes the expectation over the random space $I_z$ to get
\be
\begin{split}
\frac{\partial}{\partial t} \int_{I_{z}}f_M(\z,\v,t)\Phi_h(\z)p(\z)d\z &=  \int_{I_{z}}\mathcal J(f_M,f_M)(\z,\v,t)\Phi_h(\z)p(\z)d\z\end{split}.
\ee
Hence, from the orthogonality of the polynomial basis of the linear space we obtain a coupled system of $(M+1)$ equations describing the evolution of the projection coefficients
\be\label{eq:gPCMF}
\begin{split}
\frac{\partial}{\partial t} \hat f_h(\v,t)  = \nabla_v \cdot \left[ \sum_{k=0}^M \mathcal{B}_{hk}[f_M]\hat f_k(\v,t) + \nabla_v \left(\sum_{k=0}^M D_{hk}(\v)\hat f_k(\v,t) \right) \right],\qquad {h = 0,\dots,M},
\end{split}
\ee
where
\be\label{eq:Bhk}
\begin{split}
\mathcal{B}_{hk}[f_M](\v,t)  =&  \int_{I_{z}} \mathcal B[ f_M]\Phi_k(\z)\Phi_h(\z)\,p(\z)\,d\z \\
=& \int_{I_z} \int_V P(\z,\v,\v_*)|\v-\v_*|f_M(\z,\v_*,t)d\v_* \Phi_k(\z)\Phi_h(\z)\,p(\z)\,d\z\\
=& \sum_{l=0}^M\int_{I_z} \int_V P(\z,\v,\v_*)|\v-\v_*| \hat f_l(\v_*,t)\Phi_l(\z)d\v_* \Phi_k(\z)\Phi_h(\z)\,p(\z)\,d\z
\end{split}\ee
and 
\be\label{eq:Dhk}
D_{hk}(\v)  =  \int_{I_{z}} D(\z,\v)\Phi_k(\z)\Phi_h(\z)\,p(\z)\,d\z. 
\ee
Equation \eqref{eq:gPCMF} can be rewritten in vector notations as follows
\be\begin{split}\label{eq:sG_system}
\frac{\partial}{\partial t} \hat{\textbf{f}}(\v,t) =  \nabla_v \cdot \left[\textbf{B}[f_M](\v,t)\hat{\textbf{f}}(\v,t)  + \nabla_v \textbf{D}(\v)\hat{\textbf{f}}(\v,t)  \right],
\end{split}\ee
where $\hat{\textbf{f}}(\v,t)= \left( \hat f_0(\v,t),\dots,\hat f_M(\v,t)\right)^T$ and the components of the matrices $\textbf B[f_M] = \{\mathcal B_{hk}[f_M]\}_{h,k=0}^M$, $\textbf{D}=\{D_{hk}\}_{h,k=0}^M$ are given by \eqref{eq:Bhk}-\eqref{eq:Dhk}.
Indicating now with $\|\hat{\textbf{f}} \|_{L^2(V)}$ the standard $L^2$ norm of the vector $\hat{\textbf{f}}$ in the physical space
\be
\|\hat{\textbf{f}} \|_{L^2(V)} = \left( \int_{ V}\sum_{{h=0}}^M \hat f_{{h}}^2 d\v\right)^{1/2},
\ee
we observe that since the norm of $f_M$ in $L_p^2(I_z\times V)$ reads
\be
\| f_M\|_{L_p^2(I_z\times V)} =\left( \int_{I_z} \int_{V} \left(\sum_{{h=0}}^M \hat f_{{h}} \Phi_{{h}}(\z) \right)^2 d\v p(\z)d\z \right)^{1/2}
\ee
then, thanks to the orthonormality of the chosen polynomial basis we have 
$$\|f_M \|_{L_p^2(I_z\times V)} = \|\hat{\textbf{f}} \|_{L^2(V)}.$$
Furthermore, from the definition of the matrices $\textbf B$, $\textbf D$ we have
\be
\textbf B_{hk} = \textbf B_{kh}\qquad \textbf D_{hk} = \textbf D_{kh},
\ee
and therefore $\textbf{B}$, $\textbf{D}$ are symmetric. 

Once that the system of equations \eqref{eq:sG_system} has been derived then it can be solved through suitable numerical methods which approximates the time and the physical variables. To that aim, insights on statistical quantities like the expected value and variance of the solution of the differential problem are defined in terms of the introduced projections into the space of the stochastic polynomials. In particular the following quantities are defined from the gPC-{sG} methods
\be
\mathbb E_{\z}[f(\z,\v,t)] \approx \hat f_0(\v,t),
\ee
and 
\be
\textrm{Var}_{\z}[f(\z,\v,t)] \approx \int_{I_{z}}\left( \sum_{{h=0}}^M \hat f_{{h}}(\v,t)\Phi_{{h}}(\z)-\hat f_0(\v,t) \right)^2 p(\z)d\z,
\ee
and thanks to the orthogonality of polynomial expansion
\be
\textrm{Var}_{\z}[f(\z,\v,t)] \approx \sum_{{h=0}}^M \hat f_{{h}}^2(\v,t)- \hat f_0^2(\v,t).
\ee
In the following, considering $D(\z,\v)= D(\v)$, we can prove the following stability result:
\begin{thm}
\normalfont
If $\|\partial_v \mathcal B_{hk}\|_{L^{\infty}}\le C_{\textbf{B}}$ with $C_{\textbf{B}}>0$ for all $h,k=0,\dots,M$ and if the diffusion functions are such that $D(\v)\le C_{\textbf{D}}$ for all $h,k=0,\dots,M$, hence we have
\[
\|\hat{\textbf{f}}(t) \|_{L^2}^2 \le e^{t(C_{\textbf{B}})} \| \hat{\textbf{f}}(0) \|_{L^2}^2
\]
hence stability of the {sG} expansion.
\end{thm}
\begin{proof}
 We multiply \eqref{eq:gPCMF} by $\hat f_h$ and we integrate over $V$, this gives
\be\label{eq:thm_stab1}
\begin{split}
&\int_{V} \left[ \frac{\partial}{\partial t} \left(\dfrac{1}{2}\hat f_h^2\right) \right] dv=  \int_{V} \left\{\sum_{i=1}^{d_v}\partial_{v_i}  \left[ \sum_{k=0}^M \mathcal B_{hk}\hat f_k + \partial_{v_i}  D\hat f_h \right]\right\} \hat f_h  d\v.
\end{split}
\ee
Furthermore, we have integrating by part and using the zero flux boundary conditions
\be\label{eq:thm_stab2}
\begin{split}
&\sum_{k=0}^M \int_{V} \hat f_h \partial_{v_i} \left( \mathcal B_{hk}\hat f_k\right)  d\v \\
&\qquad = \sum_{k=0}^M \int_{V}  \left( \hat f_k \hat f_h\partial_{v_i}  \mathcal B_{hk} + \mathcal B_{hk}  \hat f_h\partial_{v_i}\hat f_k \right)  d\v \\
& \qquad  = - \sum_{k=0}^M \int_{V} \mathcal B_{hk} \partial_{v_i} \left(\hat f_{h}\hat f_{k}\right)d\v-\sum_{k=0}^M \int_{V}  \hat f_k\partial_{v_i} \left( \mathcal B_{hk}\hat f_h \right) d\v
\end{split}
\ee
now from \eqref{eq:thm_stab2} summing $h$ up to $M$ and thanks to the symmetry of $\textbf B$ we have
\[
\begin{split}
2\sum_{h,k=0}^M \int_{V} \hat f_h \partial_{v_i}(\mathcal  B_{hk}\hat f_k) d\v &= - \sum_{h,k=0}^M \int_{V} \mathcal B_{hk}\partial_{v_i}  (\hat f_{k}\hat f_{h})d\v \\
&= \sum_{h,k=0}^M \int_{ V} \hat f_h \hat f_k (\partial_{v_i} \mathcal B_{hk})d\v
\end{split}
\]
and since $\|\partial_v \mathcal B_{hk} \|_{L^\infty}\le C_{\mathbf B}$ we obtain
\[
\begin{split}
\sum_{h,k=0}^M \int_{V} \hat f_h \partial_{v_i}( \mathcal B_{hk}\hat f_k) d\v \le \dfrac{C_{\textbf{B}}}{2} \sum_{h,k=0}^M \int_{V} \hat f_h \hat f_k d\v, 
\end{split}\]
and from the Cauchy-Schwarz inequality we obtain
\be\label{eq:ineq_B}
\sum_{h,k=0}^M \int_{V} \hat f_h \partial_{v_i}( \mathcal B_{hk}\hat f_k) d\v \le \dfrac{C_{\textbf{B}}}{2} \| \hat{\textbf{f}} \|_{L^2}^2
\ee
Furthermore since  $ D(\v) \le C_{\textbf{D}}$ 
 we have 
\be\label{eq:ineq_D}
\begin{split}
	\sum_{h=0}^M \int_{V} \partial_{v_i}  (\partial_{v_i} D\hat f_h )\hat f_hd\v\le  - C_{\textbf{D}}\sum_{h=0}^M \int_{V}( \partial_{v_i}  \hat f^2_h) d\v
\end{split}
\ee
Hence
\[
\begin{split}
\dfrac{1}{2}\frac{\partial}{\partial t} \| \hat{\textbf{f}} (t)\|^2_{L^2} &\le \dfrac{C_{\textbf{B}}}{2}\|\hat{\textbf{f}}(t) \|_{L^2}^2 - C_{\textbf{D}} \| \partial_v \hat{\textbf{f}}(t) \|_{L^2}^2\\
&\le \left(\dfrac{C_{\textbf{B}}}{2} \right) \| \hat{\textbf{f}}(t) \|_{L^2}^2
\end{split}\]
Finally, thank to the Gronwall's theorem we can conclude. 
\end{proof}

\subsection{Equilibrium preserving stochastic Galerkin schemes}\label{sec:micro_macro}


{Let us know consider a simple formal analysis of the consistency error of the sG method discussed in the previous section. As} well known, for the {sG} approximation if the function $f\in H^{s}(I_z)$ one has an error estimate of the type^^>\cite{xiu2010BOOK}
\be\label{consistency}
\|f-f_M\|_{L_p^2(I_z)}\leq C_s M^{-s}
\ee
where $C_s$ represents a constant that depends on the regularity of the distribution $f(\z,\v,t)$ with respect to the random variable $\z$. 
When, in addition, equation \eqref{eq:sG_system} is additionally discretized in the physical space then for the approximated distribution $f_{M,\Delta v}$, where $\Delta v$ represents the typical mesh size, one gets an estimate of the form
\be\label{consistency1}
\|f-f_{M,\Delta v}\|_{L_p^2(I_z\times V)}\leq C_s M^{-s}+ C_v(\Delta v^p)
\ee
where $p$ is the order of the discretization employed in the physical space and $C_v$ a suitable constant. However, for large times, the behavior of the distribution $f(\z,\v,t)$ tends towards an equilibrium state $f^\infty(\z,\v)$ for which one has that $\mathcal J(f^\infty,f^\infty)(\z,\v,t)=0$. Now, for the fully discretized scheme, it holds an analogous estimate of \eqref{consistency1}, i.e.
\be\label{residual}
\| \mathcal J_M (f^\infty_{M,\Delta v},f^\infty_{M,\Delta v})\|_{L_p^2(I_z\times V)}\leq C_{s} M^{-s}+ C_{v}(\Delta v^p),
\ee
and therefore the method does not possess the highly desirable property that the orthogonal projection of the equilibrium distribution in the polynomial space $\mathbb P_M$ annihilate the right hand side at the discrete level. In other words, {we get 
$\mathcal J_M(f_M^\infty,f_M^\infty)(\z,\v,t)\neq 0$ whereas at the continuous level we have $\mathcal J(f^\infty,f^\infty)(\z,\v,t)=0$. 
Preserving the latter structural property at a discrete level}, in fact, will avoid the accumulation of errors for large times and permits to have an accurate description of steady state solutions. Several works addressed this problem at the deterministic level (see for example^^>\cite{CC,PZ2018,Zan20,Gosse,Pareschi2017,Filbet,Kuramoto}), however, due to the intrusive nature of stochastic Galerkin methods these properties are lost during the projection process. 


\subsubsection{A micro-macro stochastic Galerkin formulation}
In order to guarantee {accuracy for large times} in the context of stochastic-Galerkin schemes  
we propose a new method based on a micro-macro decomposition. A similar technique has been also recently developed in the fully deterministic case to construct spectral methods for Boltzmann-type collision operators that preserves exactly the Maxwellian asymptotic profile of the system^^>\cite{Pareschi2017,Filbet}. 

Let us recall at this point the notion of quasi-equilibrium state \eqref{eq:quasi_eq} associated to the nonlinear Fokker-Planck equation of interest introduced in the previous section. We observe that $f^q(\z,\v,t)$ is a time dependent distribution satisfying at each instant of time the relation $J(f^q,f^q)(\z,\v,t) = 0$. Therefore, similarly to what happens with classical micro-macro formulations for each time $t\ge 0$, one can decompose the solution of the Fokker-Planck model \eqref{eq:MF_general}-\eqref{eq:J_meanfield} in a quasi-equilibrium part and a reminder term as follows
\be\label{mm}
f(\z,\v,t) = f^q(\z,\v,t) + g(\z,\v,t),
\ee
with $g(\z,\v,t)$ a distribution such that
\be\label{mom}
\int_{\mathbb R^{d_v}}\varphi(\v)g(\z,\v,t)d\v = 0 \qquad \textrm{if}\qquad  \int_{\mathbb R^{d_v}}\varphi(\v)f(\z,\v,t)d\v =  \int_{\mathbb R^{d_v}}\varphi(\v)f^q(\z,\v,t)d\v,
\ee
for all test functions $\varphi(\v)$ characterizing the conservation properties of the Fokker-Planck equation under study. For example, if the quasi-equilibrium state shares the same density of the original kinetic distribution $f(\z,\v,t)$, then for $\phi(\v)=1$ equation \eqref{mom} is satisfied.
We have the following result:
\begin{prop}
Let us consider the nonlinear Fokker-Planck model \eqref{eq:MF_general}-\eqref{eq:J_meanfield}. For each time $t\ge 0$ and $\z\in I_z$ using the micro-macro decomposition \eqref{mm} the operator $\mathcal J(f,f)$ can be rewritten as 
\be
\mathcal J(f,f) = \mathcal J(g,g) + \mathcal G(f^q,g),
\ee
where 
\be
\mathcal G(f^q,g)(\z,\v,t) = \nabla_v\cdot\left(\mathcal B[f^q](\z,\v,t)g(\z,\v,t) + \mathcal B[g](\z,\v,t)f^q(\z,\v,t)\right).
\label{eq:G}
\ee
Moreover, if $f(\z,\v,t) \to f^\infty(\z,\v)$ as $t\to \infty$ then $f^q(\z,\v,t)\to f^\infty(\z,\v)$ and 
$g(\z,\v,t) \to 0$.
\end{prop}
\begin{proof}
The above proposition is proved by substituting \eqref{mm} into the right hand side of \eqref{eq:MF_general} and using the fact that $\mathcal J(f^q,f^q)=0$. The last part is a consequence of \eqref{mm} and the definition of quasi-equilibrium state  \eqref{eq:quasi_eq}.  
\end{proof}
We observe that, in view of the results summarized in Section \ref{sect:classic_FP} for classic Fokker-Planck equations, the nonequilibrium part $g(\z,\v,t)$ exponentially decays to zero for all $\z\in I_{z}$ or in more general situations, as for instance \eqref{sect:FP_collective}, at least a slower but monotone decay can be observed. 

The key idea now is to apply the stochastic-Galerkin approximation to problem
\be
\begin{split}
\frac{\partial}{\partial t}f(\z,\v,t) &= \mathcal J(g,g)(\z,\v,t) + \mathcal G(f^q,g)(\z,\v,t)
\end{split}
\label{eq:MF_eq}
\ee
with $f=f^q+g$, 
instead of the equivalent problem \eqref{eq:MF_general}-\eqref{eq:J_meanfield}. Note, in fact, that the formulation \eqref{eq:MF_eq} with $\mathcal G(f^q,g)$ defined by \eqref{eq:G} embeds the equilibrium assumption $\mathcal J(f^q,f^q)=0$. 

Therefore, in analogy with the micro-macro decomposition \eqref{mm}, in the sG setting we introduce the following decomposition 
\be
f_M(\z,\v,t) = f^q_M(\z,\v,t) + g_M(\z,\v,t),
\label{mmsG}
\ee
being 
\be
f^q_M(\z,\v,t) = \sum_{h=0}^M \widehat{f_h^q}(\v,t)\Phi_h(\z),\qquad \widehat{f^q_h}(\v,t) = \int_{I_z}f^q(\z,\v,t)\Phi_h(\z)p(\z)d\z.
\label{eq:eqpr}
\ee
The resulting stochastic Galerkin scheme then reads
\begin{equation}\label{eq:MMSG}
\begin{split}
\dfrac{\partial}{\partial t} f_M(\z,\v,t) &= \mathcal J_M(g_M,g_M)(\z,\v,t)  + \mathcal G_M(f^q_M,g_M)(\z,\v,t).
\end{split}
\end{equation}
Proceeding now as in Section \ref{sec: SG1}, we obtain the following system of coupled PDEs for the evolution of the coefficients $\hat f_h$, {$h=0,\ldots,M$}
\be
\label{eq:fhath}
\begin{split}
\dfrac{\partial}{\partial t} \hat f_h(\v,t) =& \nabla_v \cdot \left[ \sum_{k=0}^M \mathcal B_{hk}[g_M]\hat g_k + \nabla_v \left( \sum_{k=0}^M D_{hk}(v) \hat g_k\right) \right]\\
& + \nabla_v \cdot \left[ \sum_{k=0}^M \mathcal{B}_{hk}[f^q_M] \hat g_k + \mathcal{B}_{hk}[g_M]\widehat{f^q_k}\right],
\end{split}
\ee
where $B_{hk}[g_M]$ and $B_{hk}[f^q_M]$ are defined as
\be\begin{split}\label{eq:Bhk1}
 &\mathcal{ B}_{hk}[g_M](\v,t)  = \int_{I_{z}} \mathcal B[ g^M]\Phi_k(\z)\Phi_h(\z)p(\z)d\z,\\
& \mathcal{B}_{hk}[f^q_M](\v,t)  = \int_{I_{z}} \mathcal B[ f^q_M]\Phi_k(\z)\Phi_h(\z)p(\z)d\z,
\end{split}\ee
and $\mathcal B[\cdot]$ is the operator defined in \eqref{eq:Bf}. 
 In order to write the resulting system in a more compact vector form, we introduce the vectors 
\be
\hg(\v,t) = \left( \hat g_0(\v,t),\dots,\hat g_M(\v,t)\right)^T,\qquad \hfq(\v,t) = \left( \widehat{f^q}_0(\v,t),\dots,\widehat{f^q}_M(\v,t)\right)^T
\ee
and the matrices $\textbf{B}[g_M] = \left\{\mathcal B_{hk}[g_M] \right\}_{h,k=0}^M$, and $\textbf{B}[f^q_M] = \left\{\mathcal B_{hk}[f^q_M] \right\}_{h,k=0}^M$  together with the diffusion matrix $\textbf{D} = \left\{D_ {hk}\right \}_{h,k=0}^M$ defined in \eqref{eq:Dhk}. Hence the micro-macro gPC scheme may be formulated in vector form as follows
\begin{equation}\label{mm1}
\begin{split}
\dfrac{\partial }{\partial t} \hf(\v,t) =& \nabla_v \cdot \left[ \textbf{B}[g_M](\v,t) \hg(\v,t) + \nabla_v (\textbf{D}(\v)\hg(\v,t)) \right] \\
&+ \nabla_v \cdot \left[ \textbf{B}[f^q_M] \hg(\v,t) + \textbf{B}[g_M]\hfq(\v,t)\right].
\end{split}\end{equation}
For practical purposes, however, it is more convenient to use a different equivalent formulation which also enlighten the relationship with the standard sG approximation \eqref{eq:sG_system}. To this aim we can prove the following result.
\begin{prop}\label{prop4}
Let us consider the stochastic-Galerkin scheme \eqref{mm1} obtained from the micro-macro decomposition \eqref{mmsG} applied to problem \eqref{eq:MF_eq}-\eqref{eq:G}. We have equivalently
{\rm \begin{equation}
\label{eq:MM_ref}
\begin{split}
\dfrac{\partial }{\partial t} \hf(\v,t) =&  \nabla_v \cdot \left[\textbf{B}[f_M](\v,t)\hat{\textbf{f}}(\v,t)  + \nabla_v \textbf{D}(\v)\hat{\textbf{f}}(\v,t)  \right] \\
&\qquad- \nabla_v \cdot \left[\textbf{B}[f^q_M](\v,t)\hfq(\v,t)  + \nabla_v \textbf{D}(\v)\hfq(\v,t)  \right].
\end{split}
\end{equation}}
\end{prop}
\begin{proof}
In fact we have from \eqref{mmsG}
\[
\mathcal J_M(f_M,f_M) = \mathcal J_M(g_M,g_M) + \mathcal J_M(f^q_M,f^q_M) + \mathcal G_M(f^q_M,g_M)
\]
and then we can rewrite \eqref{eq:MMSG} as
\begin{equation}\label{eq:MMSG2}
\begin{split}
\dfrac{\partial}{\partial t} f_M(\z,\v,t) &= \mathcal J_M(f_M,f_M)(\z,\v,t) - \mathcal J(f^q_M,f^q_M)(\z,\v,t).
\end{split}
\end{equation}
The result then follows simply rewriting the last identity in terms of the sG coefficients in vector form.
\end{proof}
{Thank to Proposition \ref{prop4} we may observe that the micro-macro decomposition does not require additional information to compute the evolution of $f_M(\v,\z,t)$ being the quasi-equilibrium obtained from $f_M$ itself. Furthermore,  the} advantage of using equation \eqref{eq:MM_ref} with respect to equation \eqref{eq:sG_system} is clear since when $f_M=f^q_M$ the right-hand side of \eqref{eq:MM_ref} vanishes. This gives to the resulting scheme a well-balanced property, in the sense that is capable to capture exactly the steady states of the system. Note that, the micro-macro decomposition written in the form \eqref{eq:MM_ref} is usually referred to as residual equilibrium formulation (see^^>\cite{Pareschi2017}). 

\begin{rem}
Finally, we observe that the practical evaluation of the quasi-equilibrium state \eqref{eq:quasi_eq} requires a suitable quadrature formula. The accuracy of the formula will be directly transferred to the accuracy of the numerical scheme in physical space for long times, regardless of the accuracy adopted to discretize the differential terms in \eqref{eq:J_meanfield}. The choice of the quadrature formula may depend on the specific problem; a simple choice, in the case where the equilibrium state vanishes at the boundary, is to use the trapezoidal rule on equidistant mesh points, which leads to spectral accuracy for smooth solutions. This choice is adopted in the next section when the equilibrium state is not known a priori.
\end{rem}

 \section{Numerical results}\label{sec:numerics}
 
 In this section we present several numerical examples where the nonlinear Fokker-Planck equations will be solved using the micro-macro stochastic Galerkin (MMsG) scheme introduced in the previous section. In particular, we will highlight how the new scheme leads to very accurate descriptions of the large time behavior of the underlying system compared to a standard stochastic Galerkin formulation. We will take both into account the case of known equilibria as well as we will give numerical evidence of the effectiveness of the quasi-equilibrium formulation of the method in cases where the uncertain equilibrium state is unknown. In all test cases a simple second order method based on central difference approximations has been adopted for the discretization of the derivatives in the physical space of the corresponding Fokker-Planck model.
 
 \subsection{Test 1: The classical Fokker-Planck equation }\label{sec:FP_classic}
 
 In this section, with the aid of an analytical solution, thus ignoring velocity and time discretization errors, we investigate the impact of uncertain parameters of the large time behavior of a classical 1D Fokker-Planck model. In details, first we will take into account the problem with an uncertain relaxation rate complemented with deterministic initial distribution. These hypotheses will imply a deterministic equilibrium distribution $f^\infty(v)$. Next we consider the case of a stochastic initial distribution, giving rise to an uncertain equilibrium distribution $f^\infty(\z,v)$. In details, the accuracy of the gPC projection in the two above mentioned cases is very different. {This highlights the importance for the numerical methods to be able to capture the uncertain long time behavior of the equation together with the potential difficulties emerging in the degenerate case when the uncertainty vanishes with time.
 }
 
 \subsubsection*{Case (a). Uncertain relaxation rate}\label{subsect:11}
Let us consider a Fokker-Planck model characterized by an uncertain relaxation rate $K(\z)>0$ 
\begin{equation}
\label{eq:FP_T11}
\partial_t f(\z,v,t) = K(\z) \partial_v \left[ vf(\z,v,t) + \sigma \partial_v f(\z,v,t) \right], \qquad v\in \mathbb R, K(\z)>0, 
\end{equation}
and complemented with a deterministic initial distribution $f_0(v)$ We consider a deterministic initial distribution   
$$f_0(v) = \alpha v^2\exp\left\{-\beta x^2\right\}, \quad \alpha = \dfrac{3}{2\sigma}\sqrt{\dfrac{3}{2\sigma}},\quad \beta= \dfrac{3}{2\sigma},\quad \sigma >0.$$
For each $t\ge 0$ the exact solution of \eqref{eq:FP_T11} is given by  
\begin{equation}
\label{eq:exact_T1}
f(\z,v,t) = (A(\z,t) + B(\z,t)v^2)\exp\{-s(\z,t)v^2\},
\end{equation}
where $A(\z,t)$, $B(\z,t)$, and $s(\z,t)$ have been defined in \eqref{eq:AB} and \eqref{eq:sevo} (see Appendix \ref{appA}). Since for $t\rightarrow +\infty$ we have  
$$
s\rightarrow \frac{1}{2\sigma}, \quad A \rightarrow \frac{1}{\sqrt{2\pi\sigma}},\quad B\rightarrow 0,
$$ 
the large time distribution of  \eqref{eq:FP_T11}  is defined by the deterministic Maxwellian distribution
\[
f^\infty(v) = \dfrac{1}{\sqrt{2\pi\sigma}} \exp\left\{ - \dfrac{v^2}{2\sigma}\right\}. 
\]
In Figure \ref{fig:EV_T1} we represent the evolution of $\mathbb E[f(\z,v,t)]$ and $\textrm{Var}(f)(\z,v,t)$, being $f(\z,v,t)$ the exact solution of the problem  \eqref{eq:FP_T11}. We consider $K(z) = \frac{1}{2}(z+1)+\frac{1}{2}$, $z \sim \mathcal U([-1,1])$ and we compute the evolution of statistical quantities through a gPC expansion of the exact solution and $M = 5$, therefore
\[
\mathbb E[f(z,v,t)] \approx \hat f_0(v,t), \qquad \textrm{Var}(f(z,v,t)) \approx \sum_{k=0}^{M} \hat f_k^2(v,t)\Phi_k^2(z) - \hat f_0^2(v,t).
\]
We may observe how the variance vanishes for large times since $f^\infty(v)$ is deterministic. 

In Figure \ref{fig:err_T11} we show the evolution of $\textrm{Var}(\|f\|_{L^1(\mathbb R)})$ (left plot) for $M = 1,2,3$. Furthermore, in the second row we represent the $L^1$ relative error for the  variance computed with respect to a reference variance $\textrm{Var}^{\textrm{ref}}(f)$ of the analytical solution $f(z,v,t)$ obtained with $M = 50$. In details, we consider the following $L^1$ relative error 
\begin{equation}
\label{eq:L1}
\epsilon_{\textrm{var}}(t) =\dfrac{ \| \textrm{Var}(f)(v,t) - \textrm{Var}^{\textrm{ref}}(f)(v,t)\|_{L^1(\mathbb R)}}{\| \textrm{Var}^{\textrm{ref}}(f)(v,t)\|_{L^1(\mathbb R)}}. 
\end{equation}
{Since the variance is close to zero we also report in the same figure the absolute error.}

Note that, similarly to^^>\cite{GSVK}, the spectral accuracy of the gPC projection is lost for large times for the solution of \eqref{eq:FP_T11}. {In fact, in the degenerate case in which the equilibrium does not depend on the uncertainties, the variance in the random space vanishes asymptotically. Even if a numerical method captures well this deterministic asymptotic state, the vanishing asymptotic impact of uncertainties may indeed cause a degradation of accuracy in the random space in the transient regime where the variance is close to zero but not zero. A possible way to partially alleviate this problem is to construct a time dependent polynomial basis, as proposed in^^>\cite{GSVK}.}  

{Next, in Figure \ref{fig:err_T11_num} we show the evolution of the absolute $L^1$ error  obtained from the standard sG scheme and the micro-macro sG approach for the Fokker-Planck equation \eqref{eq:FP_T11}. We considered as reference value for the variance the one obtained from the analytical solution of the problem with $M = 50$. In both methods the derivatives in the physical space are solved by standard central differences. We may observe how, using a standard sG scheme, the error saturates in time whereas the micro-macro sG approach provides far more accurate results and for large times yields essentially the same error as the exact solution.   }

\begin{figure}
\centering
\includegraphics[scale = 0.35]{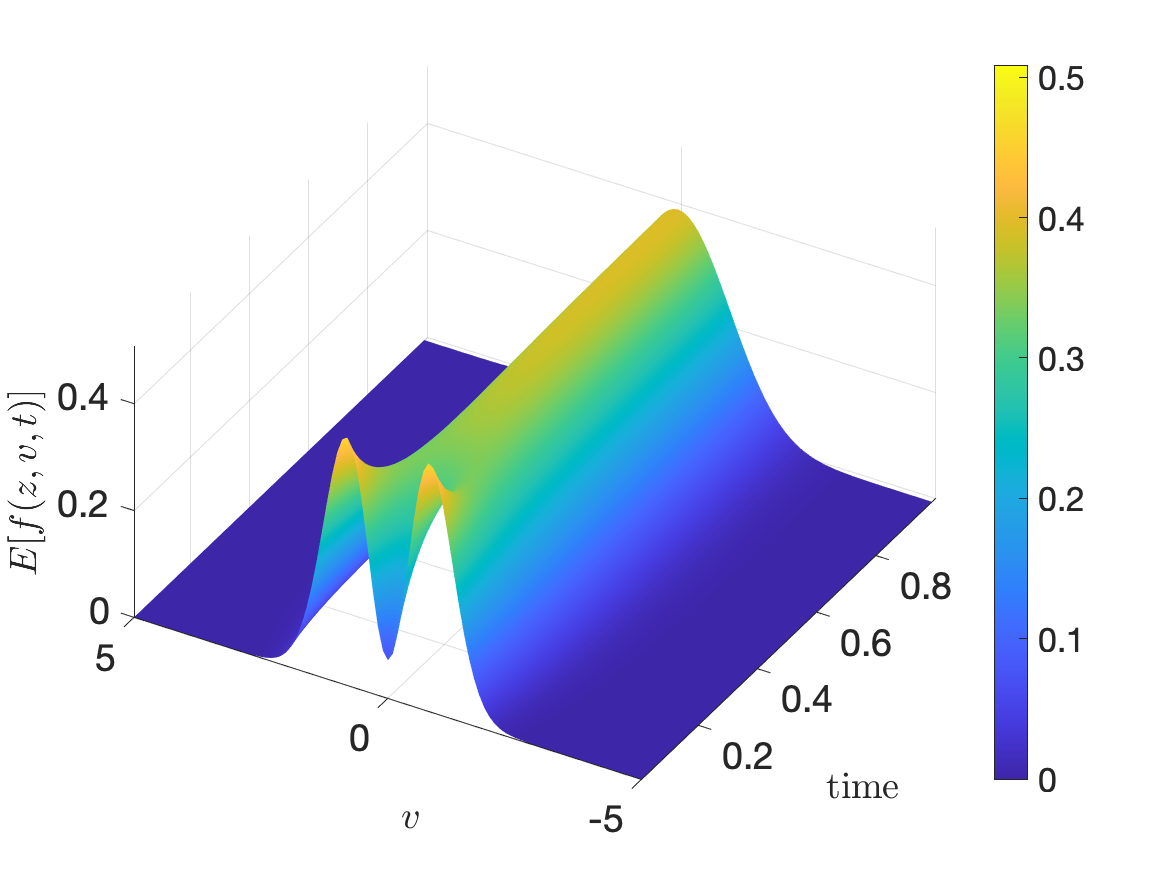}
\includegraphics[scale = 0.35]{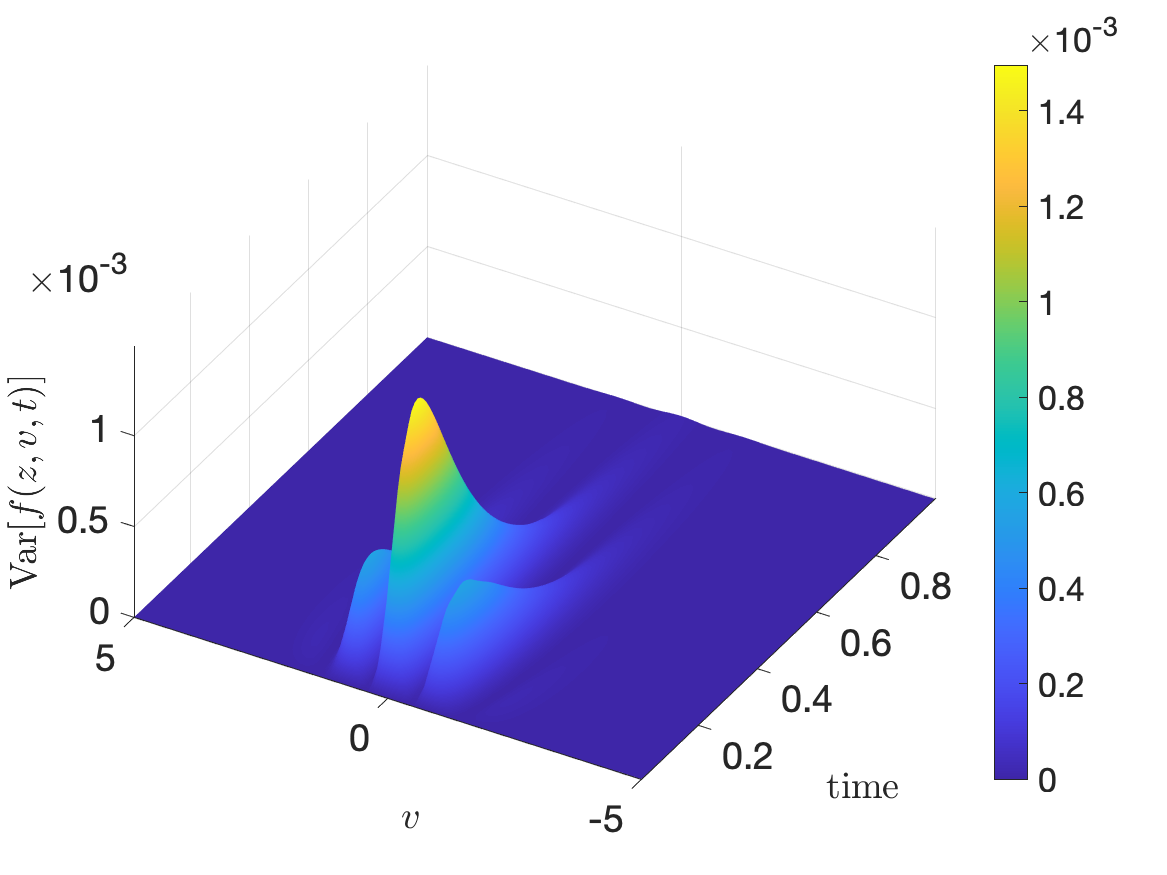}
\caption{\textbf{Test 1(a)}. Evolution of $\mathbb E[f(z,v,t)]$ and $\textrm{Var}[f(z,v,t)]$, with $f(z,v,t)$ the exact solution of \eqref{eq:FP_T11} where $K(z) = \frac{1}{2}(z+1)+\frac{1}{2}$, $z \sim \mathcal U([-1,1])$. The statistical quantities have been determined with a gPC expansion with $M = 5$. }
\label{fig:EV_T1}
\end{figure}

\begin{figure}
\centering
\includegraphics[scale = 0.35]{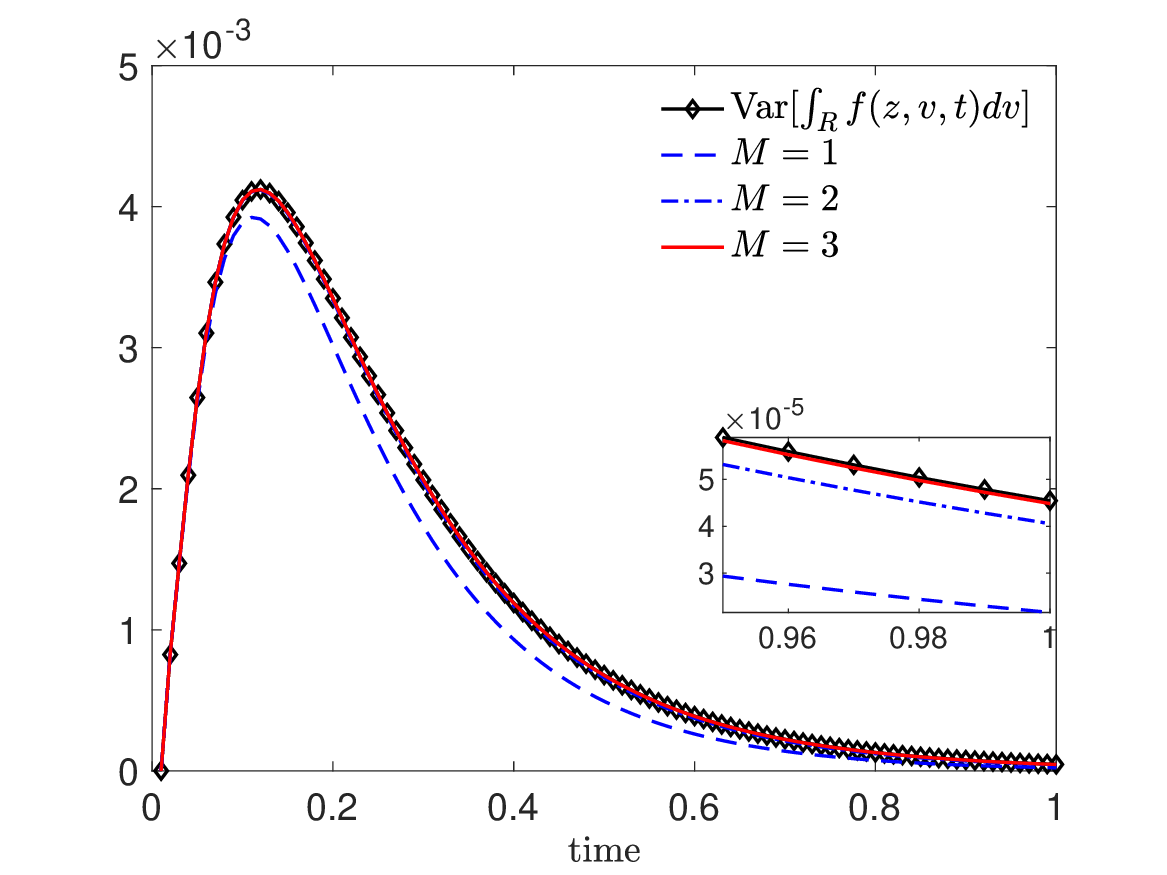}
\includegraphics[scale = 0.35]{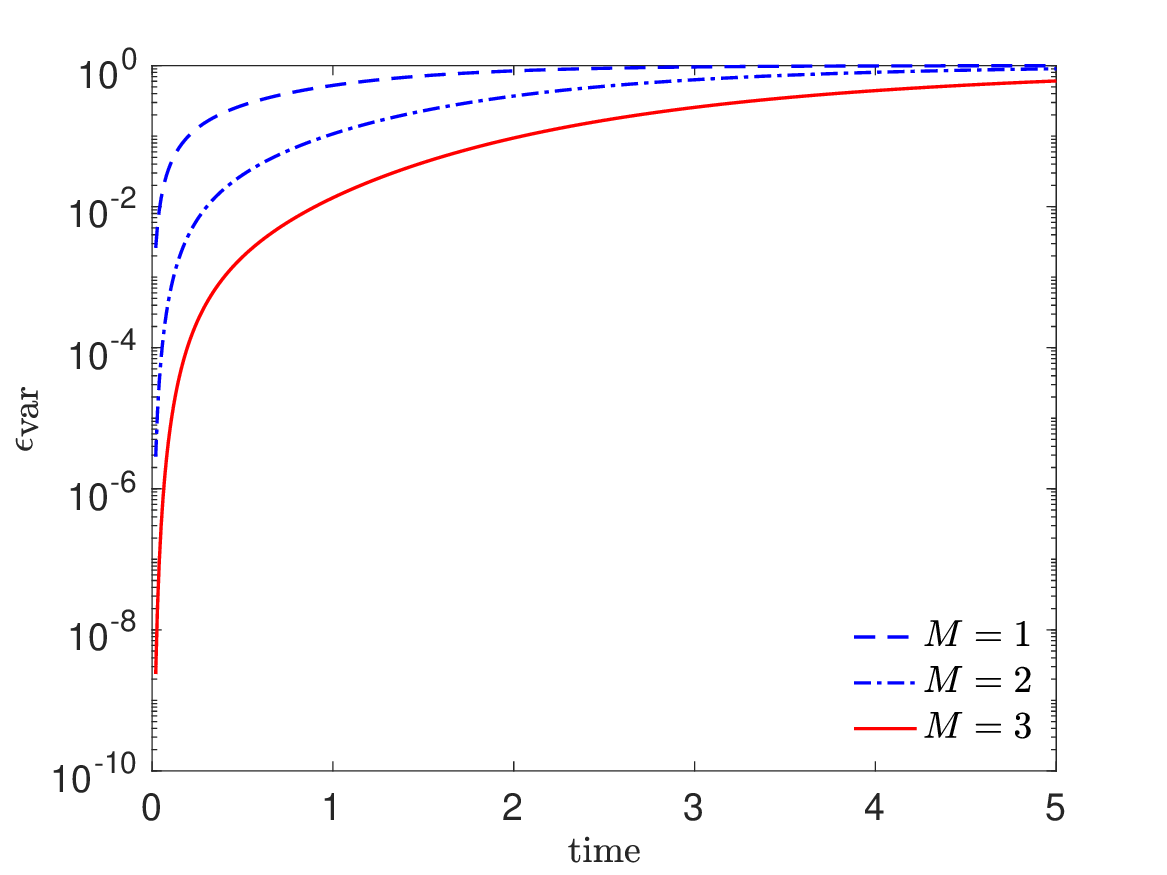}\\
\includegraphics[scale = 0.35]{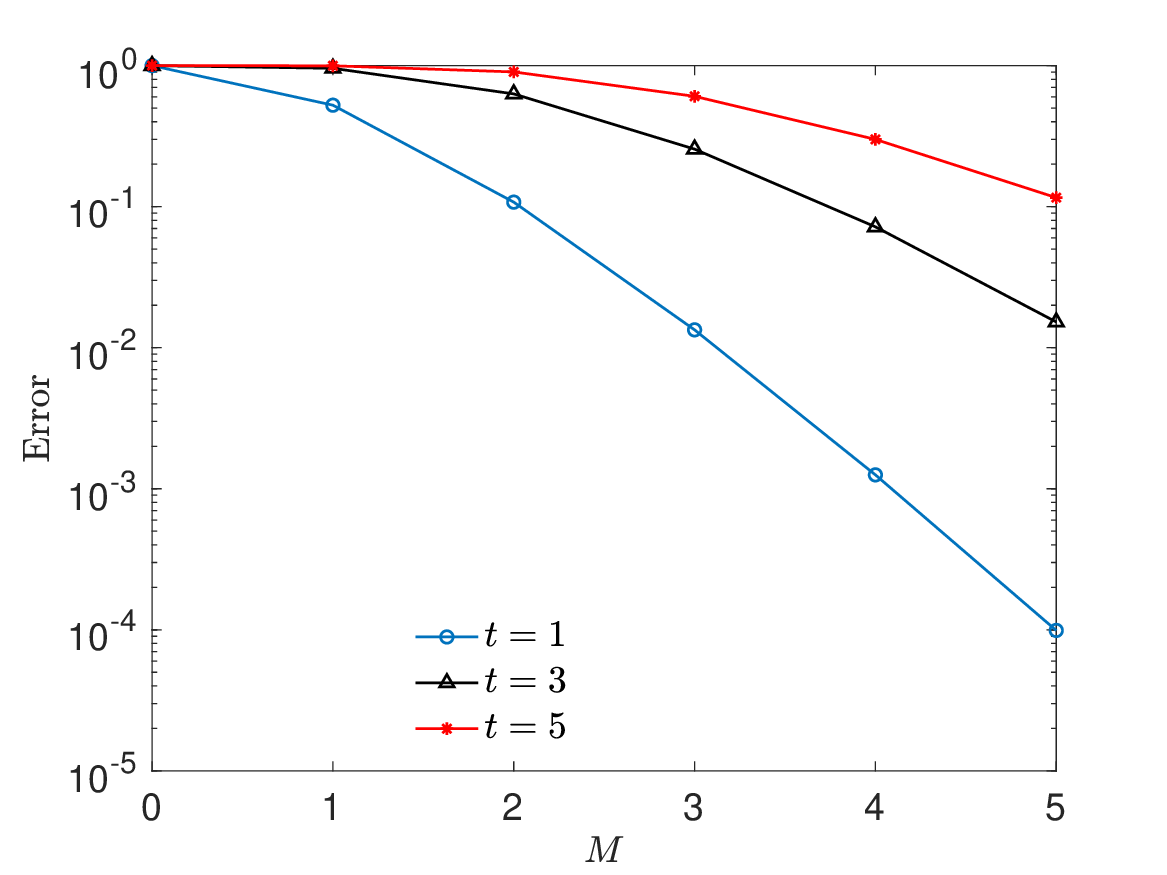}
\includegraphics[scale = 0.35]{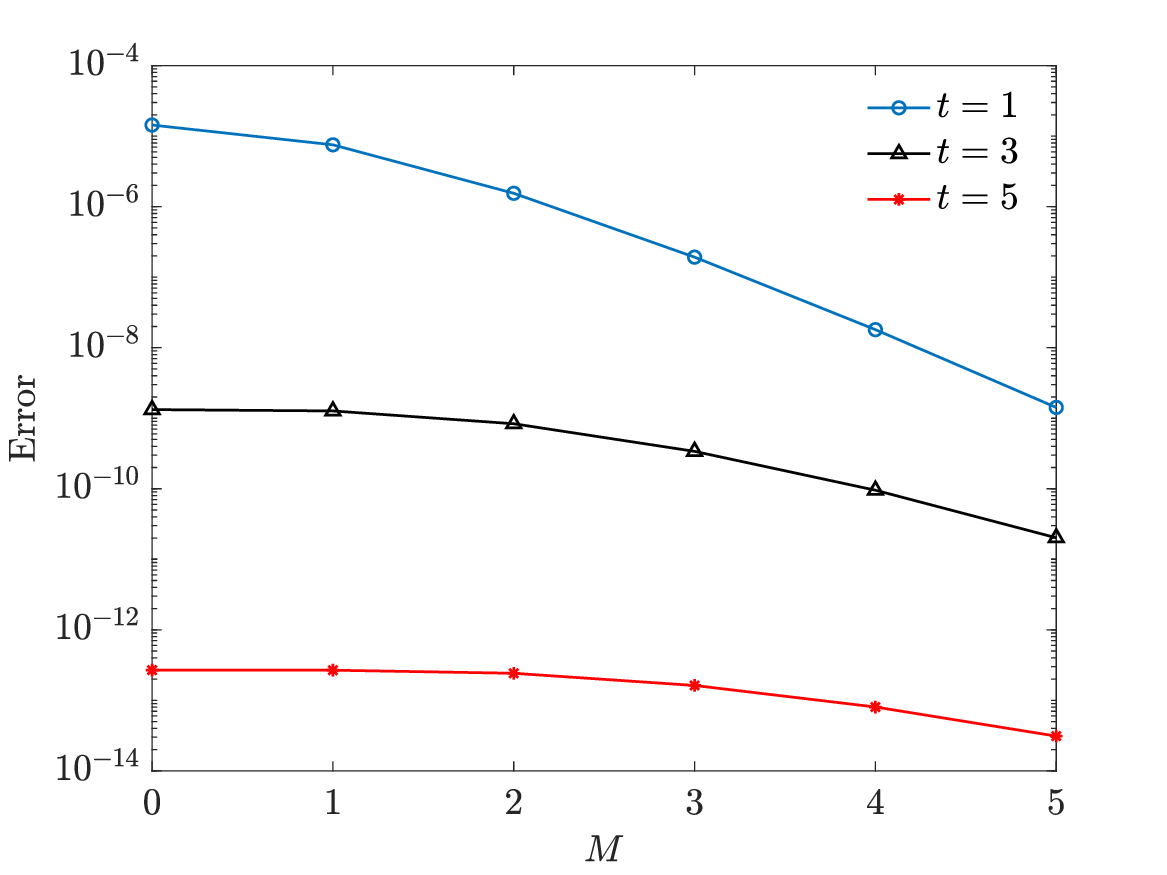}
\caption{\textbf{Test 1(a)}. {Top row: evolution of $\textrm{Var}[\|f(\z,v,t) \|_{L^1(\mathbb R)}]$ and of its gPC approximation for $M=1,2,3$ (left), evolution of the relative $L^1$ error $\epsilon_{\textrm{var}}$ defined in \eqref{eq:L1} for $M = 1,2,3$ (right). Bottom row: gPC expansion error at three fixed time horizons $t = 1$, $t = 3$ and $ t = 5$ and computed for several $M>0$, relative error (left), absolute error (right). }  }
\label{fig:err_T11}
\end{figure}

\begin{figure}
\centering
\includegraphics[scale = 0.35]{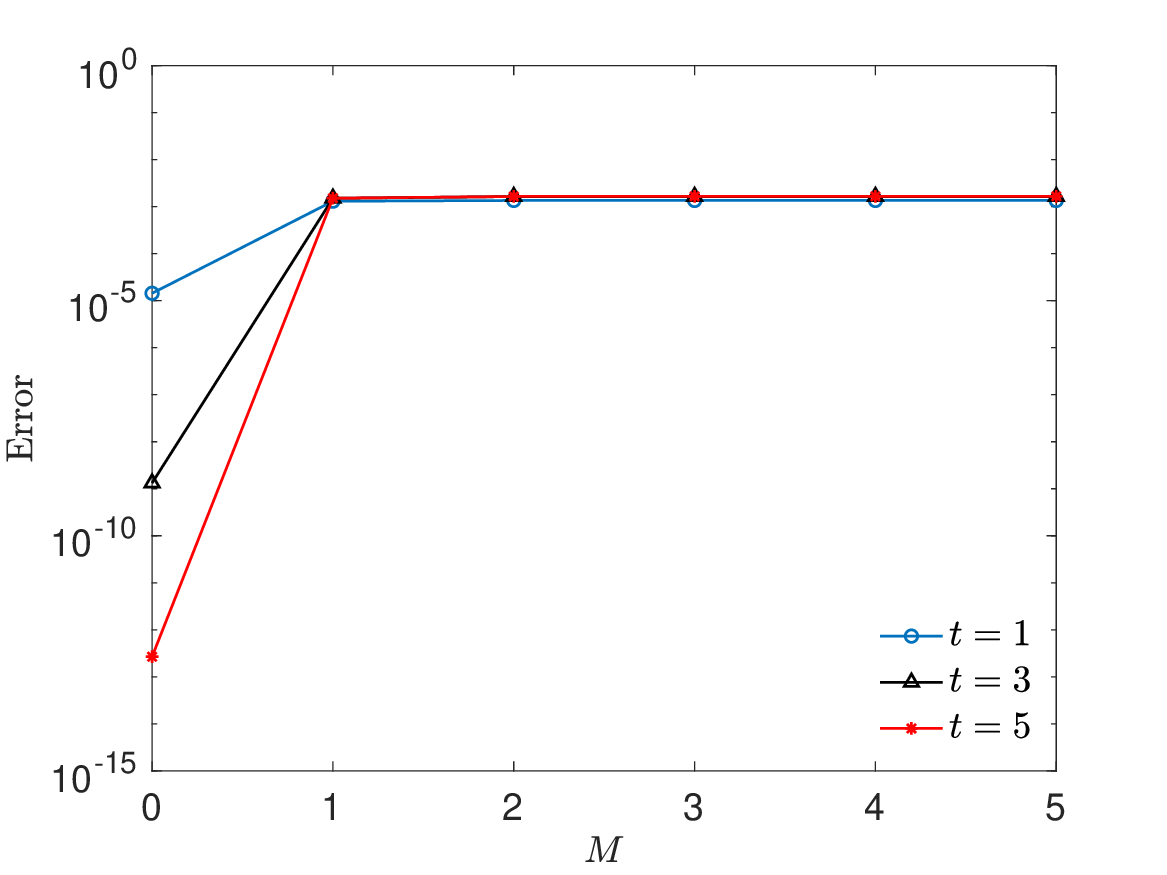}
\includegraphics[scale = 0.35]{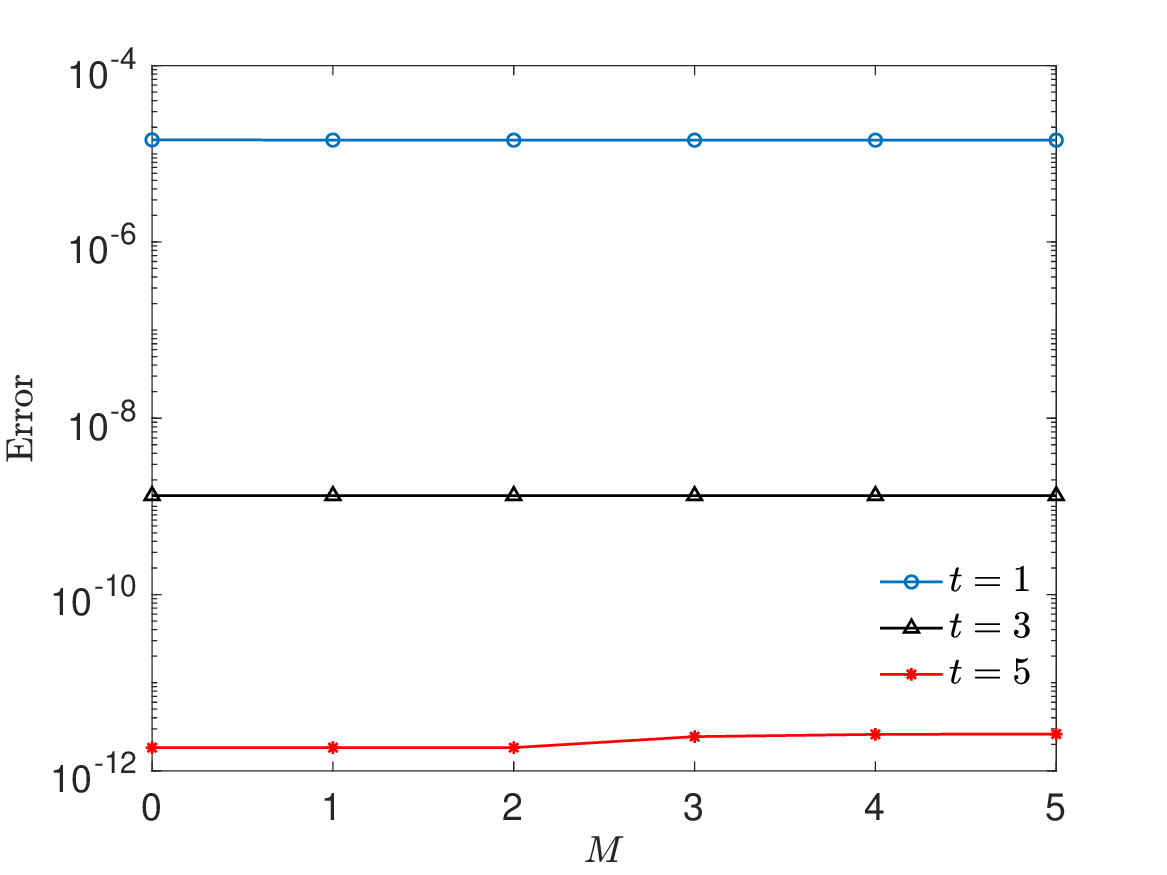}
\caption{{\textbf{Test 1(a)}. Evaluation of the absolute $L^1$ error for the uncertain Fokker-Planck equation \eqref{eq:FP_T11} obtained using standard sG (left) and micro-macro sG (right) using central differences at three different times  $t = 1$, $t = 3$ and $ t = 5$. We consider as reference solution the gPC expansion of the analytical one evaluated with $M = 50$. }}
\label{fig:err_T11_num}
\end{figure}

\subsubsection*{Case (b). Uncertain initial temperature}\label{subsect:12}
We consider a Fokker-Planck model with uncertain initial temperature $\sigma(\z)>0$
\begin{equation}
\label{eq:FP_T12}
\partial_t f(\z,v,t) = \partial_v \left[ vf(\z,v,t) + \sigma(\z)\partial_v f(\z,v,t)\right], \qquad v \in \mathbb R, \sigma(\z)>0
\end{equation}
and complemented with uncertain initial distribution $f_0(\z,v)$ such that $\int_{\mathbb R}v^2 f_0(\z,v) = \sigma(\z)$. Proceeding like in Appendix \ref{appA} we obtain the exact solution defined in \eqref{eq:exact_T1}. It is easy to observe that for $t \rightarrow +\infty$ the equilibrium distribution depends on the uncertainties of the system
\[
f^\infty(\z,v) = \dfrac{1}{\sqrt{2\pi\sigma(\z)}}\exp\left\{ -\dfrac{v^2}{2\sigma(\z)}\right\}.
\]

We consider $\sigma(z) =  \frac{1}{2}(z+1)+\frac{1}{2}$, where $z \sim \mathcal U([-1,1])$. In Figure \ref{fig:EV_T12}  we represent the evolution of $\mathbb E[f(\z,v,t)]$ and $\textrm{Var}(f)(\z,v,t)$, being $f(\z,v,t)$ the exact solution of the problem  \eqref{eq:FP_T12}. The statistical quantities have been approximated through a gPC expansion with $M = 5$. Unlike the test in Section \ref{subsect:11} the variance does not vanishes since the large time distribution depends on the uncertainties of the system.

In Figure \ref{fig:err_T12} we show the evolution of $\textrm{Var}(\| f\|_{L^1(\mathbb R)})$ for $M = 1,2,3$ (left plot) and the evolution of $\epsilon_{\textrm{var}}$ defined in \eqref{eq:L1}. Therefore, for an increasing $M>0$ the gPC expansion keeps a spectral accuracy in the approximation of statistical quantities of $f(z,v,t)$ also for large times. 

{Subsequently, in Figure \ref{fig:err_T12_num} we report the evolution of the relative $L^1$ error  obtained from the standard sG scheme and the micro-macro sG approach for the Fokker-Planck equation \eqref{eq:FP_T12}. We considered as reference value for the variance the one obtained from the analytical solution of the problem with $M = 50$. In both methods the derivatives in the physical space are solved by standard central differences. We may observe how, using a standard sG scheme, the error saturates in time whereas the micro-macro sG approach yields spectral accuracy for large times.   }

\begin{figure}
\centering
\includegraphics[scale = 0.35]{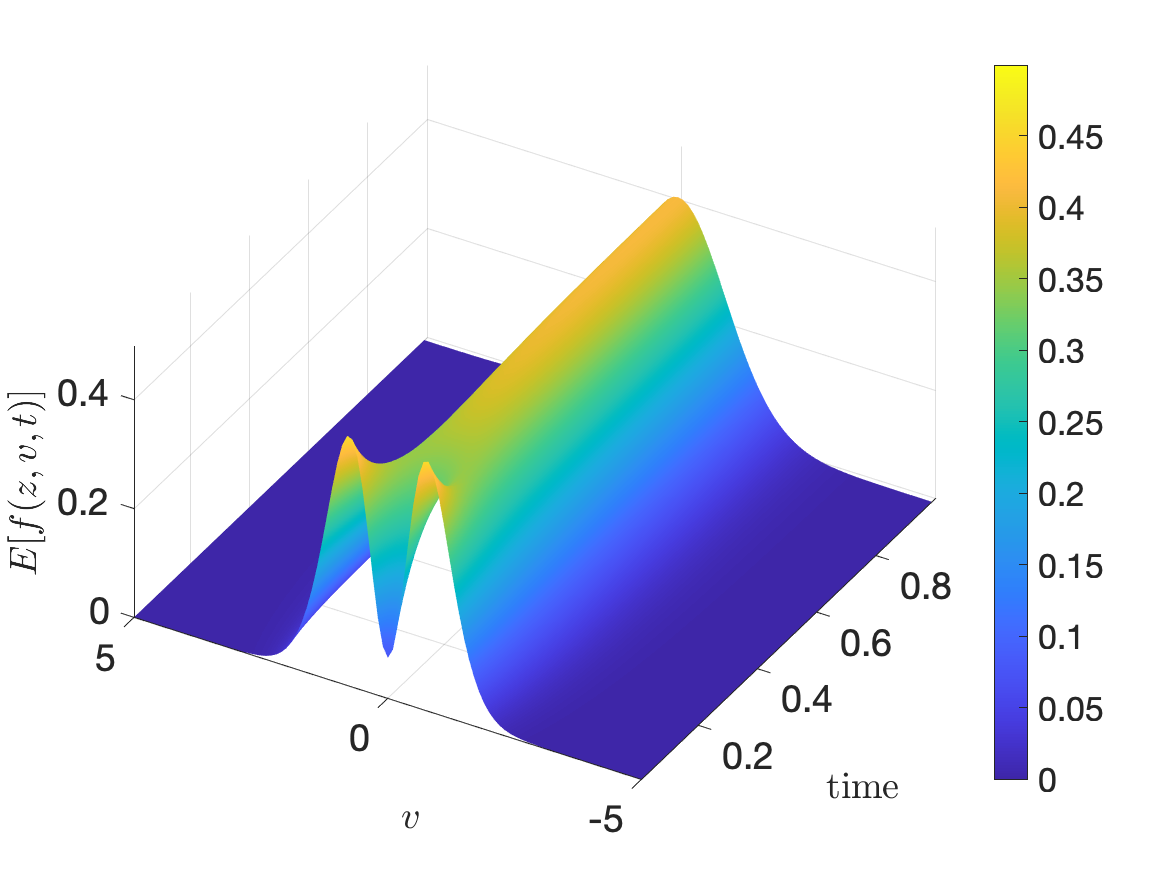}
\includegraphics[scale = 0.35]{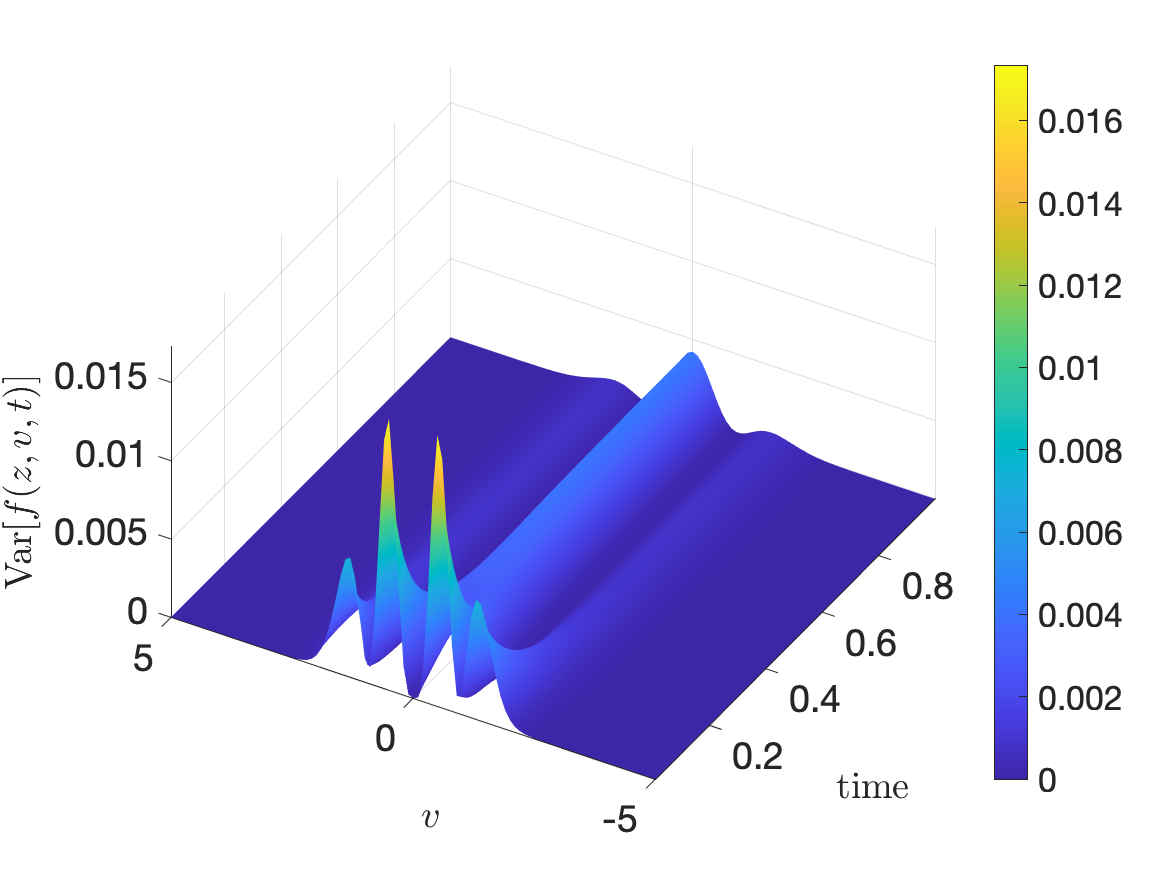}
\caption{\textbf{Test 1(b)}. Evolution of $\mathbb E_{\z}[f(\z,v,t)]$ and $\textrm{Var}[f(\z,v,t)]$, $f(\z,v,t)$ is exact solution of \eqref{eq:FP_T12}. The statistical quantities have been determined with a gPC expansion with $M = 5$. }
\label{fig:EV_T12}
\end{figure}

\begin{figure}
\centering
\includegraphics[scale = 0.235]{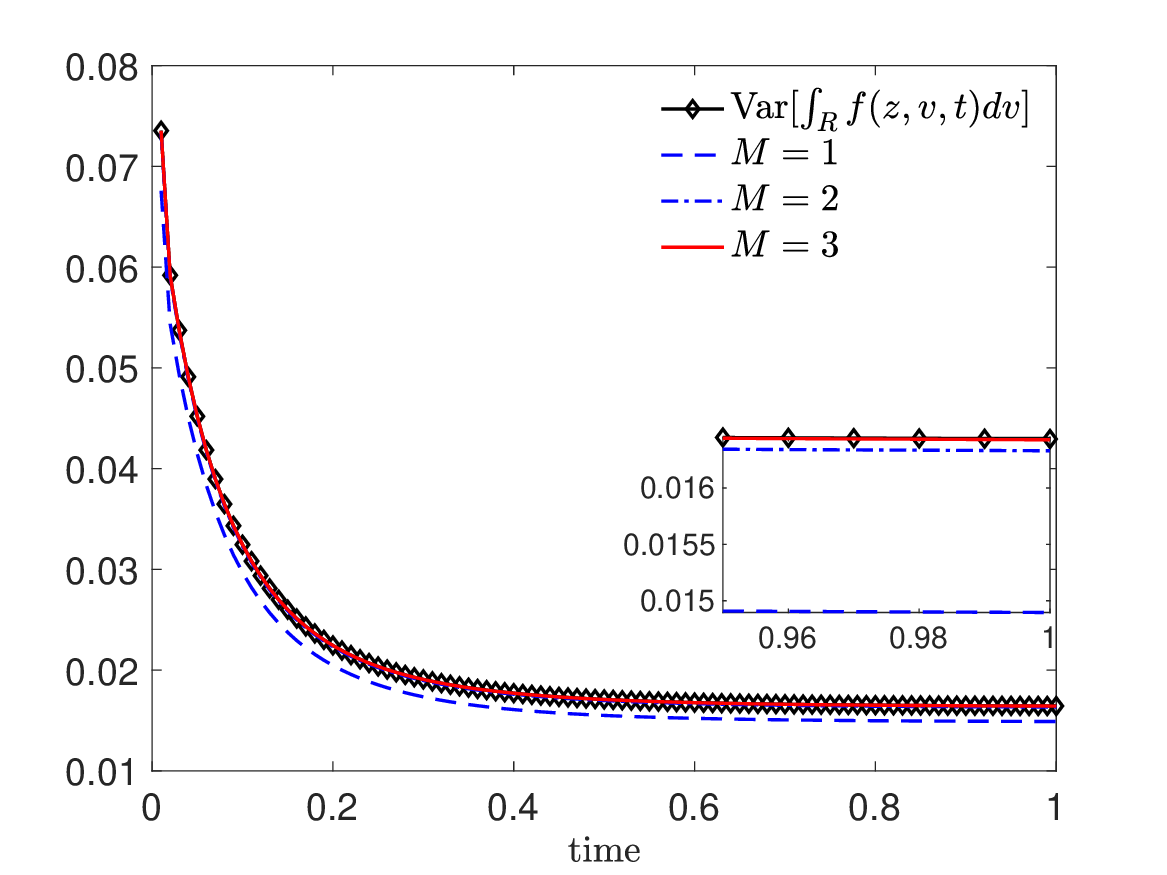}
\includegraphics[scale = 0.235]{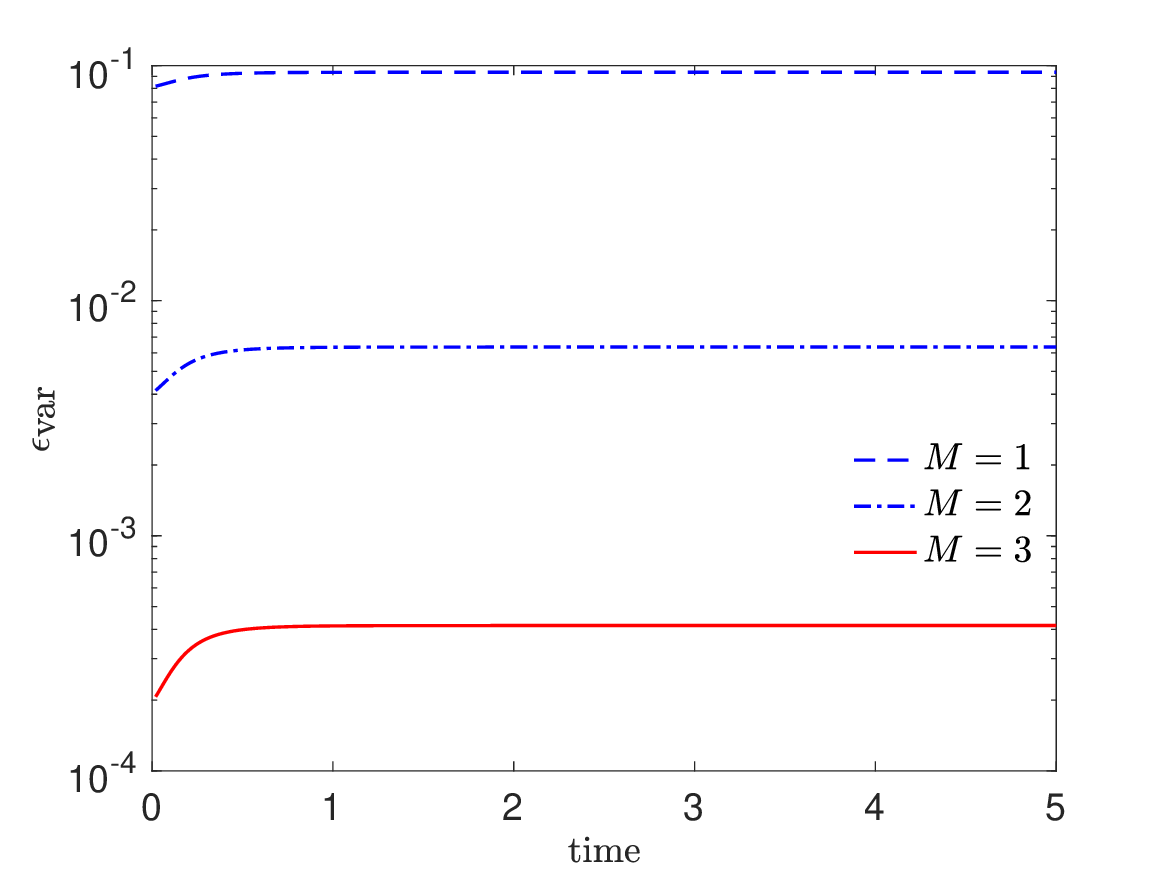}
\includegraphics[scale = 0.235]{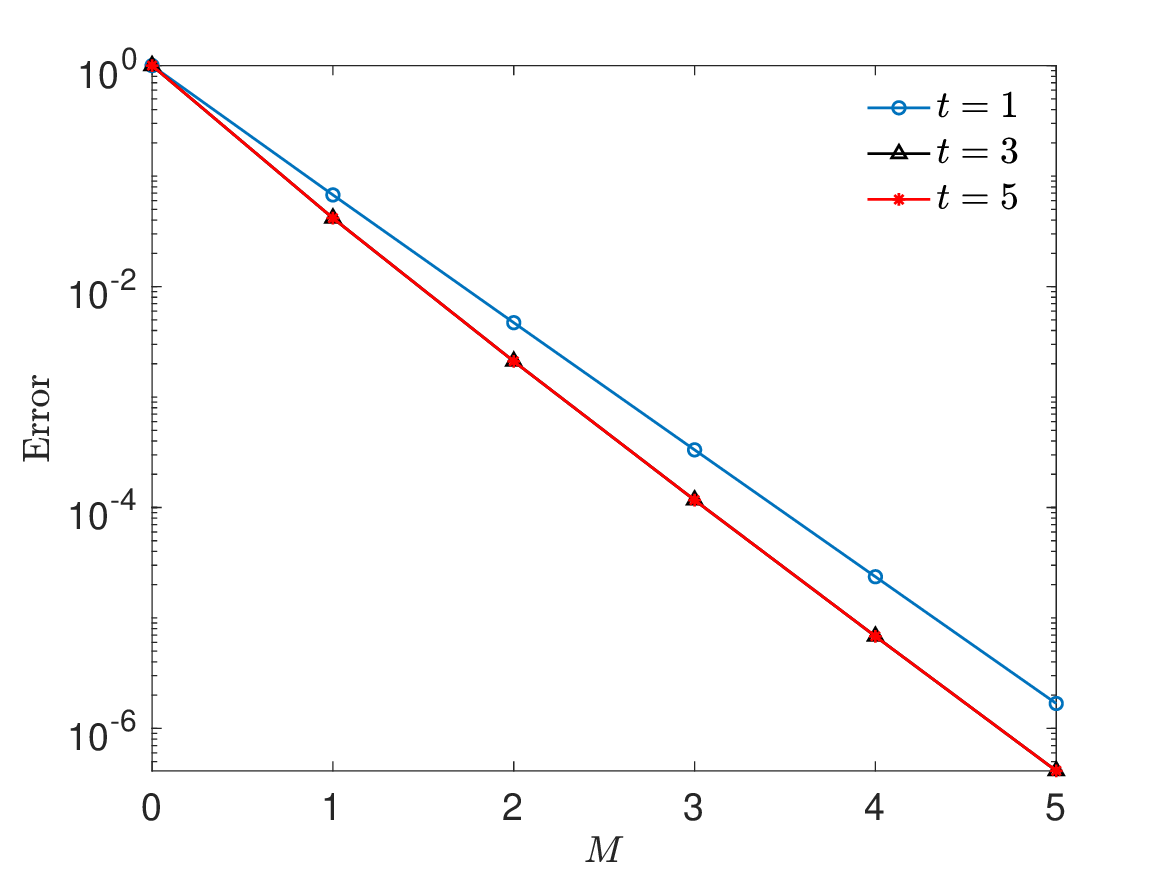}
\caption{\textbf{Test 1(b)}. {Evolution of $\textrm{Var}[\|f(\z,v,t) \|_{L^1(\mathbb R)}]$ and of its gPC approximation for $M=1,2,3$ (left), evolution of the relative $L^1$ error $\epsilon_{\textrm{var}}$ defined in \eqref{eq:L1} for $M = 1,2,3$ (center), gPC expansion error $\epsilon_{\textrm{var}}$ at three fixed time horizons $t = 1$, $t = 3$ and $ t = 5$ and computed for several $M>0$ (right).} }
\label{fig:err_T12}
\end{figure}

\begin{figure}
\centering
\includegraphics[scale = 0.35]{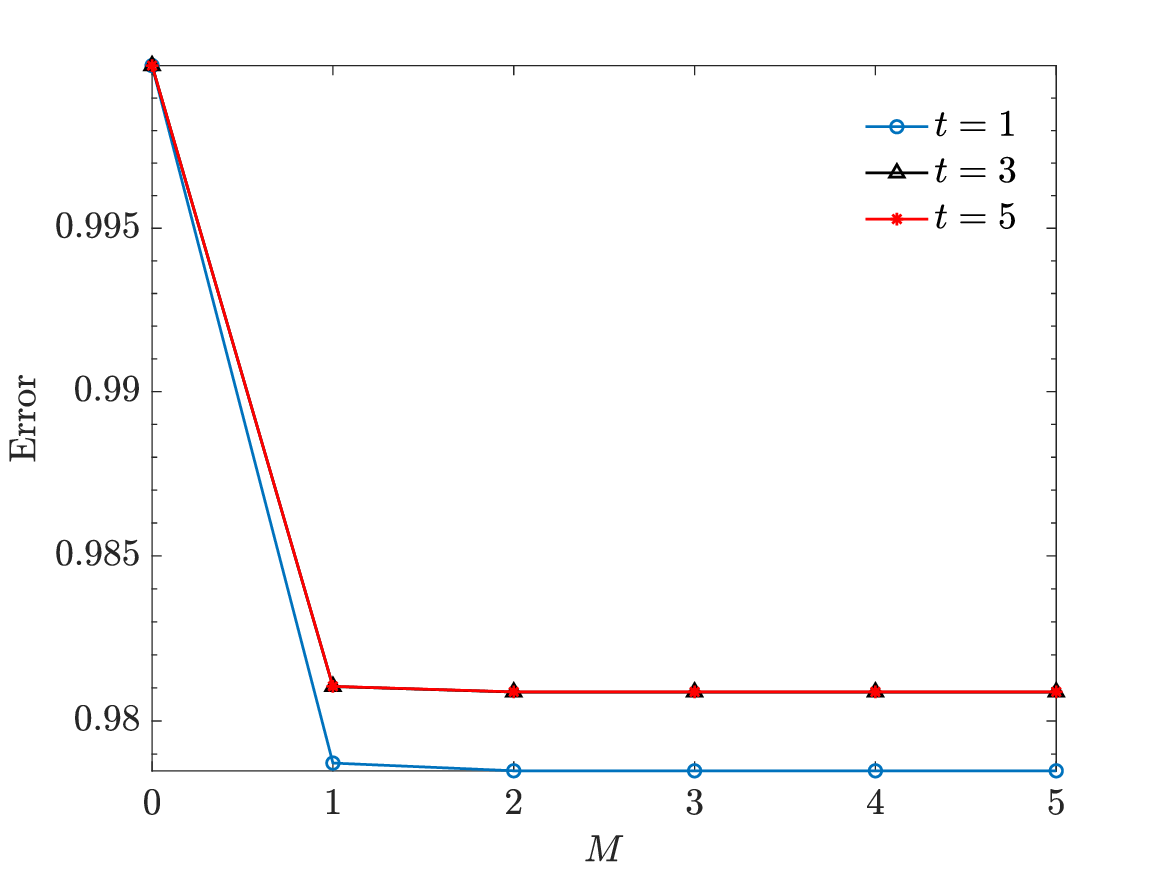}
\includegraphics[scale = 0.35]{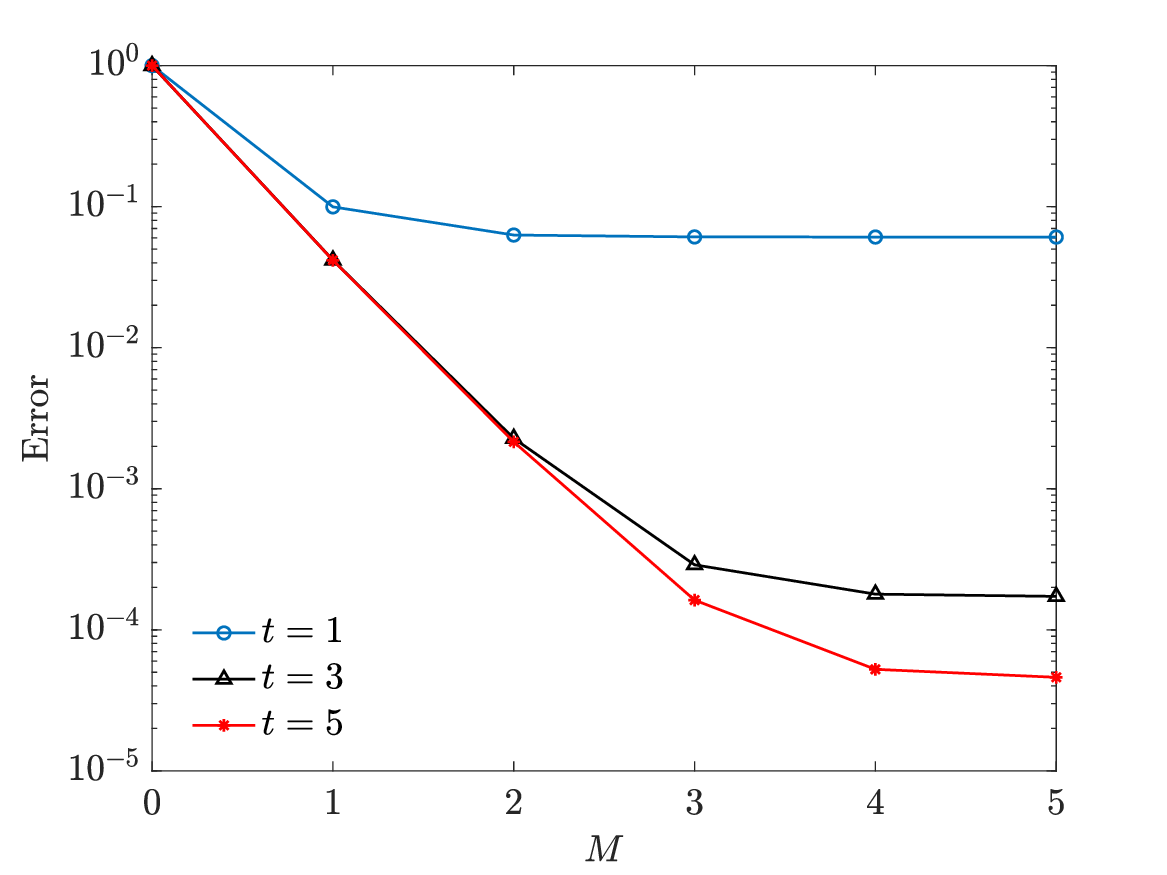}
\caption{{ \textbf{Test 1(b)}. Evaluation of the relative $L^1$ error for the uncertain Fokker-Planck equation \eqref{eq:FP_T12} obtained using standard sG (left) and micro-macro sG (right) using central differences at three different times  $t = 1$, $t = 3$ and $ t = 5$. We consider as reference solution the gPC expansion of the analytical one evaluated with $M = 50$. }}
\label{fig:err_T12_num}
\end{figure}

{
\subsubsection*{Case (c).  2D uncertainty}\label{subsect:multiD}
In this section we test the performance of the introduced MMsG in the case of a 2D uncertainty $\z = (z_1,z_2)$, $z_1 \in I_1 \subset \mathbb R$, $z_2 \in I_2 \subset \mathbb R$, with independent components such that  $z_1\sim p_1(z_1)$ and $z_2\sim p_2(z_2)$. If $\{\Phi_h\}_{h=0}^{M_1}$ and $\{\Psi_r\}_{r=0}^{M_2}$ are the families of orthonormal polynomials of $z_1$ and $z_2$, respectively, with degree up to $M_1\ge 0$ and $M_2\ge 0$. For any $f$ sufficiently regular we consider its sG approximation 
\[
f_{M_1,M_2}(\z,v,t) = \sum_{h=0}^{M_1}\sum_{r=0}^{M_2}\hat{f}_{hr}(\v,t)\Phi_h(z_1)\Psi_r(z_2), 
\]
where $\hat{f}_{hr}(v,t) = \int_{I_1\times I_2} f(\z,v,t)\Phi_h(z_1)\Psi_r(z_2)p_1(z_1)p_2(z_2)dz_1\,dz_2$. 
We consider then a Fokker-Planck model characterized by  uncertain initial temperature $\sigma(z_1)>0$ and uncertain relaxation rate $K(z_2)>0$ given by 
\begin{equation}
\label{FP2D}
\partial_t f(\z,v,t) = K(z_2) \partial_v\left[ vf(\z,v,t) + \sigma(z_1)\partial_v f(\z,v,t)\right], \qquad v \in \mathbb R,
\end{equation}
and complemented by uncertain initial distribution $f_0(\z,v) = f_0(z_1,v)$ such that $\int_{\mathbb R}v^2 f_0(\z,v)dv = \sigma(z_1)$. In the following, we will consider $K(z_2)$ as in Case (a) and $\sigma(z_1)$ as in Case (b). The equilibrium distribution of model \eqref{FP2D} is 
\[
f^\infty(z_1,v,t) = \dfrac{1}{\sqrt{2\pi\sigma(z_1)}}\exp\left\{-\dfrac{|v|^2}{2\sigma(z_1)}\right\}.
\]
The exact time evolution of the introduced problem can be obtained as in Appendix \ref{appA}. The gPC expansion of model \eqref{FP2D} is given by 
\[
\partial_t \hat{f}_{hr}(v,t) = \partial_v \left[ \sum_{k=0}^{M_1}\sum_{\ell = 0}^{M_2}\mathcal K_{kh\ell r}\hat{f}_{k\ell}(v,t) + \sum_{k=0}^{M_1}\sum_{\ell = 0}^{M_2}\mathcal S_{kh\ell r}\partial_v f_{k\ell}(v,t)   \right],
\] 
being $$\mathcal K_{kh\ell r} = \int_{I_1\times I_2} K(z_2)\Phi_k(z_1)\Phi_h(z_1)\Psi_\ell(z_2) \Psi_r(z_2)p_1(z_1)p_2(z_2)dz_1\,dz_2,$$ and $$\mathcal S_{kh\ell r} = \int_{I_1\times I_2} K(z_2)\sigma(z_1)\Phi_k(z_1)\Phi_h(z_1)\Psi_\ell(z_2) \Psi_r(z_2)p_1(z_1)p_2(z_2)dz_1\,dz_2.$$ 
In Figure \ref{fig:2D} we report the evolution of the $L^1$ error for the variance in the $\log_{10}$ scale  obtained from a standard sG scheme and the micro-macro sG approach for the Fokker-Planck equation \eqref{FP2D}. As a reference value for the variance we considered the one obtained from the analytical solution of the problem with $M_1 = M_2 = 25$. In both methods the derivatives in the physical space are solved through a central difference scheme and $N = 81$ gridpoints in the velocity space. As before, we may observe how the error saturates in time for a standard sG scheme whereas the micro-macro sG approach provides spectral accuracy for large times. 
\begin{figure}\centering
\includegraphics[scale = 0.35]{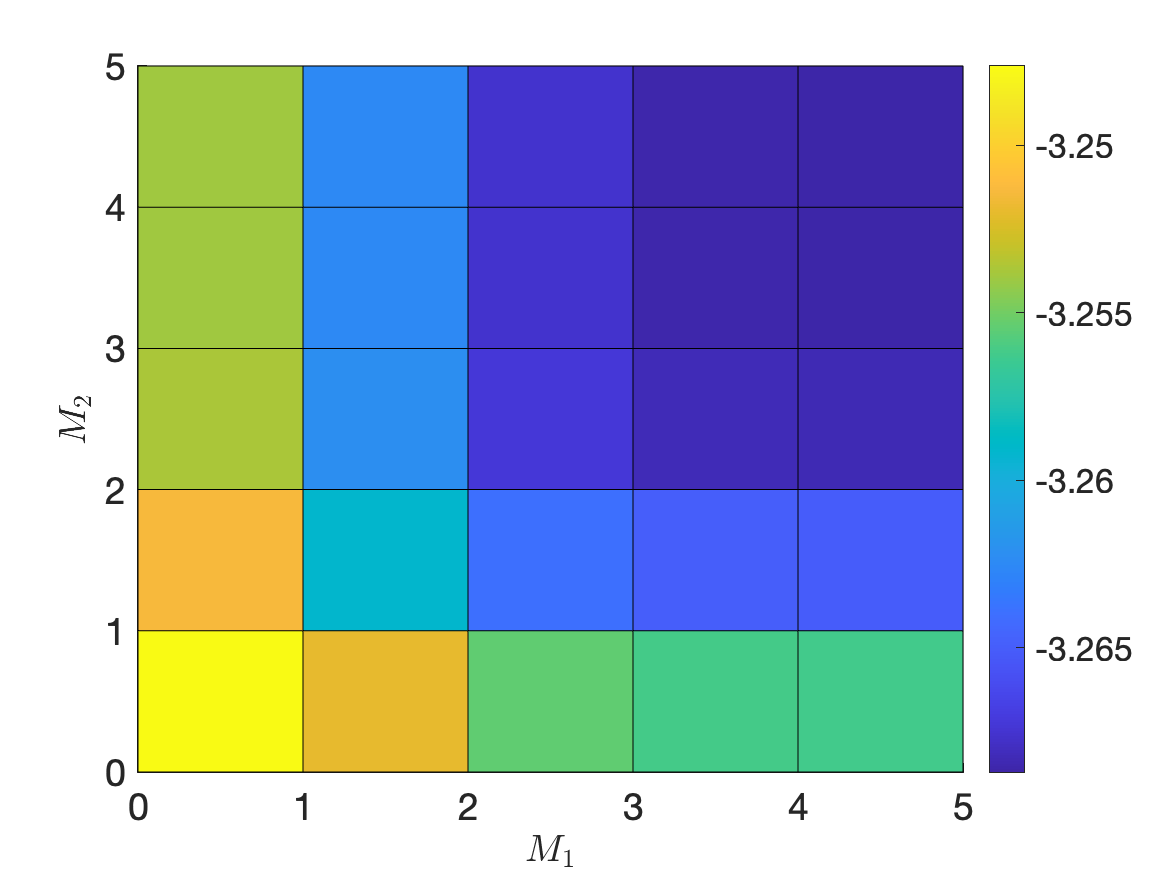}
\includegraphics[scale = 0.35]{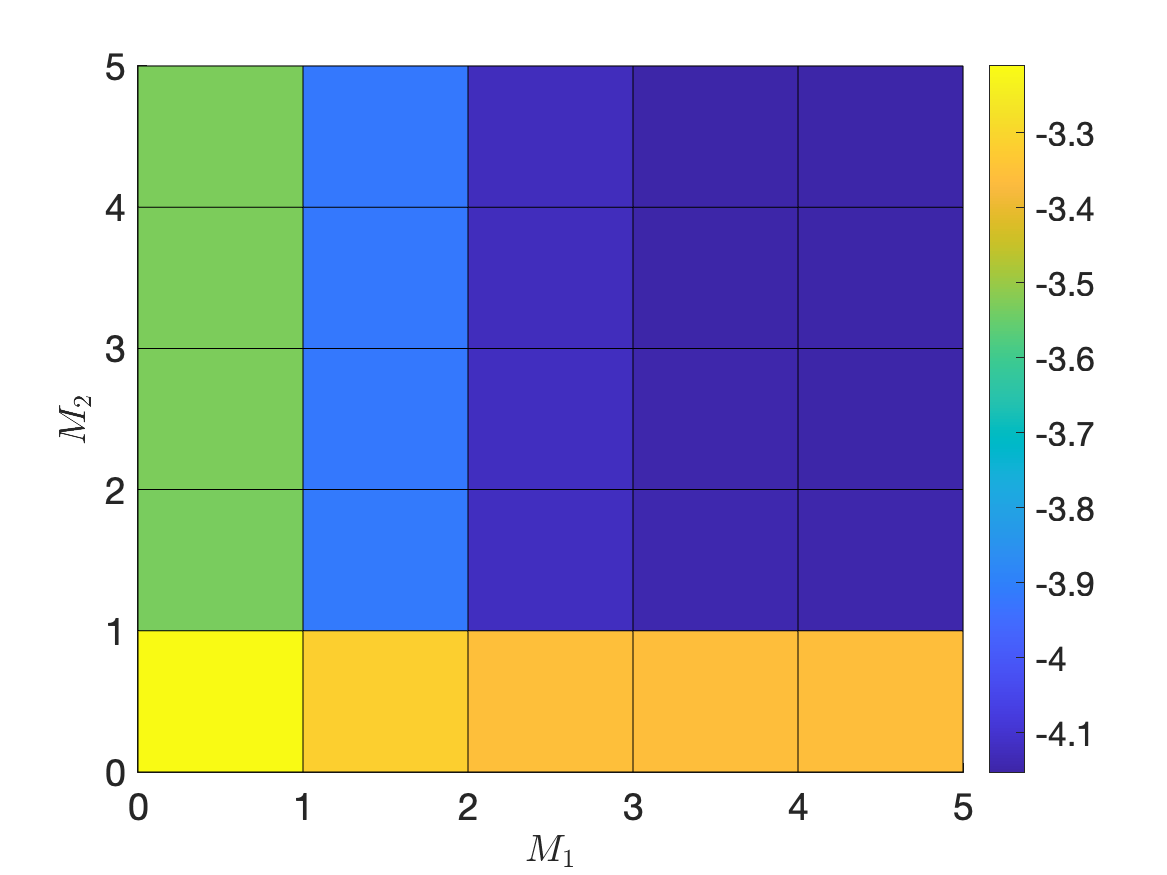}\\
\includegraphics[scale = 0.35]{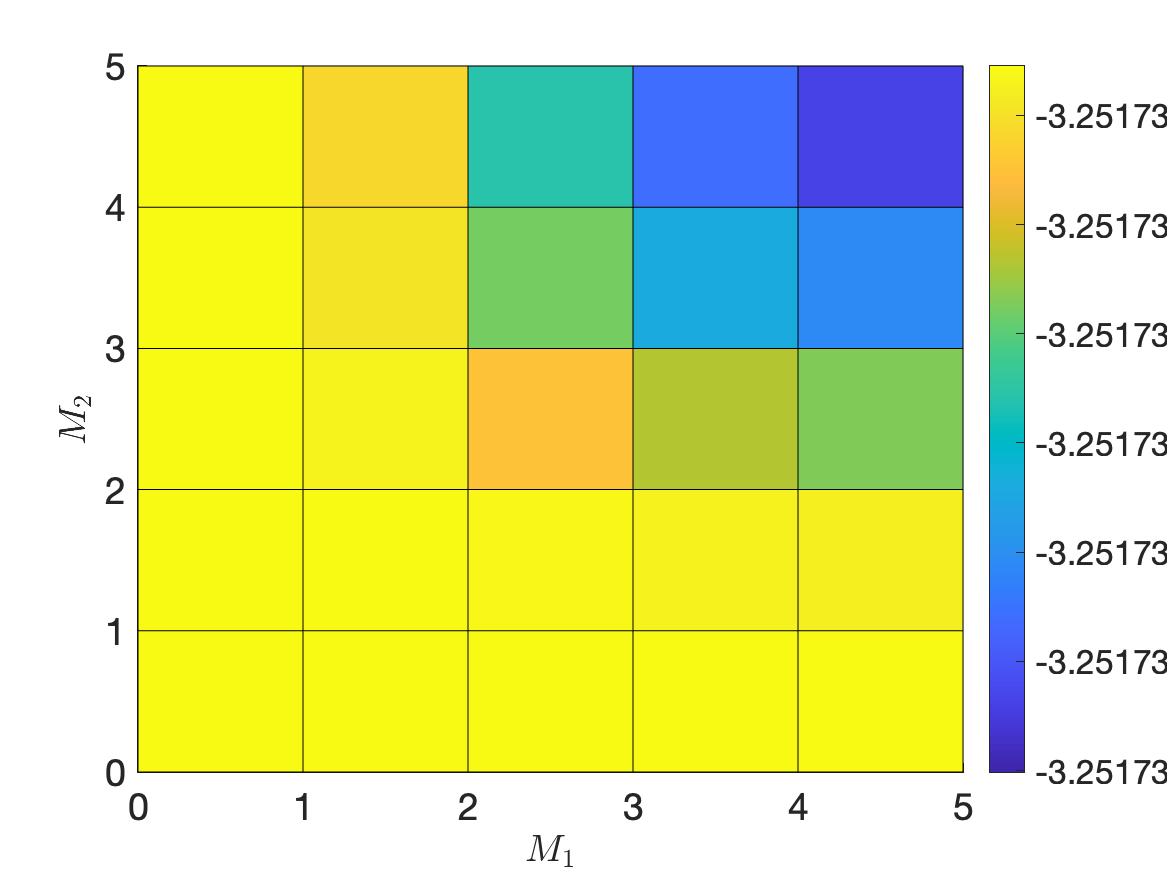}
\includegraphics[scale = 0.35]{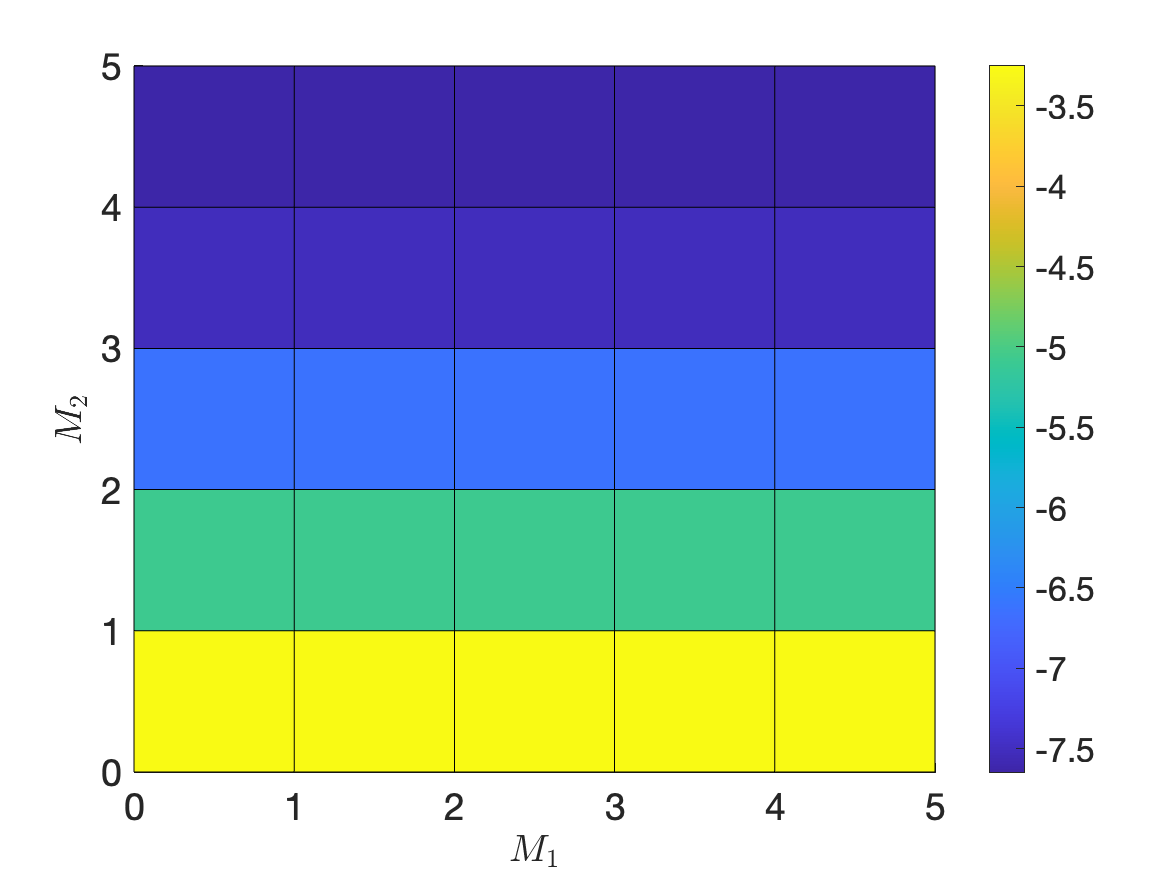}
\caption{{\textbf{Test 1(c)}. Evaluation of the $L^1$ error in $\log_{10}$ scale for the uncertain Fokker-Planck equation \eqref{FP2D} obtained using sG (left) and micro-macro sG (right) at time $t=1$ (top row) and $t = 10$ (bottom row). We used central difference for the derivates in the velocity space with $N = 81$ gridpoints. We consider as reference solution the gPC expansion of the analytical solution with $M_1 = M_2 = 25$.} }
\label{fig:2D}
\end{figure}
}

\subsection{Test 2: Explicit steady state in opinion dynamics}\label{sec:global_Max}

We consider a model describing the opinion formation through binary interactions already presented in Section \ref{sec:examples}. In this model we have that the evolution of the distribution function $f(z,v,t)$ is given by a Fokker-Planck equation with $v\in[-1,1]$ and defined by \eqref{eq:BD_opinion}. The explicit equilibrium state of the resulting Fokker-Planck equations is then given by \eqref{eq:exact_opinion}.

We consider an initial state of the form
\begin{equation}
\label{eq:f0_test1}
f_0(v) = C_0 \left( \exp \left\{ -\dfrac{(v-\mu_0)^2}{\sigma_0^2} \right\} +\exp\left\{ -\dfrac{(v+\mu_0)^2}{\sigma_0^2} \right\}\right),\qquad \mu_0 = \frac{1}{2},\,\sigma_0^2 = \dfrac{1}{20},
\end{equation}
with $C_0>0$ a normalization constant. In this case, the average opinion $u = \int_{-1}^1 vf(z,v,t)dv$ is conserved in time since no randomness is present on the initial distribution $f_0(v)$.

\begin{figure}
\centering
\includegraphics[scale=0.35]{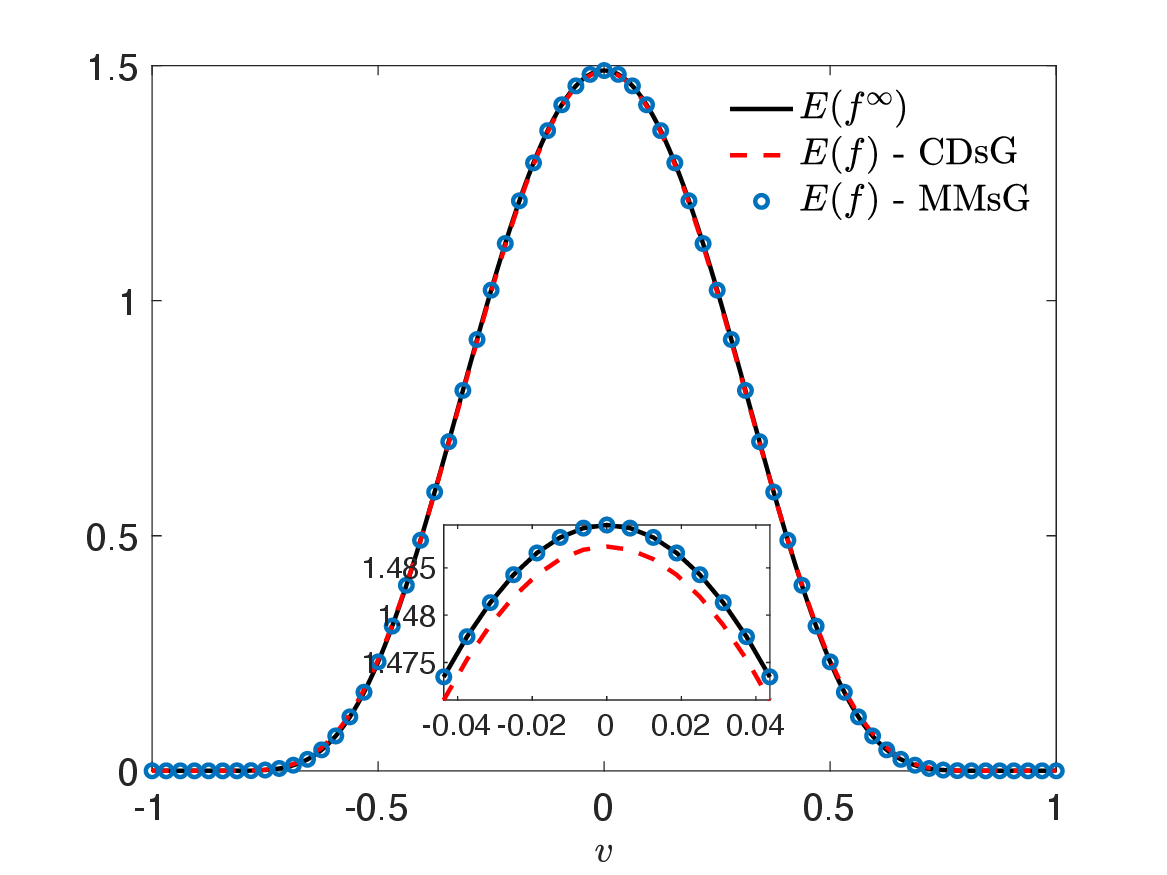}
\includegraphics[scale=0.35]{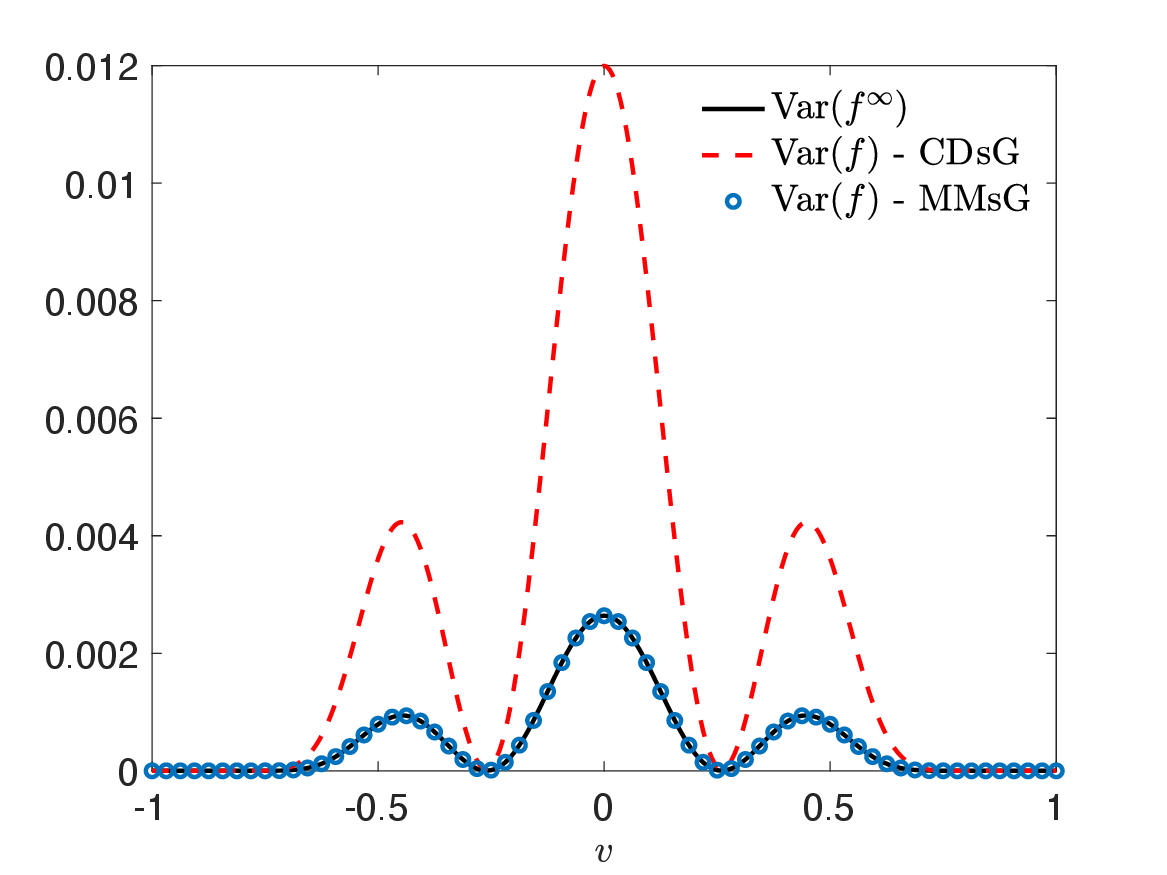}
\caption{ \textbf{Test 2}. Expected equilibrium (left) and its variance (right) for the central difference sG (CDsG) scheme and the micro macro sG scheme (MMsG) using $N = 41$ gridpoints and $M = 5$ projections for both methods.  {The numerical solution has been computed over the time interval $[0,T]$, $T = 15$.}}
\label{fig:op0}
\end{figure}

\begin{figure}
\centering
\includegraphics[scale=0.35]{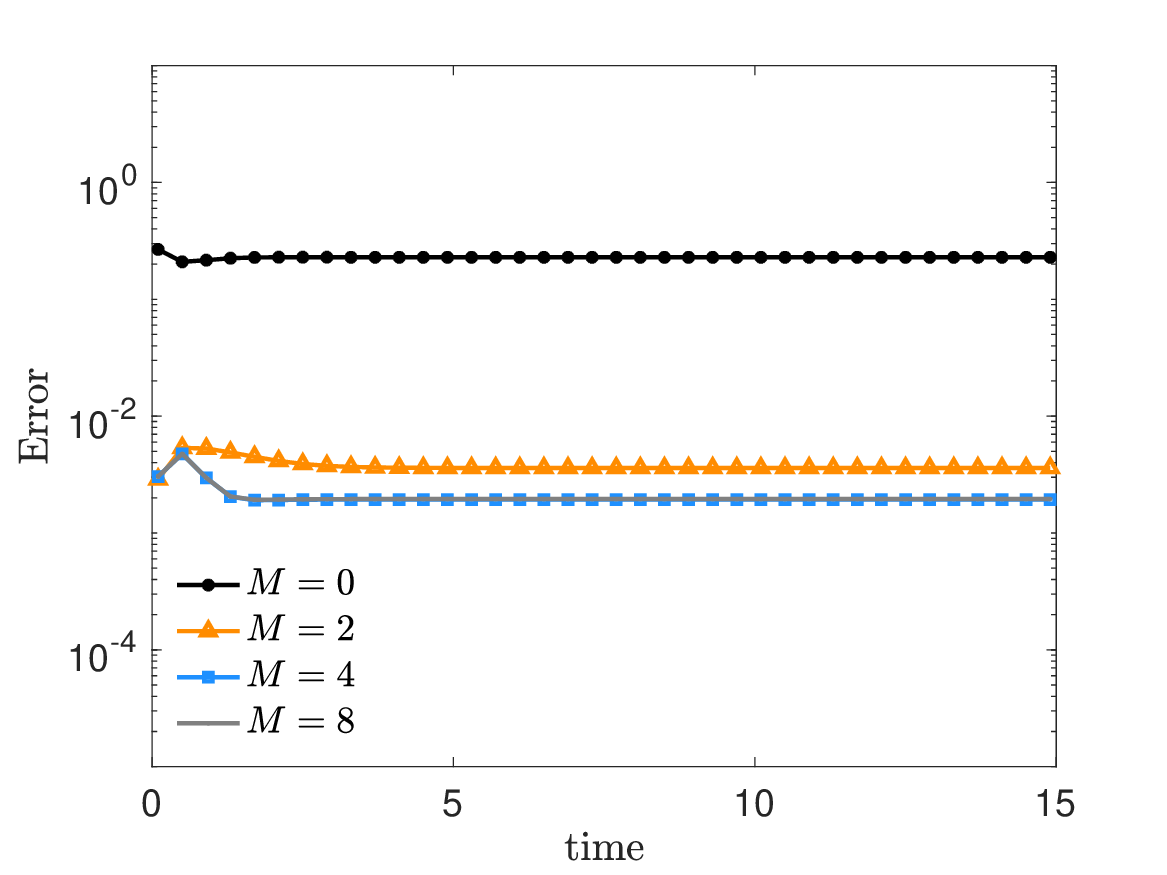}
\includegraphics[scale=0.35]{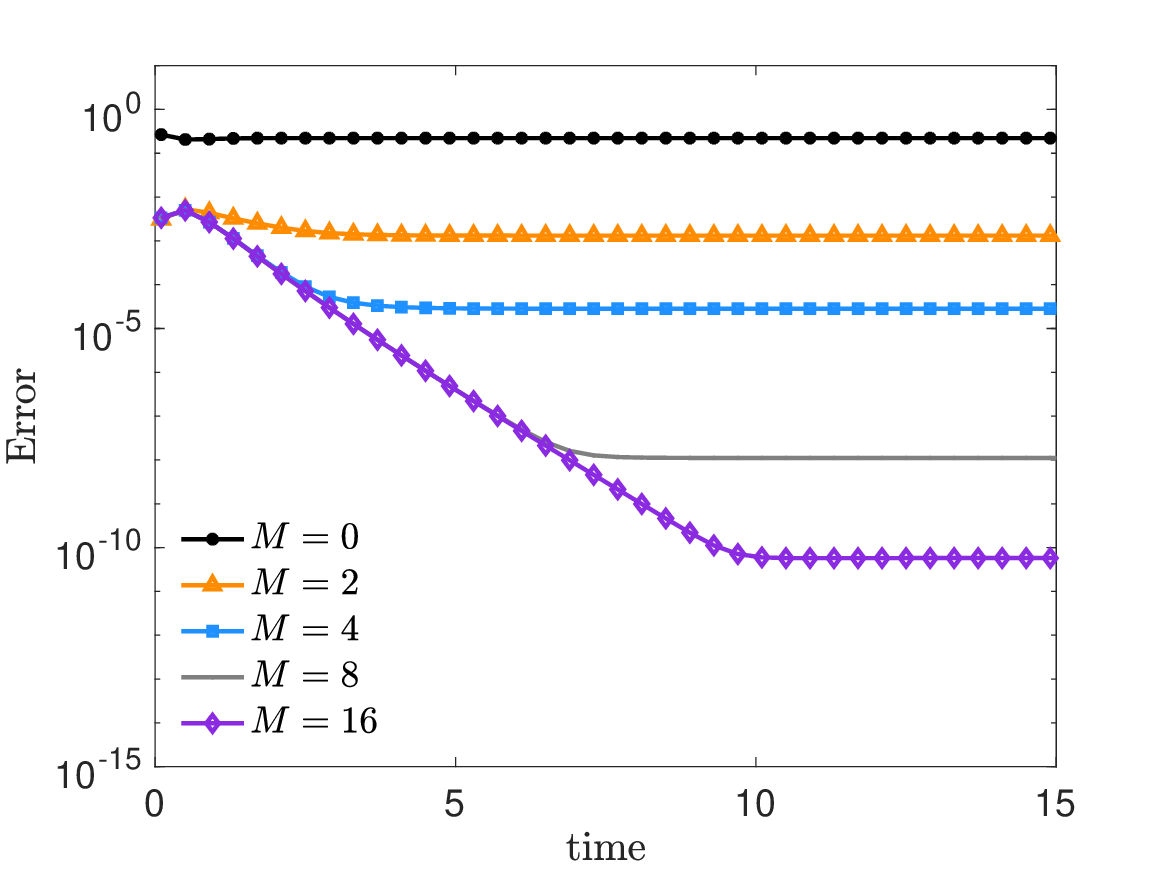}
\caption{ \textbf{Test 2}. $L^2$ error for the central difference sG scheme (left) and for the micro-macro sG scheme (right) with $N = 41$ gridpoints. The reference solution is computed using $M = 40$ and $N=321$ grid points.  }
\label{fig:op1}
\end{figure}
Since the problem is linear, for the time integration we adopted a second order DIRK implicit scheme over the time interval $[0,15]$ and time step $\Delta t = 10^{-1}$. We also consider  
$$\gamma(z)= \frac{3+z}{4}, \qquad z\sim \mathcal U([-1,1]),$$
and $\sigma^2 = 0.1$ in \eqref{eq:BD_opinion}. 

In Figure \ref{fig:op0}, we compare the large time distribution obtained for the standard sG scheme and the micro-macro sG scheme with the exact steady state solution given in \eqref{eq:exact_opinion}. We considered $M = 5$ modes for the stochastic Galerkin approximation and we compare expectation and variance of the distributions. {A Legendre polynomial basis in the random space and a second order implicit DIRK scheme for the time integration with $\Delta t = 10^{-1}$ have been used}. It is easily observed that MMsG correctly catches the large time behavior of the problem with high accuracy while the standard sG method is limited by the second order accuracy of central differences in the physical space. In Figure \ref{fig:op1}, we present the evolution of the $L^2$ error over the time interval $[0,15]$ for the standard sG and for the MMsG schemes. The reference solution of the problem has been computed with $M=40$ modes for the gPC decomposition and $N= 321$ gridpoints for the discretization of the interval $[-1,1]$. We can easily observe how with a finite number of modes we dramatically improve the accuracy of the scheme for all times using the MMsG scheme. Furthermore, for large times a very high accuracy is reached as expected since the perturbation $g(z,v,t)$ converges towards zero.

\begin{figure}[h!]
	\centering
	\includegraphics[scale = 0.32]{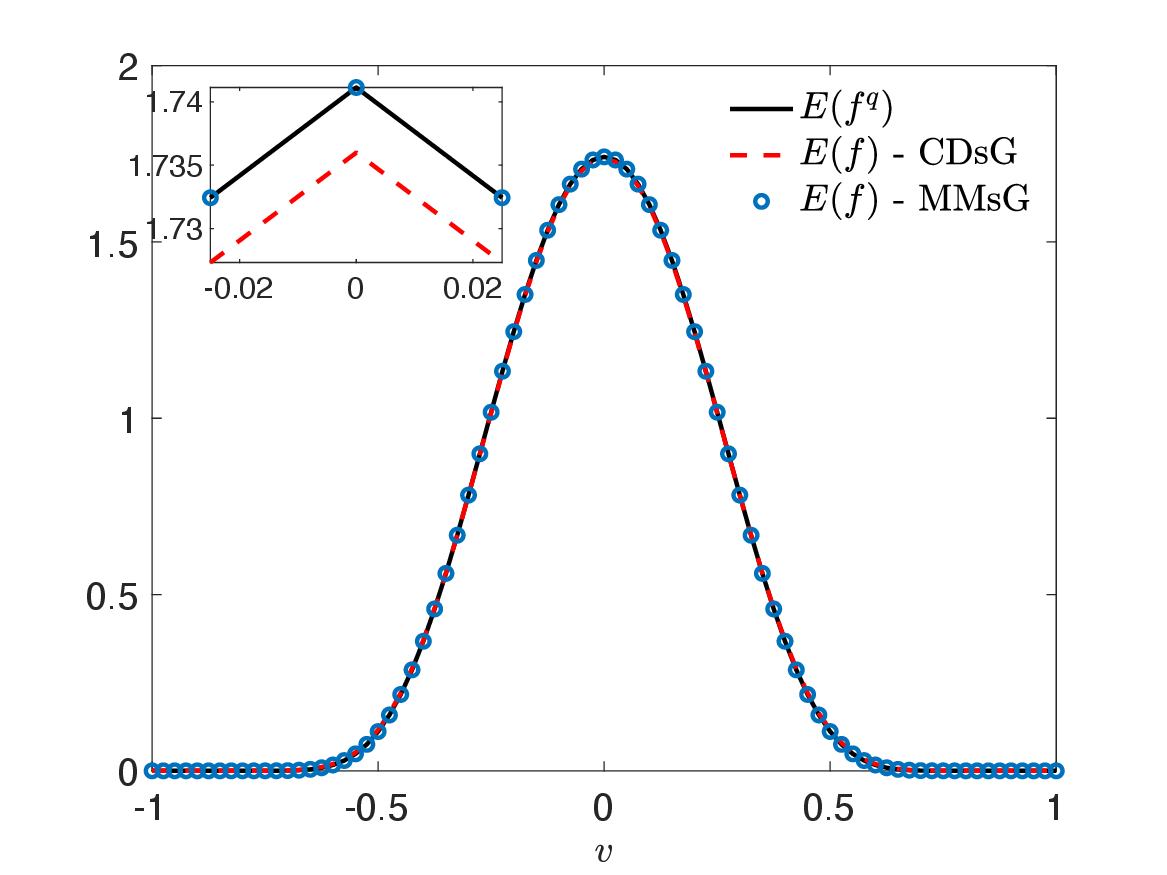}
	\includegraphics[scale = 0.32]{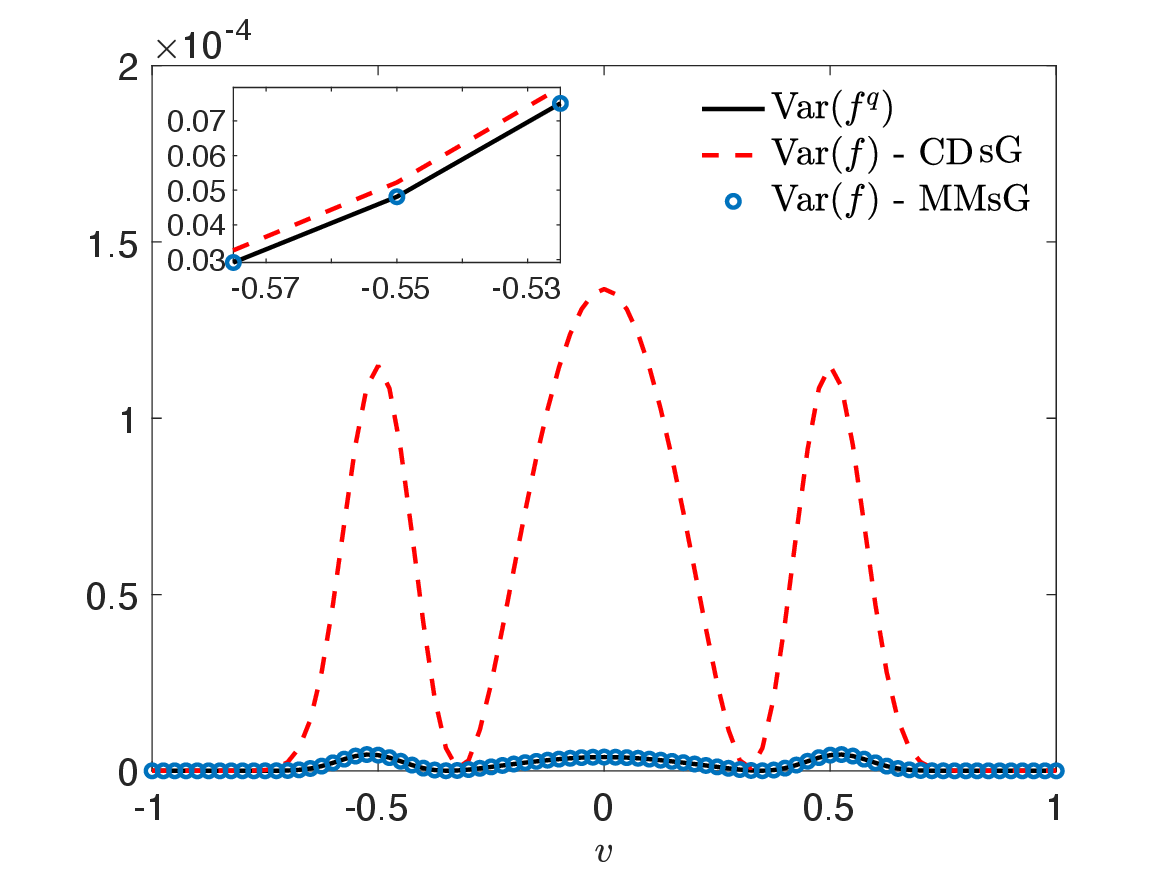} \\
	\includegraphics[scale=0.32]{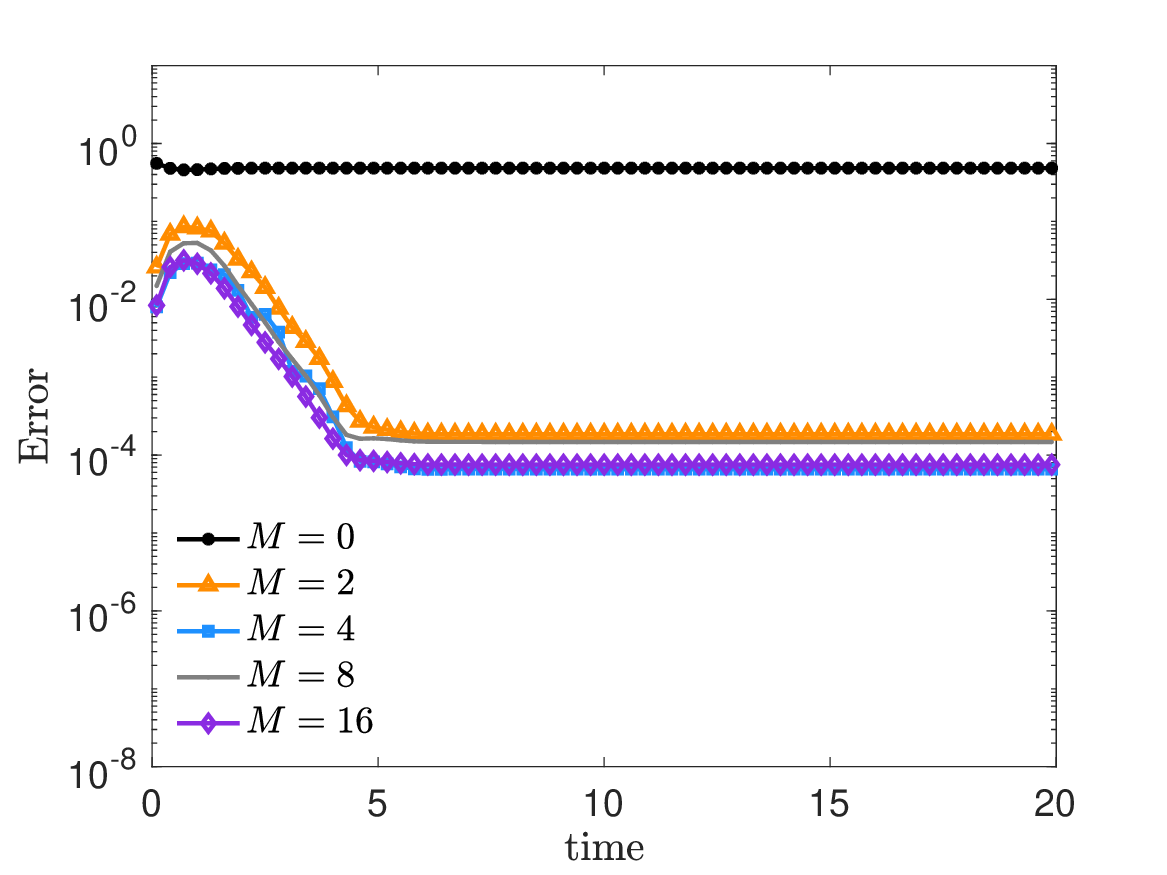}
	\includegraphics[scale=0.32]{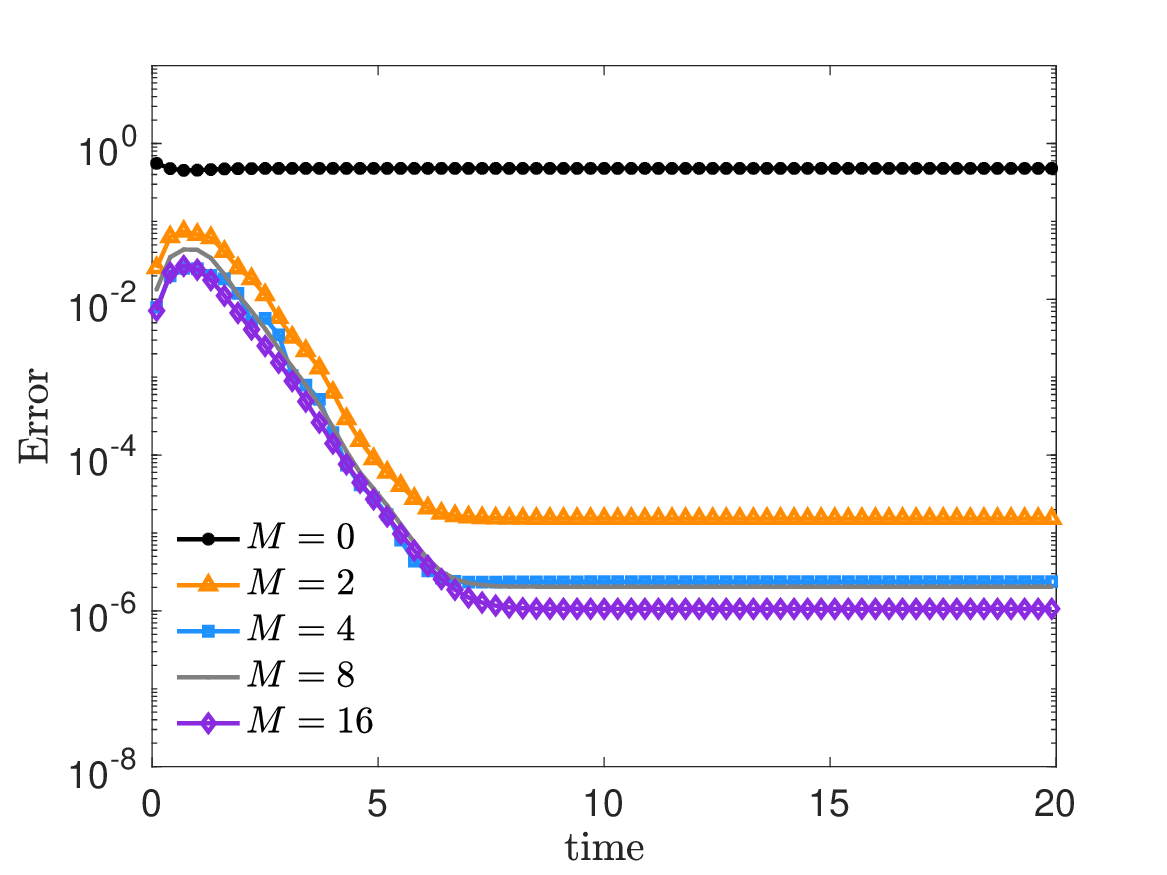}
	\caption{\textbf{Test 3}. Top row: large time distributions obtained from CDsG scheme and MMsG scheme in terms of expected distribution (left) and its variance (right) for the bounded confidence model with $\Delta(z)$ as in \eqref{eq:Delta}. Bottom row: $L^2$ error for CDsG (left) and MMsG (right). The reference solution is obtained with $M = 40$ modes and $N = 321$ gridpoints.}
	\label{fig:test2_1}
\end{figure}

\subsection{Test 3:  Quasi equilibrium states in bounded confidence dynamics}\label{test2}

In this part, we consider the evolution of the distribution $f(z,v,t)$, $v\in[-1,1]$ in the case of bounded confidence interactions which corresponds to consider the following nonlocal drift term
\be\label{eq:Bdef_BC}
\mathcal B[f](z,v,t) = \int_{-1}^1 P(z,v,v_*)(v-v_*)f(z,v_*,t)dv_*,\qquad P(z,v,v_*)=\chi(|v-v_*|\le \Delta(z))
\ee
where $\chi(\cdot)$ is the indicator function and $\Delta (z) \in[0,2]$ is an uncertain interaction threshold depending on $z\in I_{z}$. The diffusion function is of the form $D(v) = \sigma^2/2 (1-v^2)^2$. The same initial distribution $f_0(v)$ defined in \eqref{eq:f0_test1} is employed.

\begin{figure}[h!]
\centering
\subfigure[$\mathbb E(f)$]{
\includegraphics[scale=0.32]{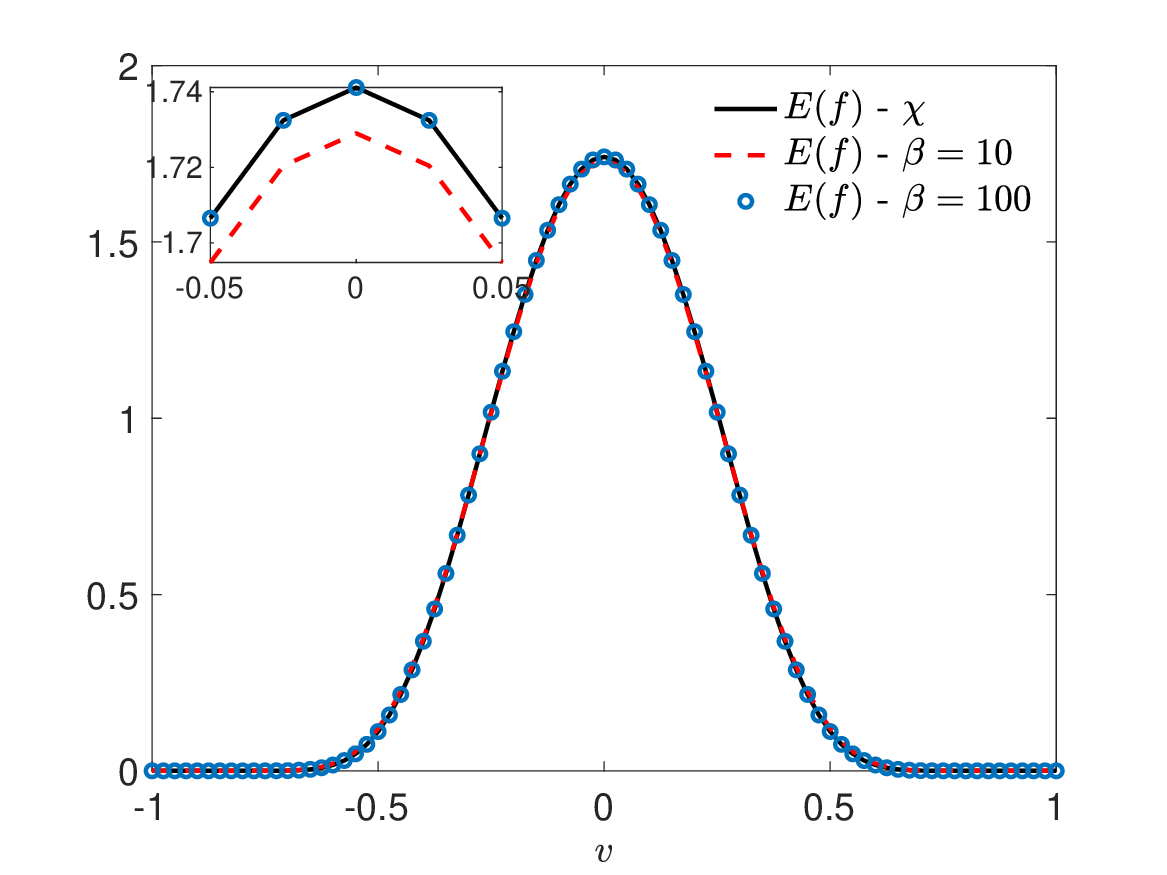}}
\subfigure[$\textrm{Var}(f)$]{
\includegraphics[scale=0.32]{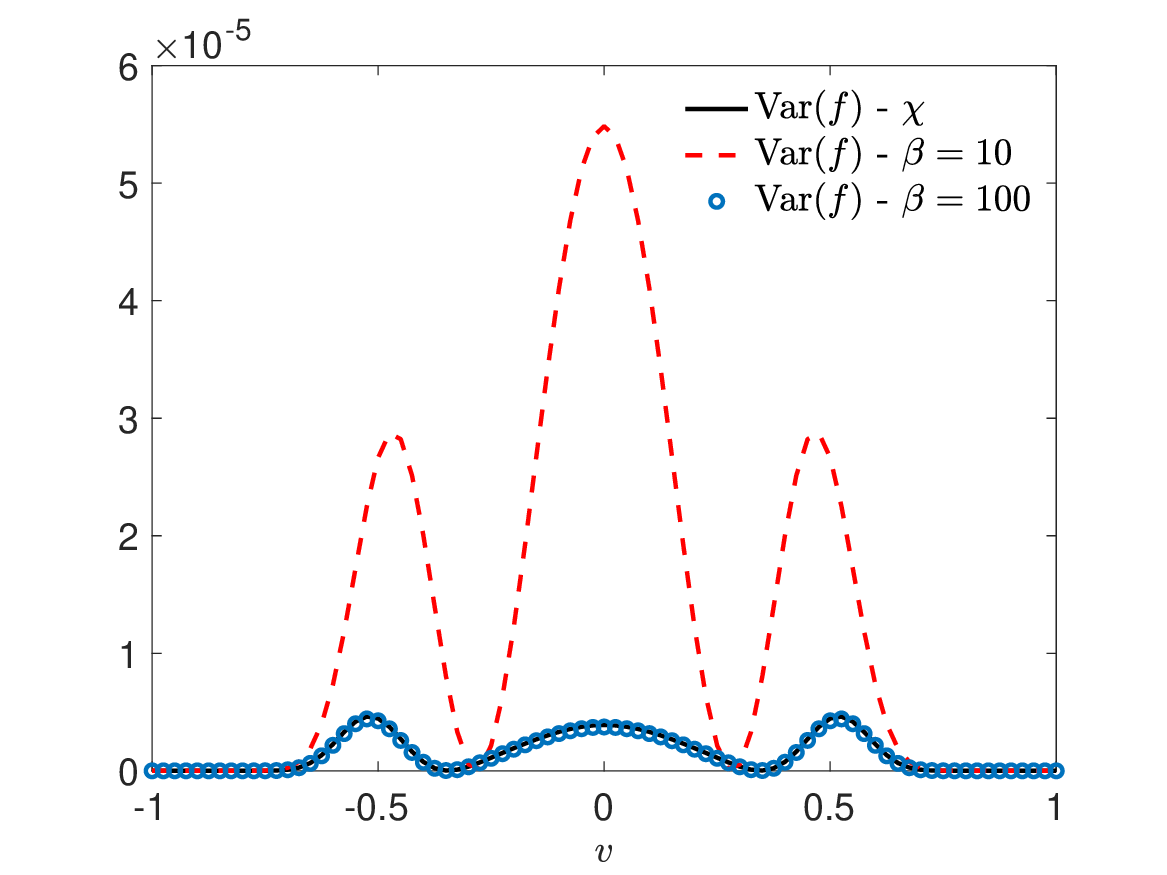}}\\
\subfigure[$\beta=10$]{
\includegraphics[scale=0.32]{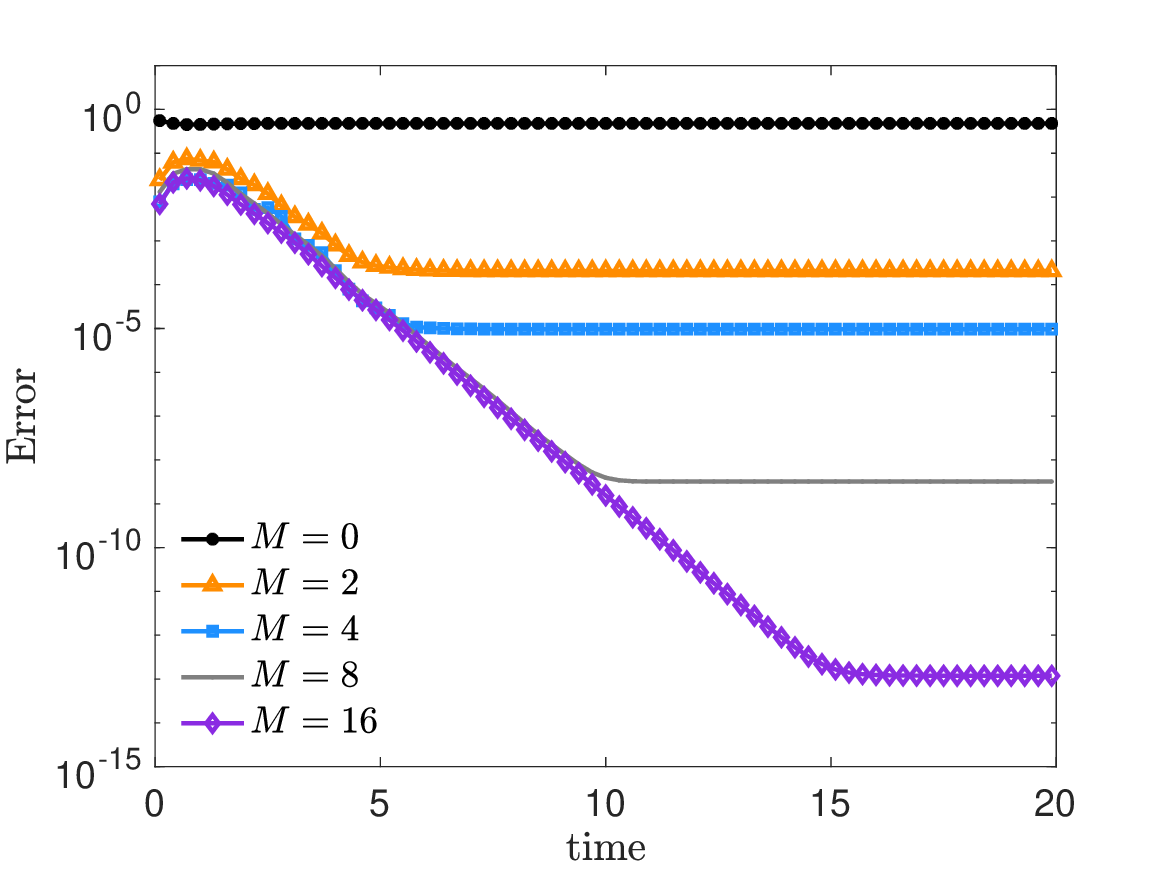}}
\subfigure[$\beta=100$]{
\includegraphics[scale=0.32]{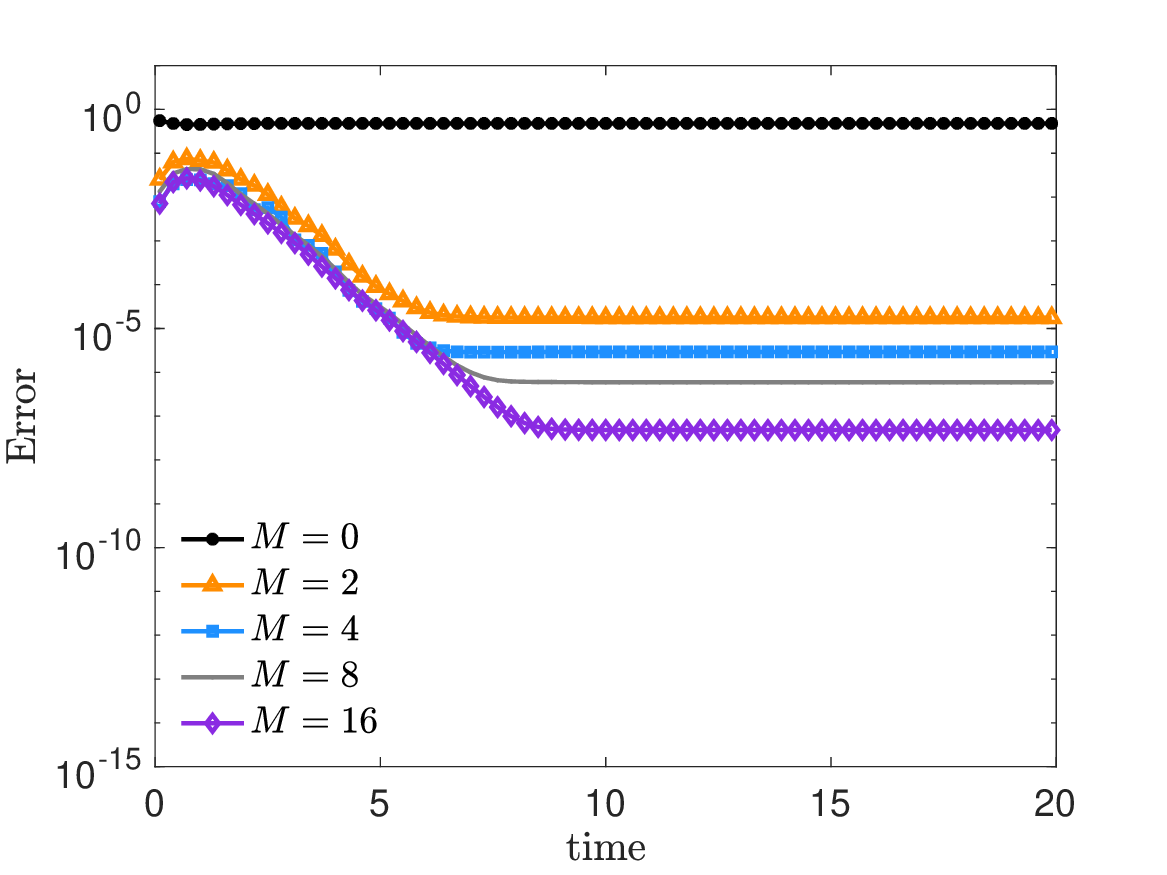}}
\caption{\textbf{Test 3}. Top row: expected distribution and variance obtained with MMsG scheme. The interaction function is defined through step function $\chi$ and with its approximation $K^\beta$, $\beta = 10,10^2$. Bottom row: $L^2$ error for the approximated bounded confidence model with $\beta = 10$ (left) and $\beta=10^2$ (right). }
\label{fig:BC_beta}
\end{figure}

We remark how in this case, due to the nonlinearity in the interactions, the equilibrium state is not known explicitly. Hence we will exploit here the quasi-equilibrium micro-macro formulation of the scheme as detailed in \eqref{eq:MM_ref}. The quasi-equilibrium state has been computed by trapezoidal quadrature whereas all derivatives in the physical space by central differences. {We considered a second order semi-implicit scheme for the time integration with $\Delta t = 10^{-1}$ and a Legendre polynomial basis in the random space.} 
 
The following random interaction threshold is taken
\be
\label{eq:Delta}
\Delta(z)  = 1 + \dfrac{z+1}{4}, \qquad z\sim \mathcal U([-1,1]).
\ee
In Figure \ref{fig:test2_1}, we present the expected distribution and its variance at time $T=20$ obtained through the CDsG scheme and MMsG scheme with $N=81$ and a second order semi-implicit time integration with $\Delta t = 10^{-1}$ (we refer to^^>\cite{PZ2018} for details about the semi-implicit scheme). In the bottom row, we present the evolution of the $L^2$ error computed with respect to a reference solution obtained with $M = 40$ modes and $321$ points in physical space. Thanks to the MMsG scheme we gain two orders of accuracy with respect to the CDsG scheme. Note, however, that spectral accuracy of the MMsG method is lost. This behavior is essentially due to the discontinuity of the function $\chi(\cdot)$ in \eqref{eq:Bdef_BC} which causes a loss of regularity in the solution and consequently the loss of accuracy in the numerical approximation. In order to avoid this problem we introduce a continuous approximation of the step function through sigmoid functions which regularize the solution. Specifically, we consider
\[
K^\beta(z,v-v_*) = \dfrac{1}{(1 + \exp\left( -\beta(v-v_* + \Delta(z))\right))} \cdot\dfrac{1}{(1 + \exp\left(-\beta(-v+v_* + \Delta(z))\right))}, \qquad \beta>0.
\]

For decreasing $\beta> 0$ now regularity is higher and, as expected, the resulting equilibrium is consistently approximated permitting to recover the spectral convergence of the scheme as $\beta>0$ decreases. The convergence is deteriorated for $\beta\gg0$, see Figure \ref{fig:BC_beta} for the details.

\subsection{Test 4. Swarming with self-propulsion}

\begin{figure}
\centering
\subfigure[$\alpha = 2, t = 0$]{
\includegraphics[scale=0.26]{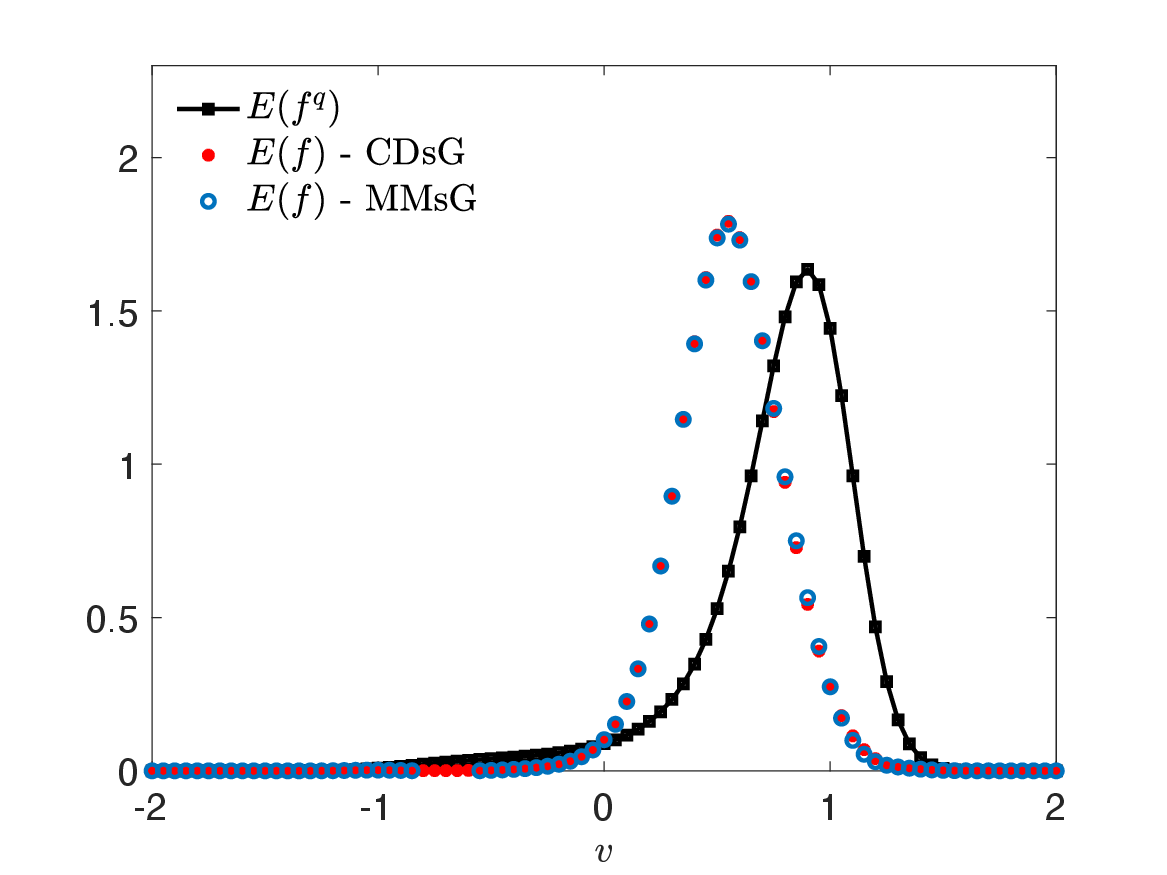}\hspace{-0.5cm}}
\subfigure[$\alpha = 2, t = 1$]{
\includegraphics[scale=0.26]{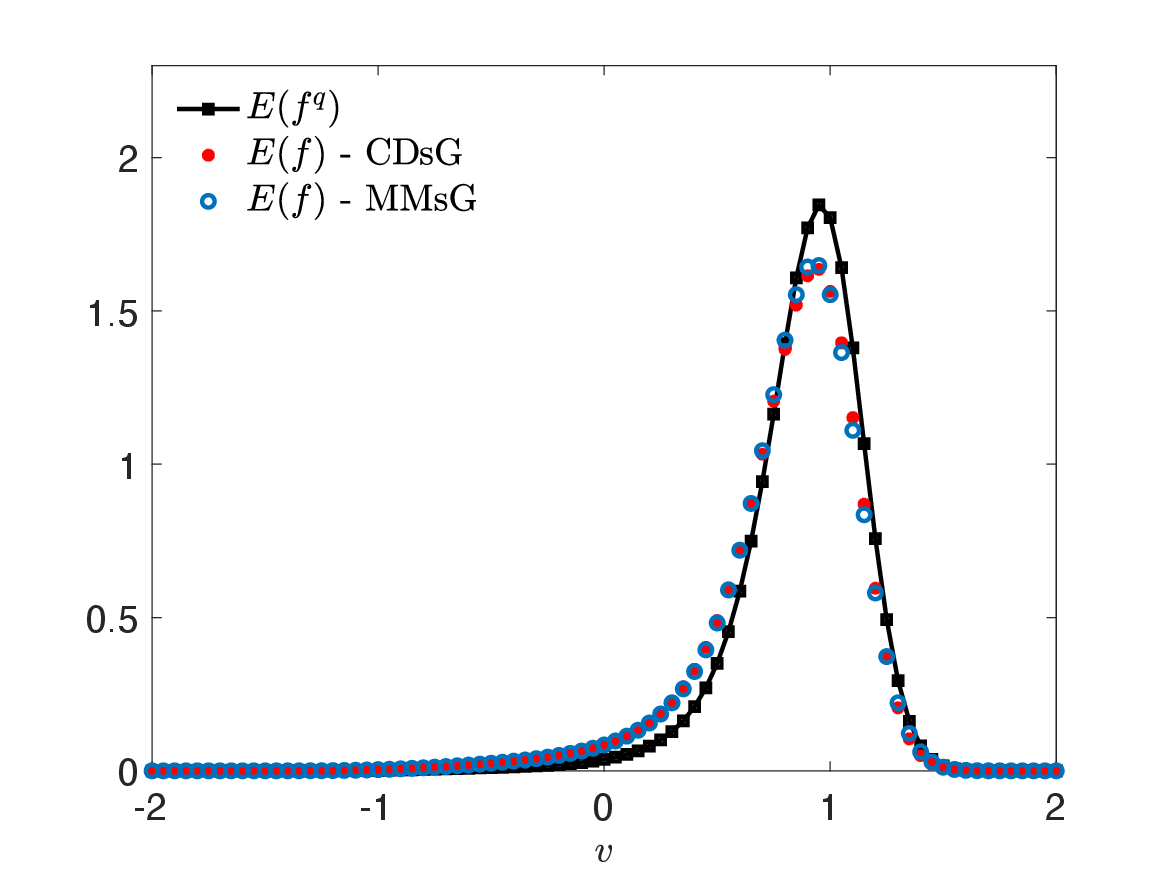}\hspace{-0.5cm}}
\subfigure[$\alpha = 2, t = 10$]{
\includegraphics[scale=0.26]{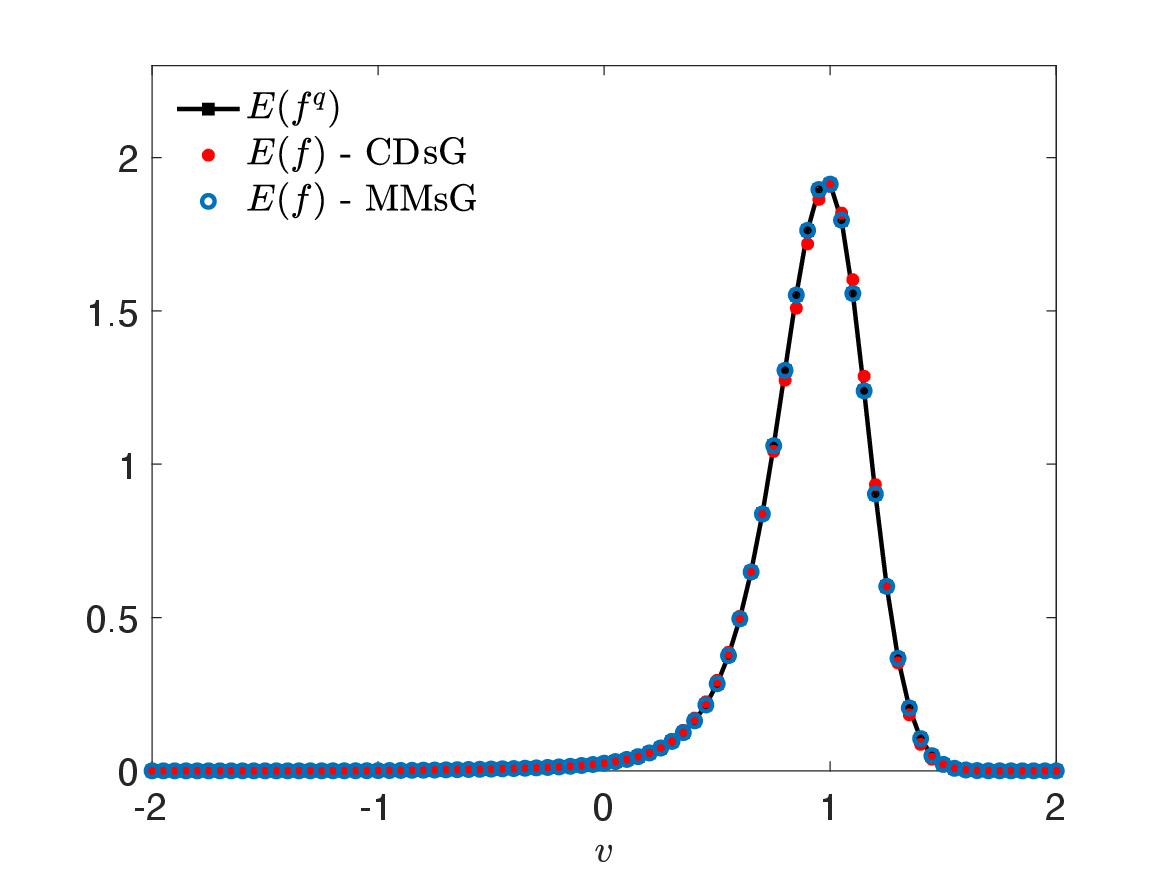}} \\
\subfigure[$\alpha = 4, t = 0$]{
\includegraphics[scale=0.26]{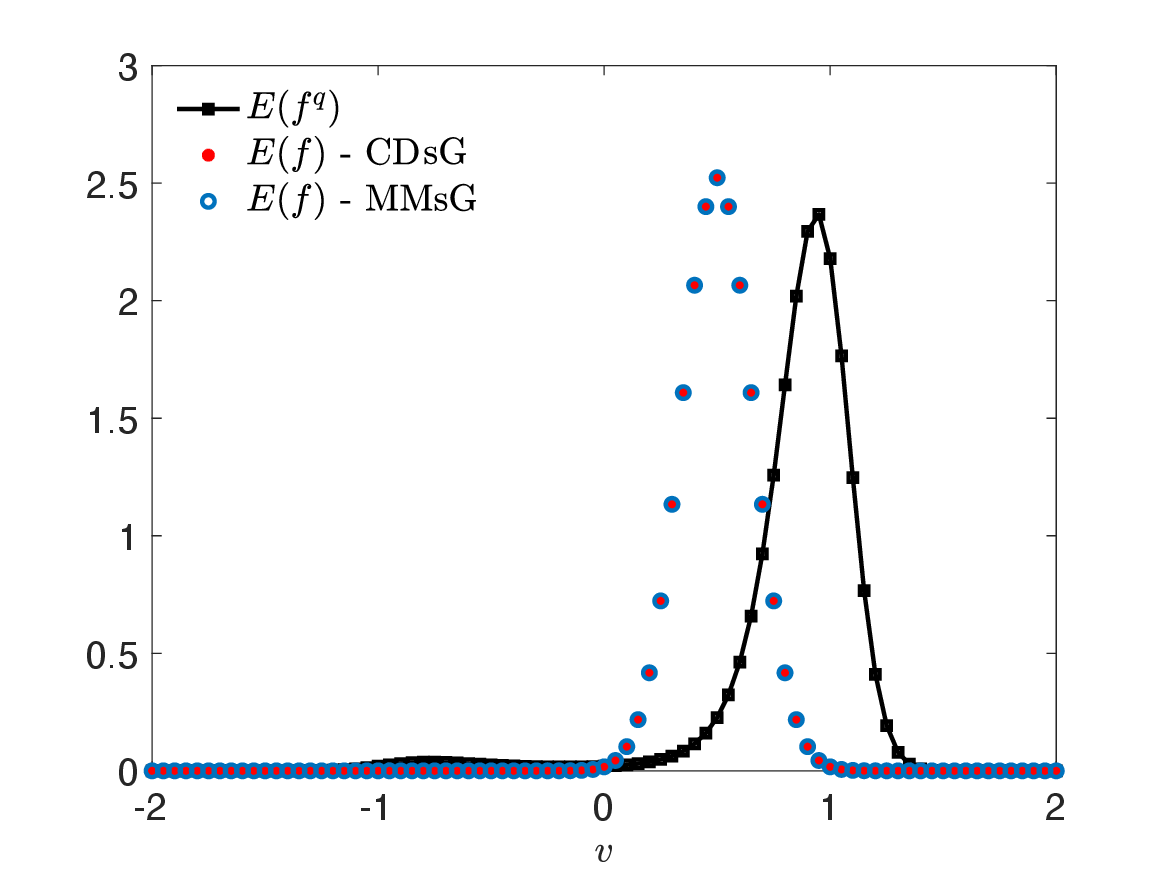}\hspace{-0.5cm}}
\subfigure[$\alpha = 4, t = 1$]{
\includegraphics[scale=0.26]{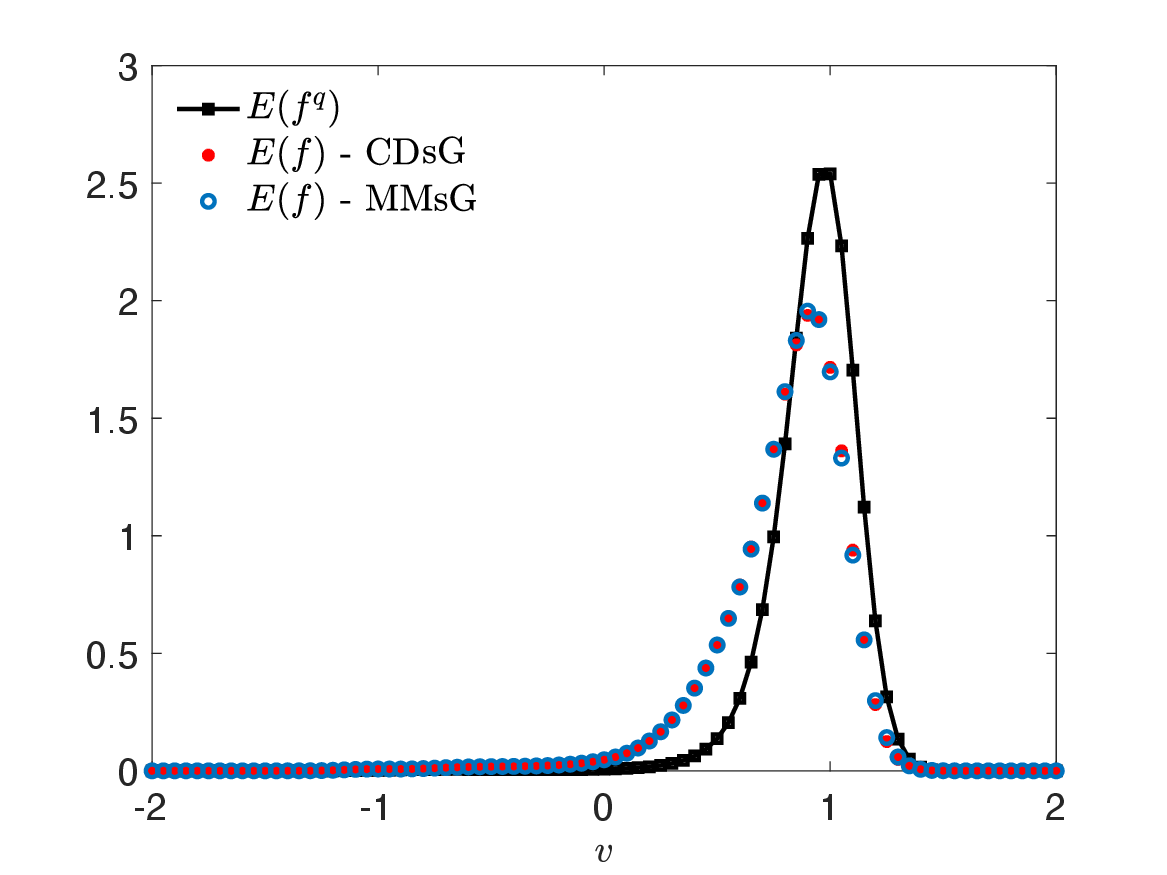}\hspace{-0.5cm}}
\subfigure[$\alpha = 4, t = 10$]{
\includegraphics[scale=0.26]{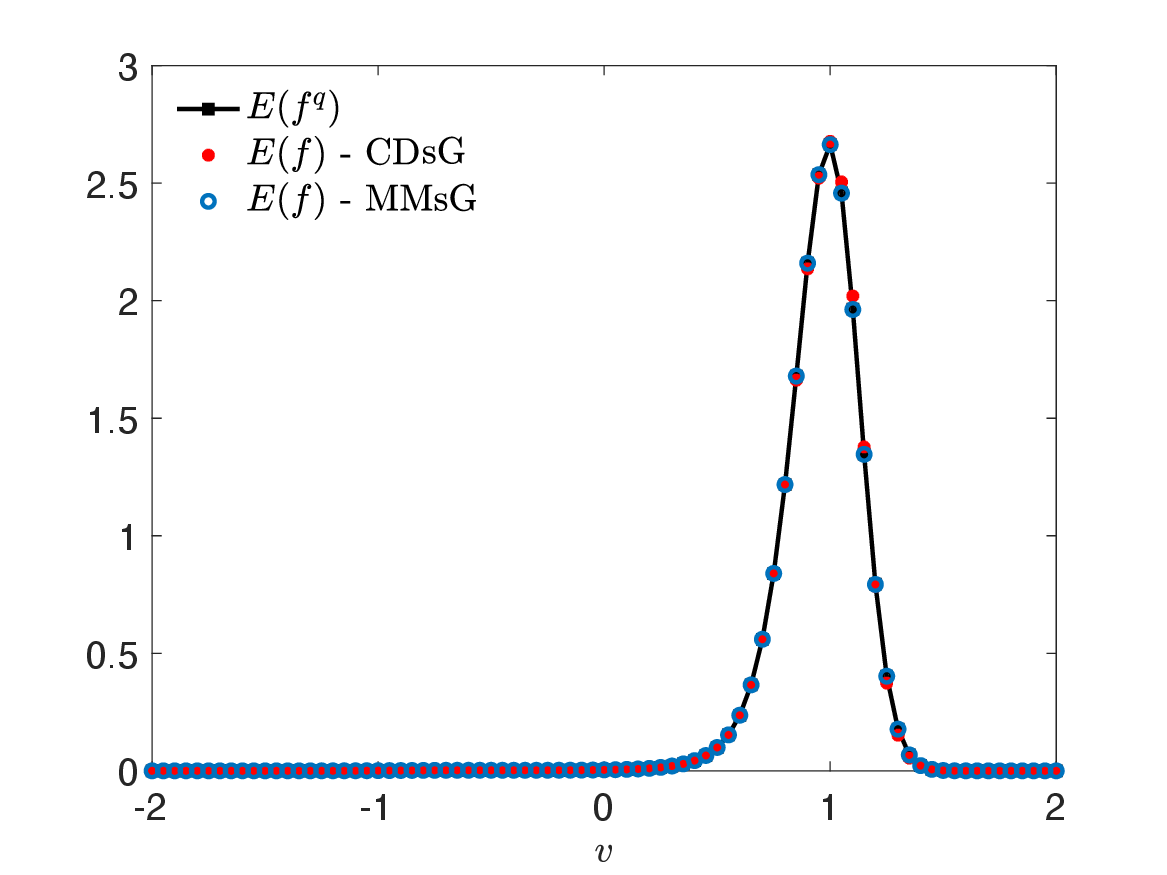}}
\caption{\textbf{Test 4}. Time evolution for the swarming model with self propulsion \eqref{eq:swarming_hom} for the CDsG and MMsG schemes and quasi-equilibrium state using $\alpha=2$ (top) and $\alpha = 4$ (bottom). }
\label{fig:evo_swarming}
\end{figure}

\begin{figure}
\centering
\includegraphics[scale=0.36]{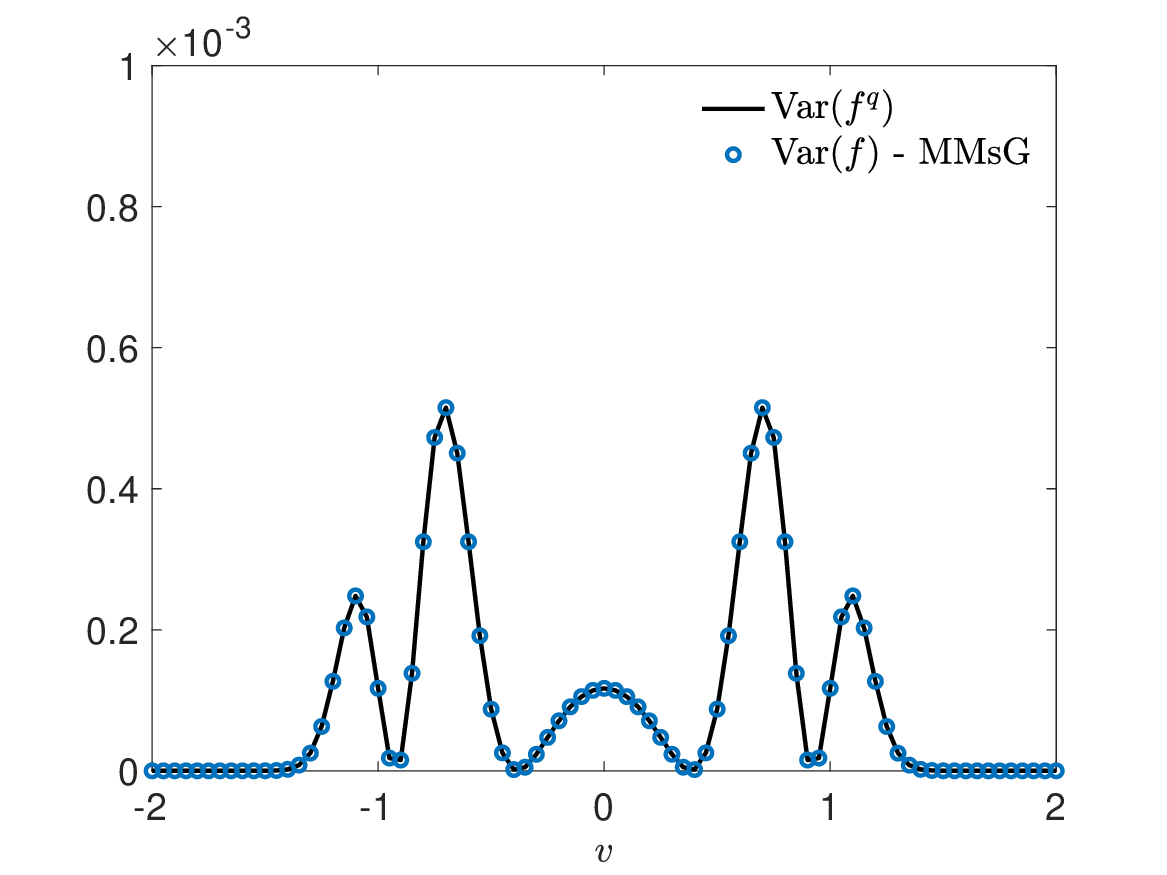}
\includegraphics[scale=0.36]{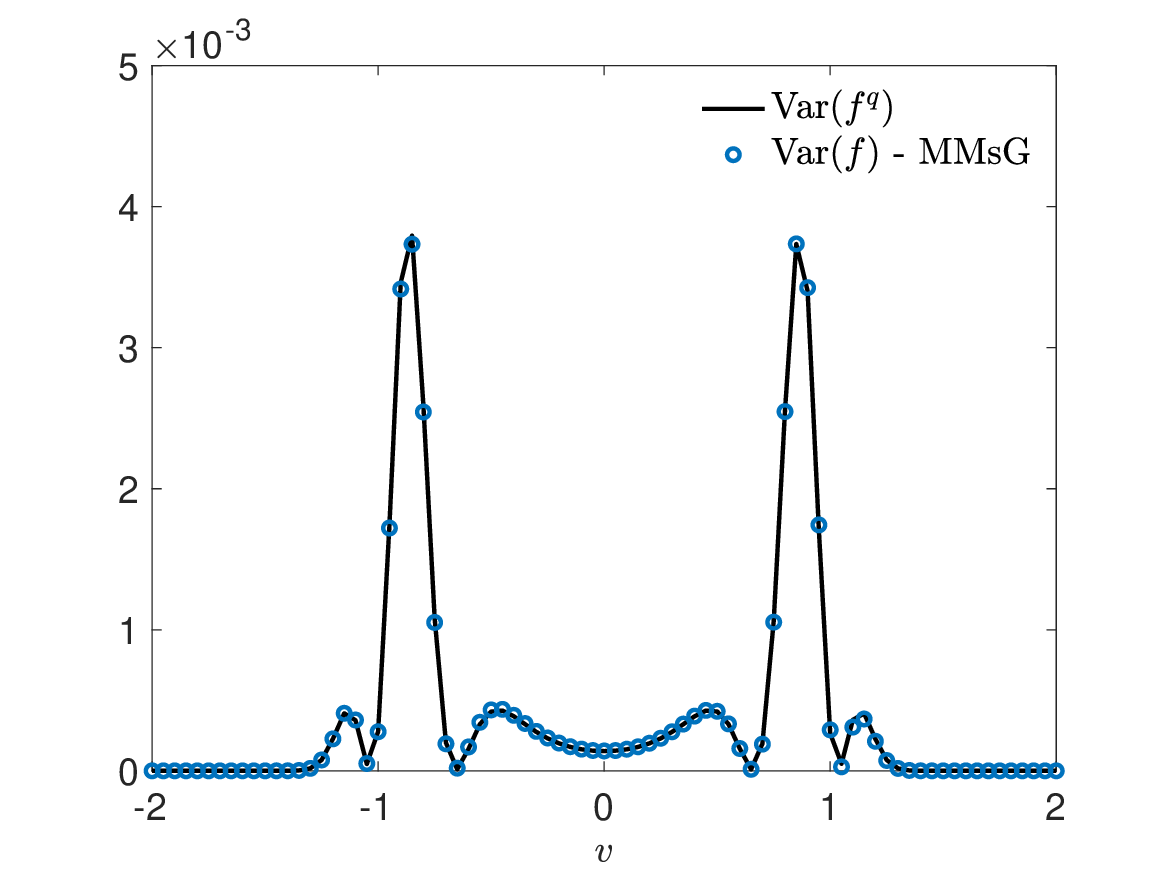}
\caption{\textbf{Test 4}. Variance for the expected distribution of the swarming model \eqref{eq:swarming_hom} at final time using MMsG scheme together with the correspondent variance of the quasi-equilibrium state using $\alpha=2$ (left) and $\alpha = 4$ (right).}
\label{fig:var_swarming}
\end{figure}

In this last test, we consider the swarming model with self-propulsion of Cucker-Smale type discussed in Section \ref{sec:swarming}. 
 
Now, in the classical gPC setting, one solves the following system of PDEs for $h=0,\dots,M$ 
\be\label{eq:swarming_system}
\begin{split}
\frac{\partial}{\partial t} \hat f_h(v,t) = \nabla_v \cdot \left[ \sum_{k=0}^M S_{hk} \hat{f}_k(z,v,t) + \sum_{k=0}^M B_{hk}\hat f_k(v,t) + \nabla_v\sum_{k=0}^M D_{hk} \hat f_k(v,t) \right],
\end{split}\ee
where the following matrices for $h,k = 0,\dots,M$ are defined
\[
\begin{split}
S_{hk}(v) &= \int_{I_z} \alpha(z)(|v|^2-1)v \Phi_h(z)\Phi_k(z)p(z)dz, \\
B_{hk}[f_M] &= \int_{I_{z}}\left(v-\int_{\RR^{d_v}}v \sum_{r=0}^M \hat f_r(v,t)\Phi_r(z)dv\right)\, \Phi_h(z)\Phi_k(z)p(z)dz, \\
D_{hk} &= \int_{I_{z}} D(z)\Phi_h(z)\Phi_k(z)p(z)dz. 
\end{split}
\]
while in the micro-macro setting one solves
\be\label{eq:swarming_system_mm}
\begin{split}
	\frac{\partial}{\partial t} \hat f_h(v,t) &= \nabla_v \cdot \left[ \sum_{k=0}^M S_{hk} \hat{f}_k(z,v,t) + \sum_{k=0}^M B_{hk}\hat f_k(v,t) + \nabla_v\sum_{k=0}^M D_{hk} \hat f_k(v,t) \right]\\
	&-\nabla_v \cdot \left[ \sum_{k=0}^M S_{hk} \hat{f^q_k}(z,v,t) + \sum_{k=0}^M B_{hk}\hat f^q_k(v,t) + \nabla_v\sum_{k=0}^M D_{hk} \hat f^q_k(v,t) \right].
\end{split}\ee
Again central differences have been used in the physical space, whereas a trapezoidal rule is applied to evaluate the quasi-equilibrium state 
\be
f^{q}(v,z,t) = C\exp \left\{-\dfrac{1}{D(z)}\left(\alpha(z)\dfrac{|v|^4}{4}+(1-\alpha(z))\dfrac{|v|^2}{2}-u_{f}(z,t) v \right) \right\}.
\ee
In this test, we considered a discretization of the domain $[-2,2]$ of $N = 81$ gridpoints and a {second order semi-implicit time integration over the time interval $[0,10]$, $\Delta t = 10^{-1}$} with the following initial distribution
\[
f_0(v) = C \exp\left\{ - \dfrac{(v-u_0)^2}{2\sigma_0^2} \right\}, \qquad u_0 = 1/2,\, \sigma_0^2 = 1/40,
\]
for which in Figure \ref{fig:evo_swarming} we present the time evolution in the case of uncertain diffusion  
\[
D(z) = \frac{1}{5}  + \dfrac{z}{10}, \qquad z\sim\mathcal U([-1,1]), 
\]
and for two different choices of self-propulsion parameter $\alpha = 2,4$ which is assumed to be a constant. {As before, a Legendre polynomial basis has been used in the random space.} We observe from this figure how the large time behavior is correctly described by both the central difference scheme and the micro macro stochastic Galerkin method. In particular, in order to highlight the good accuracy of the MMsG scheme we represent in Figure \ref{fig:var_swarming} the variance of the distribution for large times, from which we observe how we correctly catch the higher order statistics. 
\begin{figure}
\centering
\includegraphics[scale=0.35]{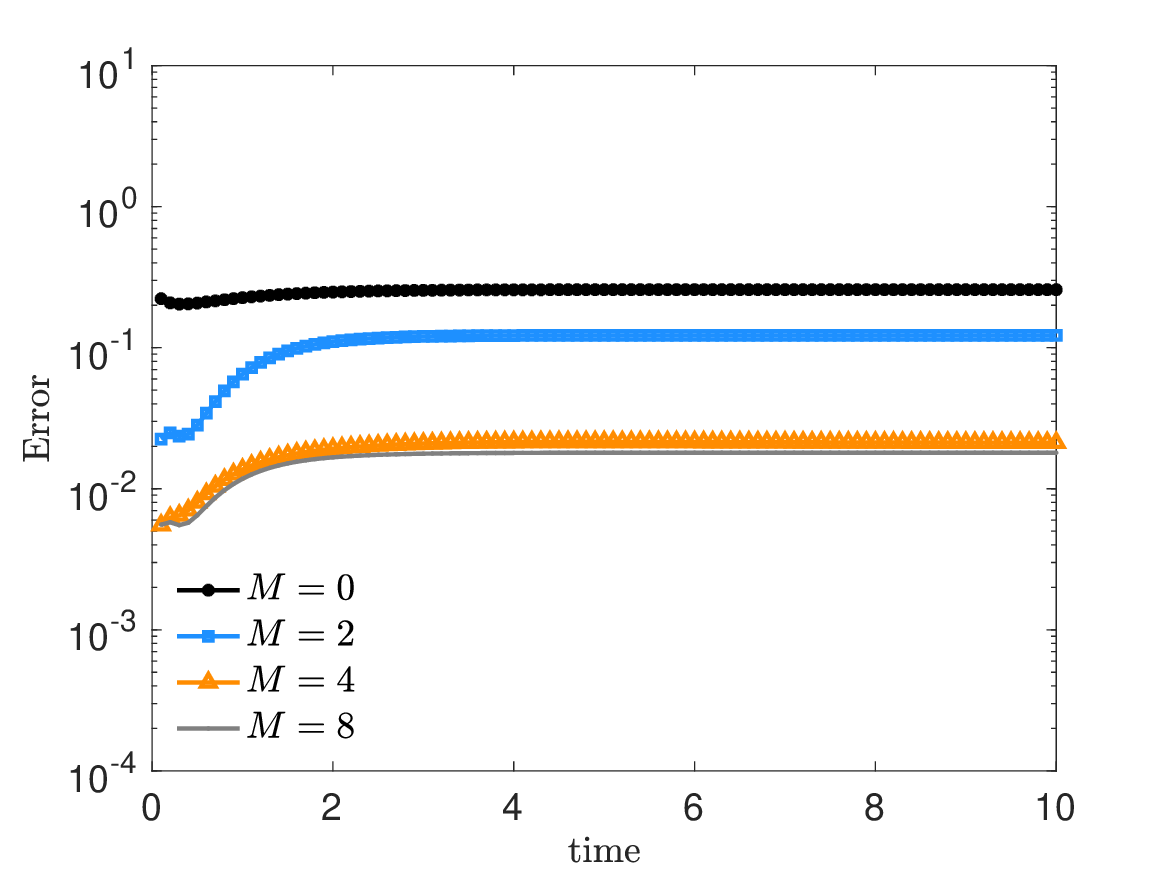}
\includegraphics[scale=0.35]{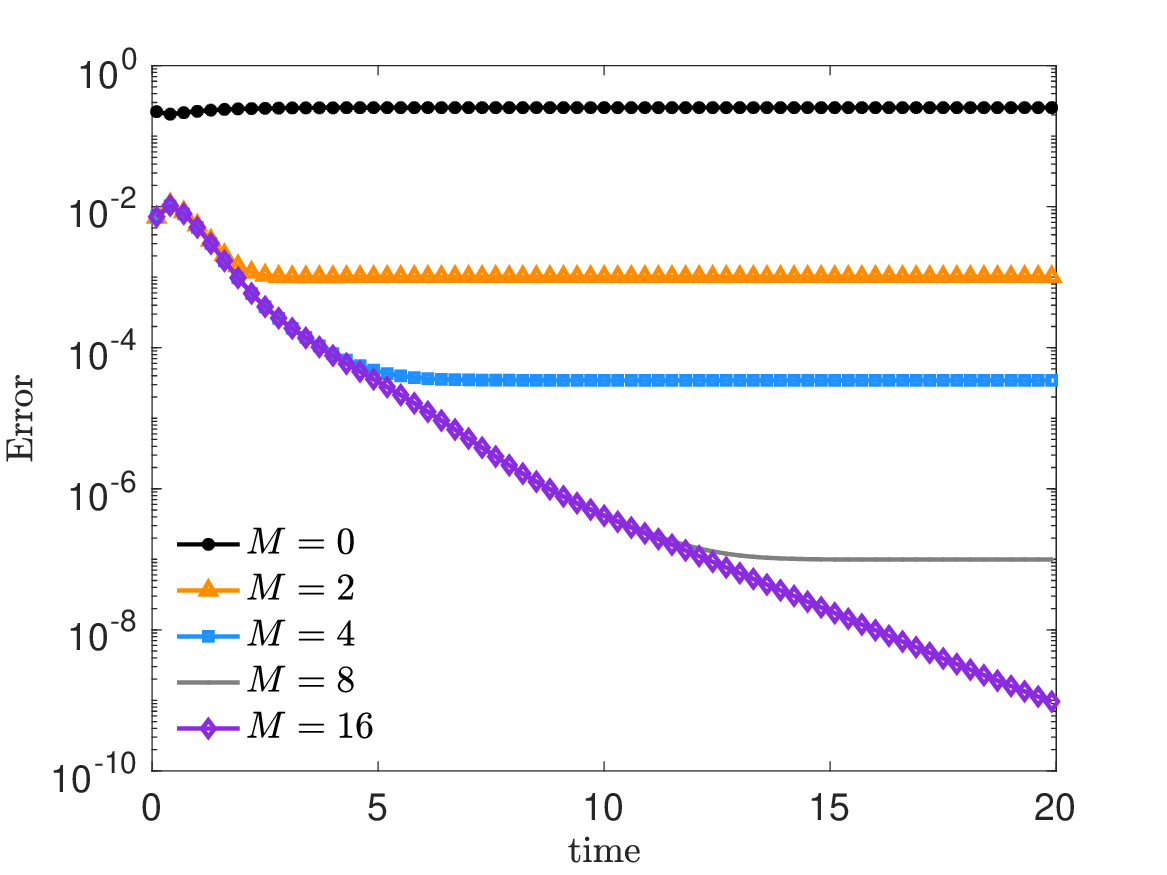}
\caption{\textbf{Test 4}. $L^2$ error for the swarming model with stochastic diffusion obtained with  central difference sG schemes (left) and with the MMsG scheme (right). The reference solution is obtained with $M=40$ and $N = 321$ mesh points.}
\label{fig:error_swarming}
\end{figure}

Finally, in Figure \ref{fig:error_swarming} we present the evolution of the $L^2$ error of the  CDsG scheme and the MMsG scheme. The reference solution is obtained with a discretization of $N = 321$ gridpoints for the interval $[-2,2]$ and $M = 40$ Galerkin projections. We observe how we improve dramatically the performance of a standard consistent scheme through the micro-macro method. Indeed with few modes we essentially reach machine precision while for the classical CS method the error as for the previous test cases saturates due to the second order accurate discretizations.

 \section*{Conclusions}\label{sec:conc}
In this work we considered the construction of stochastic Galerkin methods for linear and nonlinear Fokker-Plank equations with uncertain parameters. A key ingredient for observing the expected accuracy in the stochastic space is the regularity of the underlying model solutions. For this reason, in a first part we recalled some regularity results for several models of interest in various applications ranging from classical contexts to swarming through opinion formation and wealth distributions. In particular, these models in presence of uncertainties exhibit long time behaviors characterized by uncertain equilibrium solutions. A fundamental property of the corresponding numerical method is that it captures such steady states with high accuracy. The intrusive nature of the stochastic Galerkin approach, however, does not allow to preserve the structure preserving properties of the corresponding spatial discretizations except for the case of simple linear problems. We therefore focused on constructing a new class of numerical schemes that, thanks to an appropriate reformulation of the original problem based on a generalization of the micro-macro decomposition using the notion of quasi-equilibrium state, enjoys the property of preserving the large-time behavior of the original equations. This made it possible to obtain a new type of schemes for which the global error in the physical space decreases as the solution approaches the equilibrium state whereas spectral accuracy is achieved in the random space. The latter property, in particular, requires not only regularity in parameter space but also dependence of the equilibrium state on random parameters. Several numerical experiments have shown that indeed the expected results and improved performance of this new approach are obtained compared with standard stochastic Galerkin methods. {Future research directions will consider the design of sG methods which are exactly conservative (following the approach in \cite{Pareschi2023}), an in-depth comparison of the AP properties of sG and sC methods, and a more detailed analysis of the degenerate case where for long times the solution becomes independent of the uncertainties. }

\section*{Acknowledgments}
L.P. was partially supported by MIUR-PRIN Project 2017, No. 2017KKJP4X \emph{Innovative numerical methods for evolutionary partial differential equations and applications}.  G.D. and M.Z. were partially supported by MUR-PRIN Project 2020, No. 2020JLWP23 \emph{Integrated mathematical approaches to socio-epidemiological dynamics}. The research of M.Z. was partially supported by MUR, Dipartimenti di Eccellenza Program (2018–2022), and Department of Mathematics “F. Casorati”, University of Pavia.

\bibliographystyle{plain}

\appendix
\section{An exact solution of the Fokker-Planck equation}
\label{appA}
In the following we summarize the derivation of an exact solution for the Fokker-Planck equation 
\be
\partial_t f(\z,v,t) = K(\z)\partial_v \left[ vf(\z,v,t) + \sigma(\z) \partial_v f(\z,v,t)\right], 
\label{eq:FPu}
\ee
in the presence of uncertain quantities. The methodology is similar to the one presented in^^>\cite{Bob,Ernst} in the context of Maxwellian models of kinetic theory. 

We consider a deterministic initial distribution 
\[
f_0(\z,v) = \alpha(\z) v^2 \exp\{-\beta(\z) v^2\},\qquad v\in \mathbb R, 
\]
such that
\begin{enumerate}
\item[$i)$] $\displaystyle\int_{\mathbb R} f_0(v) dv = \dfrac{\sqrt{\pi}}{2} \dfrac{\alpha}{\beta\sqrt{\beta}}$, \\
\item[$ii)$]$\displaystyle\int_{\mathbb R} vf_0(v) dv = 0$, \\
\item[$iii)$]$\displaystyle\int_{\mathbb R} v^2 f_0(v)dv = \dfrac{3}{4}\sqrt{\pi}\dfrac{\alpha}{\beta^2\sqrt{\beta}}. $
\end{enumerate}
These quantities are conserved in time. Hence, we can determine $\alpha(\z),\beta(\z)>0$ to guarantee unitary mass and temperature $\sigma(\z)>0$
\[
\begin{cases}
\dfrac{\sqrt{\pi}}{2} \dfrac{\alpha}{\beta\sqrt{\beta}}  &= 1, \\
\dfrac{3}{4}\sqrt{\pi}\dfrac{\alpha}{\beta^2\sqrt{\beta}} &= \sigma(\z).
\end{cases}
\]
This condition translates in the following expression 
\[
\alpha(\z) = \dfrac{3}{2\sigma(\z)}\sqrt{\dfrac{3}{2\sigma(\z)}},\qquad \beta(\z) = \dfrac{3}{2\sigma(\z)}.
\]
In order to determine an explicit solution of the Fokker-Planck equation \eqref{eq:FPu} with initial datum $f_0(\z,v)$ we consider the class of solutions given by 
\[
f(\z,v,t) = (A(\z,t)+B(\z,t)v^2)\exp\{-s(\z,t)v^2\},
\]
where, by imposing conservation of mass and energy, $A(\z,t,)$, $B(\z,t)$, and $s(\z,t)$ satisfy the following system 
\begin{align*}
\dfrac{A\sqrt{\pi}}{\sqrt{s}} + \dfrac{B\sqrt{\pi}}{2s\sqrt{s}} =& \dfrac{\sqrt{\pi}}{2} \dfrac{\alpha}{\beta\sqrt{\beta}},  \\
\dfrac{A\sqrt{\pi}}{2s\sqrt{s}} + \dfrac{3B\sqrt{\pi}}{4s^2\sqrt{s}} =& \dfrac{3}{4}\sqrt{\pi}\dfrac{\alpha}{\beta^2\sqrt{\beta}},
\end{align*}
that yields  
\begin{equation}
\label{eq:AB}
\begin{cases}
A =& \dfrac{\alpha}{2\beta\sqrt{\beta}}\sqrt{s} - \dfrac{s\sqrt{s}}{2}\dfrac{\alpha}{2\beta\sqrt{\beta}} \left(\dfrac{3}{\beta}-\dfrac{1}{s} \right),  \\
B = &s^2\sqrt{s} \dfrac{\alpha}{2\beta\sqrt{\beta}} \left( \dfrac{3}{\beta}-\dfrac{1}{s}\right),
\end{cases}
\end{equation}
where for brevity we omitted the dependency of $A,B$, and $s$ on $(\z,t)$.\\
By imposing that $f(\z,v,t)$ is solution of the Fokker-Planck equation \eqref{eq:FPu} 
we get
\[
\partial_t f(\z,v,t) = \left[\dot{A} + (-A\dot{s} + \dot{B})v^2 - B\dot{s}v^4\right]\exp\{-sv^2\}
\]
and 
\[
\begin{split}
 &\partial_v \left[ vf(\z,v,t) + \sigma(\z) \partial_v f(\z,v,t)\right]  \\
 &\qquad = v^0 \exp\{-sv^2\} \left[ A - 2A\sigma(\z) s+2\sigma(\z) B \right]  \\
 &\qquad \quad +v^2 \exp\{-sv^2\} \left[ -2sA + 3B+ \sigma(\z) A 4 s^2 - 4s\sigma(\z) B\right] \\
 &\qquad\quad+ v^4 \exp\{-sv^2\} \left[ -2sB+ 4s^2B\sigma\right],
\end{split}\]
where again the dependency of $A,B$, and $s$ on $\z$ and time has been omitted. Equating now the terms in $v^{2n}$, $n = 0,1,2$ we reduce to solving 
\[
\begin{cases}
\dot{A} = K(\z)( A - 2A\sigma(\z) s+2\sigma(\z) B) \\
-A\dot{s} + \dot{B} =  K(\z)(-2sA + 3B+ \sigma(\z) A 4 s^2 - 4s\sigma(\z) B) \\
\dot{s} = K(\z)(2s - 4\sigma s^2), 
\end{cases}
\]
that, under relations \eqref{eq:AB}, are all equivalent to the following differential equation depending on the uncertain parameters 
\[
\dot{s} = K(\z)(2s - 4\sigma s^2). 
\]
Finally, we can complement the above differential equation with the initial datum $s(\z,0) = \beta(\z)$ to determine the following explicit expression
\begin{equation}
\label{eq:sevo}
s(\z,t) = \dfrac{e^{2K(\z)t}}{2\sigma(\z) e^{2K(\z)t}-{4\sigma(\z)}/{3}}.
\end{equation}

\end{document}